\documentclass[12pt,reqno]{amsart} 
\usepackage{amssymb}
\usepackage{enumerate}
\usepackage[all]{xy}
\usepackage{mathrsfs}

\newcommand{\klfour}{C_2\times C_2}

\newlength{\shorter}
\let\Gamma=\varGamma

\newcommand{\boldd}[1]{{\mathversion{bold}\textbf{#1}}}

\newcommand{\lie}[3]{\def\test{#2}\def\tst{G}\ifx\test\tst{{}^{#1}#2_{#3}}
\else{{}^{#1}\!#2_{#3}}\fi}

\newcommand{\4}[1]{\widebar{#1}}
\newcommand{\5}[1]{\widehat{#1}}
\newcommand{\9}[1]{{}^{#1}\!}   %% left conjugation
\newcommand{\EE}{\textbf{E}}

\let\oldcirc=\circ
\renewcommand{\circ}{\mathchoice
    {\mathbin{\scriptstyle\oldcirc}}{\mathbin{\scriptstyle\oldcirc}}
    {\mathbin{\scriptscriptstyle\oldcirc}}
    {\mathbin{\scriptscriptstyle\oldcirc}}}

\SelectTips{cm}{10} \UseTips   %% use computer modern tips in xy

\numberwithin{equation}{section}

\mathcode`\:="003A
\mathchardef\cdot="0201

\def\beq#1\eeq{\begin{equation*}#1\end{equation*}}
\def\beqq#1\eeqq{\begin{equation}#1\end{equation}}

\let\emptyset=\varnothing
\renewcommand{\:}{\colon}   %% as in f:X-->Y

\newcommand{\widebar}[1]{\overset{\mskip2mu\hrulefill\mskip2mu}{#1}
		\vphantom{#1}}

\newcommand{\longline}{\bigskip\centerline{\hbox to 5cm{\hrulefill}}\bigskip}

\newcommand{\e}[2]{e_{#1#2}}

\newcommand{\mxtwo}[4]{\left(\begin{smallmatrix}#1&#2\\#3&#4
\end{smallmatrix}\right)}
\newcommand{\mxthree}[9]{\left(\begin{smallmatrix}#1&#2&#3\\#4&#5&#6\\
#7&#8&#9\end{smallmatrix}\right)}

\newcommand{\mxfourb}[8]{#1&#2&#3&#4\\#5&#6&#7&#8\end{smallmatrix}\right)}
\newcommand{\mxfoura}[8]{\left(\begin{smallmatrix}#1&#2&#3&#4\\#5&#6&#7&#8\\}
\newcommand{\Mxfourb}[8]{#1&#2&#3&#4\\#5&#6&#7&#8\end{pmatrix}}
\newcommand{\Mxfoura}[8]{\begin{pmatrix}#1&#2&#3&#4\\#5&#6&#7&#8\\}
\def\trp[#1,#2,#3]{[\hskip-1.5pt[#1,#2,#3]\hskip-1.5pt]}

\DeclareMathAlphabet\EuR{U}{eur}{m}{n}
\SetMathAlphabet\EuR{bold}{U}{eur}{b}{n}

\newcommand{\higherlim}[2]{\displaystyle\setbox1=\hbox{\rm lim}
	\setbox2=\hbox to \wd1{\leftarrowfill} \ht2=0pt \dp2=-1pt
	\setbox3=\hbox{$\scriptstyle{#1}$}
	\def\test{#1}\ifx\test\empty
	\mathop{\mathop{\vtop{\baselineskip=5pt\box1\box2}}}\nolimits^{#2}
	\else
	\ifdim\wd1<\wd3
	\mathop{\hphantom{^{#2}}\vtop{\baselineskip=5pt\box1\box2}^{#2}}_{#1}
	\else
	\mathop{\mathop{\vtop{\baselineskip=5pt\box1\box2}}_{#1}}%
	\nolimits^{#2}
	\fi\fi}
\newcommand{\higherlimm}[2]{\setbox1=\hbox{\rm lim}
	\setbox2=\hbox to \wd1{\leftarrowfill} \ht2=0pt \dp2=-1pt
	\mathop{\mathop{\vtop{\baselineskip=5pt\box1\box2}}}\limits_{#1}
	\nolimits^{#2}}

\newcounter{let} 
\loop\stepcounter{let}
\expandafter\edef\csname cal\alph{let}\endcsname%
{\noexpand\mathcal{\Alph{let}}}
\ifnum\thelet<26\repeat

\newcommand{\tdef}[2][]{\expandafter\newcommand\csname#2\endcsname%
{#1\textup{#2}}}
\tdef{Iso}   \tdef{Aut}    \tdef{Out}    \tdef{Inn}    \tdef{Hom}
\tdef{End}   \tdef{Inj}    \tdef{map}    \tdef{Ker}    \tdef{Ob}
\tdef{Mor}   \tdef{Res}    \tdef{Spin}   \tdef{Id}     \tdef{Fr}
\tdef{rk}    \tdef{conj}   \tdef{incl}   \tdef{proj}   \tdef{diag} 
\tdef{trf}   \tdef{Sol}    \tdef{He}     \tdef{cj}
\tdef{free}  \tdef{supp}   \tdef{Ind}    \tdef{Ly}     \tdef{HS}
\tdef[_]{typ}   \tdef[^]{op}    \tdef[^]{ab}  \tdef{Tr}

\newcommand{\ON}{\textup{O'N}}

\newcommand{\fdef}[1]{\expandafter\newcommand\csname#1\endcsname%
{\mathfrak{#1}}}
\fdef{X}  \fdef{red}  \fdef{foc}  \fdef{hyp}

\newcommand{\bbdef}[1]{\expandafter\newcommand% 
\csname#1\endcsname{\mathbb{#1}}}
\bbdef{C} \bbdef{F} \bbdef{R} \bbdef{Z} \bbdef{G}

\newcommand{\itdef}[1]{\expandafter\newcommand\csname#1\endcsname%
{\textit{#1}}}
\itdef{PSL}  \itdef{PSU}  \itdef{SL}  \itdef{SU}  \itdef{GL} \itdef{GU}
\itdef{Sp}   \itdef{PSp}  \itdef{SO}  \itdef{SD}  \itdef{UT} \itdef{PGL}
\itdef{Sz}

\newcommand\PSSL{\textit{P$\varSigma$L}}
\newcommand\GGL{\textit{$\varGamma$L}}
\newcommand\PGGL{\textit{P$\varGamma$L}}

\newcommand{\gee}{\varepsilon}

\newcommand{\gen}[1]{\langle{#1}\rangle}
\newcommand{\Gen}[1]{\bigl\langle{#1}\bigr\rangle}

\let\nsg=\normal
\let\nnsg=\ntrianglelefteq

\newcommand{\syl}[2]{\textup{Syl}_{#1}(#2)}
\newcommand{\sylp}[1]{\syl{p}{#1}}
\newcommand{\autf}{\Aut_{\calf}}

\newcommand{\outf}{\Out_{\calf}}

\newcommand{\homf}{\Hom_{\calf}}
\newcommand{\isof}{\Iso_{\calf}}
\newcommand{\sminus}{\smallsetminus}
\newcommand{\defeq}{\overset{\textup{def}}{=}}

\renewcommand{\Im}{\textup{Im}}

\newcommand{\pcom}{{}^\wedge_p}
\newcommand{\sd}[1]{\overset{{#1}}{\rtimes}}

\let\til=\widetilde

\let\too=\longrightarrow
\let\xto=\xrightarrow
\newcommand{\longleft}[1]{\;{\leftarrow%
\count255=0 \loop \mathrel{\mkern-6mu}%
    \relbar\advance\count255 by1\ifnum\count255<#1\repeat}\;}
\newcommand{\longright}[1]{\;{\count255=0 \loop \relbar\mathrel{\mkern-6mu}%
    \advance\count255 by1\ifnum\count255<#1\repeat\rightarrow}\;}
\newcommand{\Right}[2]{\overset{#2}{\longright#1}}
\newcommand{\RIGHT}[3]{\mathrel{\mathop{\kern0pt\longright#1}
	\limits^{#2}_{#3}}}

\newcommand{\LEFT}[3]{\mathrel{\mathop{\kern0pt\longleft#1}\limits^{#2}_{#3}}
}
\newcommand{\longleftright}[1]{\;{\leftarrow\mathrel{\mkern-6mu}%
    \count255=0\loop\relbar\mathrel{\mkern-6mu}% 
    \advance\count255 by1\ifnum\count255<#1\repeat\rightarrow}\;} 
 
\newcommand{\onto}[1]{\;{\count255=0 \loop \relbar\joinrel
    \advance\count255 by1
    \ifnum\count255<#1 \repeat \twoheadrightarrow}\;}

\newcommand{\RLEFT}[3]{\mathrel{%
   \mathop{\vcenter{\baselineskip=0pt\hbox{$\kern0pt\longright#1$}%
   \hbox{$\kern0pt\longleft#1$}}}\limits^{#2}_{#3}}}

\renewenvironment{enumerate}[1][]
{\begin{enumerat}[#1]\setlength{\itemsep}{6pt}}{\end{enumerat}}

\renewenvironment{itemize}
{\begin{itemiz}\setlength{\itemsep}{6pt}\setlength{\itemindent}{-20pt}}
{\end{itemiz}}

\newenvironment{enuma}{\begin{enumerate}[{\rm(a) }]}{\end{enumerate}}
\newenvironment{enumi}{\begin{enumerate}[{\rm(i) }]}{\end{enumerate}}

\newtheorem{Thm}{Theorem}[section]
\newtheorem{Prop}[Thm]{Proposition}

\newtheorem{Lem}[Thm]{Lemma}

\newtheorem{Defi}[Thm]{Definition}   %% mudei para o proximo groupo

\newtheorem{Th}{Theorem}

\theoremstyle{definition}

\theoremstyle{remark}

\title{Reduced fusion systems over $2$-groups of small order}

\author{Kasper K. S. Andersen}
\address{Center for Mathematical Sciences, LTH, Box 118, 22100 Lund, Sweden}
\email{kksa@maths.lth.se}
\thanks{K. K. S. Andersen was partially supported by the Danish National 
Research Foundation (DNRF) through the Centre for Symmetry and Deformation, 
and by VILLUM FONDEN through the network for Experimental Mathematics in 
Number Theory, Operator Algebras, and Topology.}

\author{Bob Oliver}
\address{Universit\'e Paris 13, Sorbonne Paris Cit\'e, LAGA, UMR 7539 du CNRS, 
99, Av. J.-B. Cl\'ement, 93430 Villetaneuse, France.}
\email{bobol@math.univ-paris13.fr}
\thanks{B. Oliver is partially supported by UMR 7539 of the CNRS, and by 
project ANR BLAN08-2\_338236, HGRT}

\author{Joana Ventura}
\address{Departamento de Matem\'atica, Instituto Superior T\'ecnico, Av.
Rovisco Pais, 1049--001 Lisboa, Portugal}
\email{jventura@math.ist.utl.pt}
\thanks{J. Ventura was partially supported by FCT through program POCI 2010/FEDER 
and project PTDC/MAT/098317/2008.}


\begin{document}

\begin{abstract} 
We prove, when $S$ is a $2$-group of order at most $2^9$, that each reduced 
fusion system over $S$ is the fusion system of a finite simple group and is 
tame. It then follows that each saturated fusion system over a $2$-group of 
order at most $2^9$ is realizable. What is most interesting about this result is 
the method of proof: we show that among $2$-groups with order in this 
range, the ones which can be Sylow $2$-subgroups of finite simple groups 
are almost completely determined by criteria based on Bender's 
classification of groups with strongly $2$-embedded subgroups.
\end{abstract}

\maketitle

A \emph{saturated fusion system} over a finite $p$-group $S$ is a category 
whose objects are the subgroups of $S$, whose morphisms are monomorphisms 
between subgroups, and which satisfy certain axioms first formulated by 
Puig \cite{Puig} and motivated in part by conjugacy relations among 
$p$-subgroups of a given finite group.  
A saturated fusion system is \emph{realizable} if it is isomorphic to the 
fusion system defined by the conjugation relations within a Sylow 
$p$-subgroup of some finite group, and is \emph{exotic} otherwise.  One of 
our main goals is to try to understand when and how exotic fusion systems 
can occur, especially over $2$-groups.

A saturated fusion system $\calf$ is \emph{reduced} if $O_p(\calf)=1$ and 
$O^p(\calf)=O^{p'}(\calf)=\calf$ (see Definitions \ref{d:subgroups}(c,e) 
and \ref{d:reduced}(a)).  A saturated fusion system $\calf$ is \emph{tame} 
if it is realized by a group $G$ such that the natural homomorphism from 
$\Out(G)$ to a certain group of outer automorphisms of $\calf$ (more 
precisely, of an associated linking system) is split surjective (Definition 
\ref{d:tame}).  The main result in our earlier paper \cite{AOV1} says 
roughly that exotic fusion systems can be detected via tameness of 
associated reduced fusion systems.  More precisely, by \cite[Theorems A \& 
B]{AOV1}, if the ``reduction'' of a fusion system $\calf$ is tame, then 
$\calf$ is tame and hence realizable, while if a reduced fusion system is 
not tame, then it is the reduction of an exotic fusion system.

A saturated fusion system is \emph{indecomposable} if it does not split as a 
product of fusion systems over nontrivial $p$-groups. 
We can now state our main result.

\begin{Th} \label{ThA}
Let $\calf$ be a reduced, indecomposable fusion system over a nontrivial 
$2$-group of order at most $2^9$.  Then $\calf$ is the fusion system of a 
finite simple group, and is tame.
\end{Th}

\begin{proof} This is shown in Theorems \ref{t:order64} (for $2$-groups of 
order at most $64$), \ref{t:order128} (order $2^7$), \ref{t:order256} 
(order $2^8$), and \ref{t:order512} (order $2^9$). 
\end{proof}

The next theorem follows from Theorem \ref{ThA} and the above discussion.

\begin{Th} \label{ThB}
Each saturated fusion system over a $2$-group of order at most $2^9$ 
is realizable.
\end{Th}

\begin{proof} Let $\calf$ be a saturated fusion system over a $2$-group $S$ 
of order at most $2^9$. The reduction $\red(\calf)$ of $\calf$ is defined 
in \cite[\S\,2]{AOV1} (see Definition \ref{d:reduced}(d) below): it is a 
reduced fusion system over a subquotient of $S$. Since $\red(\calf)$ is 
tame by Theorem \ref{ThA}, $\calf$ is realizable by \cite[Theorem A]{AOV1} 
(Theorem \ref{t:reduce}(a) below). 
\end{proof}

Our proofs of these results are based in large part on computer 
computations.  Their starting point is the version of 
Alperin's fusion theorem (re)stated in Proposition \ref{p:AFT-E}: each 
morphism in a saturated fusion system $\calf$ is a composite of 
restrictions of $\calf$-automorphisms of $S$ and of certain 
``$\calf$-essential'' subgroups.  We refer to Definition 
\ref{d:ess-crit}(b) for the definition of $\calf$-essential.  

In \cite{OV2}, a procedure was developed for determining all reduced fusion 
systems over a given 2-group, taking as examples two groups of order $2^7$ 
and two of order $2^{10}$. The idea was to first determine those subgroups 
of a given $S$ which could potentially be essential in some fusion system 
over $S$ (the ``critical'' subgroups), and then study what their 
$\calf$-automorphism groups could be. In this paper, we first made a 
computer search (using Magma \cite{magma} and GAP \cite{gap}) of all 
2-groups of order at most $2^9$, to determine which of them have ``enough'' 
critical subgroups and satisfy other conditions which are necessary to 
support a reduced fusion system.  These search criteria (listed in 
Proposition \ref{p:search-criteria}) are, in fact, satisfied by very few 
$2$-groups. Reduced fusion systems over them are listed individually, using 
computer computations in some cases and computer-free proofs in others.

The following table shows how close these programs come to restricting 
attention only to groups which are Sylow 2-subgroups of simple groups.
	\[ \renewcommand{\arraystretch}{1.5} 
	\begin{array}{|l||c|c|c|} \hline
	\textup{Group order} & 2^7 & 2^8 & 2^9 \\\hline
	\textup{Number of groups} & 2328 & 56092 & \approx10^7 \\\hline
	\textup{Nr. satisfying conditions in \ref{p:search-criteria}} 
	& 9 & 20 & 34 \\\hline
	\textup{\qquad Sylows of simple groups} & 6 & 6 & 10 \\\hline
	\textup{\qquad Split as products} & 2 & 10 & 23 \\\hline
	\textup{\qquad Others} & 1 & 4 & 1 \\ \hline 
	\end{array} \]
More precisely, the number given in the third row of the table is the 
number of groups of the given order which satisfy the conditions in 
Proposition \ref{p:search-criteria}, together with the dihedral and 
semidihedral groups of that order, and the wreathed groups $C_{2^n}\wr C_2$ 
if there are any. (These latter were eliminated by condition (a) or (b) in 
Proposition \ref{p:search-criteria}, and restored afterwards.) Thus among 
the groups not eliminated by these formal conditions (based mostly on 
Bender's theorem \cite[Satz 1]{Bender}), most are either Sylow 2-subgroups 
of simple groups, or are products of smaller groups and cannot be Sylow 
2-subgroups of simple groups nor of reduced fusion systems. Note that this 
dichotomy applies only in this range: the group $(D_8\wr C_2)\times D_8$ of 
order $2^{10}$ is a Sylow 2-subgroup of the simple group $A_{14}$ (and its 
fusion system is reduced and indecomposable).

There are many examples, especially among finite simple groups of Lie type, 
of different simple groups whose fusion systems (at some given prime $p$) 
are reduced and isomorphic. The following theorem gives some examples of 
this.  We do not use this theorem, except to motivate our giving only one 
example (or one family of examples) of groups which realize any given 
fusion system. The cases most relevant to this paper are those where 
$\G=\PSL_n$ or $\PSp_{2n}$ for some $n\ge2$. 

\begin{Thm}[{\cite[Theorem A]{BMO1}}] \label{BMO1-ThA} 
Fix a prime $p$, a connected reductive group scheme $\G$ over $\Z$, and a 
pair of prime powers $q$ and $q'$ both prime to $p$.  Then the following 
hold, where ``\,$G\sim_pH$\!'' means that the $p$-fusion systems of $G$ 
and $H$ are isomorphic. 
\begin{enuma} 
\item If $\widebar{\gen{q}}=\widebar{\gen{q'}}$ as closed subgroups of 
$\Z_p^\times$, then $\G(q)\sim_p\G(q')$.

\item If $\G$ is of type $A_n$, $D_n$, or $E_6$, $\tau$ is a graph 
automorphism of $\G$, and 
$\widebar{\gen{q}}=\widebar{\gen{q'}}\le\Z_p^\times$, then 
${}^\tau\G(q)\sim_p{}^\tau\G(q')$.

\item If the Weyl group of $\G$ contains an element which acts on the 
maximal torus by inverting all elements, and 
$\widebar{\gen{-1,q}}=\widebar{\gen{-1,q'}}\le\Z_p^\times$, then 
$\G(q)\sim_p\G(q')$.

\item If $\widebar{\gen{-q}}=\widebar{\gen{q'}}\le\Z_p^\times$, then 
$\PSU_n(q)\sim_p\PSL_n(q')$ for all $n\ge2$.

\end{enuma}
\end{Thm}

For example, by (a), if $\calf$ is the fusion system of $\PSL_3(17)$ (for 
$p=2$), then it is also isomorphic to the fusion system of $\PSL_3(q)$ for 
each $q\equiv17$ (mod $32$).

Background results on fusion systems are given in Section 
\ref{s:background}, and the precise criteria which we use in our computer 
searches are listed in Section \ref{s:magma}. In Section \ref{s:Q<|F}, we 
look at the special case of reduced fusion systems over nonabelian 
$2$-groups of the form $S_0\times A$ with $A\ne1$ abelian. 
Afterwards, we handle the individual cases in Theorem \ref{ThA}: 
groups of order at most $2^7$ in Section \ref{s:128}, those of order $2^8$ 
in Section \ref{s:256}, and those of order $2^9$ in Sections 
\ref{s:512}--\ref{s:a12}. At the end, some standard background results 
about groups and representations are given in an appendix.

When $G$ is a group and $g,h\in{}G$, we write $[g,h]=ghg^{-1}h^{-1}$ for 
the commutator. Similarly, if $\alpha\in\Aut(G)$, we set 
$[\alpha,g]=\alpha(g)g^{-1}$ for $g\in G$. Also, $\9gx=gxg^{-1}$ and 
$x^g=g^{-1}xg$ in this situation, and $c_x$ denotes the homomorphism 
$(g\mapsto\9xg$). As usual, when $S$ is a finite $p$-group, $Z_2(S)\nsg S$ 
is defined by $Z_2(S)/Z(S)=Z(S/Z(S))$, its Frattini subgroup 
is denoted $\Phi(S)$, its rank is denoted $\rk(S)$, and $S$ is said to be 
\emph{of type $G$} if it is isomorphic to a Sylow $p$-subgroup of the 
finite group $G$. Also, $\sylp{G}$ is the set of Sylow $p$-subgroups 
of $G$. When $n\ge2$ and $q$ is a prime power, we let 
$\UT_n(q)\le\SL_n(q)$ be the subgroup of upper triangular matrices 
with 1's on the diagonal.  Also, we follow the usual notation for 
extraspecial $2$-groups: $2^{1+2n}_+$ is a central product of $n$ copies of 
$D_8$, while $2^{1+2n}_-$ is a central product of $n{-}1$ copies of $D_8$ 
and one copy of $Q_8$. We write $G_1\times_ZG_2$ to denote a central 
product of groups $G_1$ and $G_2$ over $Z\le Z(G_i)$. 

In Sections \ref{s:128}, \ref{s:256}, and \ref{s:512}, we frequently 
refer to the ``Magma/GAP numbers'' of $2$-groups of a given order. These 
are the numbers given by the ``Small Groups library'' (see 
\texttt{http://www.icm.tu-bs.de/ag\_algebra/software/small/}), and used by 
both Magma and GAP when referring to $2$-groups of order at most $2^9$ (as 
well as groups of other small orders).

All three authors would like to thank the Mathematics Institute at 
Copenhagen University, and especially the Centre for Symmetry and 
Deformation, for their hospitality while much of this work was carried out. 
We would also like to thank Jacob Weismann and James Avery for their 
support with some of the computer computations, especially those involving 
the groups of order $512$. Finally, we are grateful to the referee for 
reading this paper so carefully and giving us many suggestions for 
improving it.

\bigskip

\section{Background results about fusion systems}
\label{s:background}

In this section, we collect a few definitions and results needed in the 
rest of the paper.  A \emph{fusion system} over a finite $p$-group $S$ is a 
category $\calf$ whose objects are the subgroups of $S$, and whose morphism 
sets $\homf(P,Q)\subseteq\Inj(P,Q)$ contain all homomorphisms induced by 
conjugation in $S$. One also requires that each morphism in $\calf$ factors 
as the composite of an isomorphism (in $\calf$) followed by an inclusion. 

For such $\calf$, a subgroup $P\le S$ is called \emph{fully 
normalized} (\emph{fully centralized}) if $|N_S(P)|\ge|N_S(Q)|$ 
($|C_S(P)|\ge|C_S(Q)|$) for each $Q$ in the $\calf$-isomorphism class of 
$P$. The fusion system $\calf$ is \emph{saturated} if it satisfies the 
following two axioms:
\begin{enumerate}[(I) ]
\item (Sylow axiom) If $P\le S$ is fully normalized, then $P$ is fully 
centralized and $\Aut_S(P)\in\sylp{\autf(P)}$.
\item (Extension axiom) If $\varphi\in\isof(P,Q)$ where $Q$ is fully 
centralized, and we set
	\[ N_\varphi = \bigl\{ g\in N_S(P) \,\big|\, \varphi 
	c_g\varphi^{-1}\in\Aut_S(Q) \bigr\}, \]
then there exists $\4\varphi\in\homf(N_\varphi,S)$ such that 
$\4\varphi|_P=\varphi$. 
\end{enumerate}
We refer to 
\cite[Definition 1.2]{BLO2}, \cite[\S\,I.2]{AKO}, and \cite[\S\,4.1]{Craven} 
for more details and notation. 
For example, when $G$ is a finite group with $S\in\sylp{G}$, the fusion 
system of $G$ is the category $\calf_S(G)$, where for $P,Q\le{}S$, 
$\Hom_{\calf_S(G)}(P,Q)=\Hom_G(P,Q)$:  the set of homomorphisms which are 
induced by conjugation in $G$.  

If $\calf$ and $\cale$ are saturated fusion systems over $S$ and $T$, 
respectively, then we say that $\calf$ and $\cale$ are \emph{isomorphic} if 
there is an isomorphism $\alpha\:S\xto\cong T$ such that for each $P,Q\le 
S$, $\Hom_\cale(\alpha(P),\alpha(Q))=\alpha\homf(P,Q)\alpha^{-1}$. Thus 
$\calf\cong\cale$ as fusion systems if there is an isomorphism of 
categories induced by an isomorphism between the underlying $p$-groups. A 
fusion system $\calf$ over $S$ is \emph{realizable} if 
$\calf\cong\calf_S(G)$ for some finite group $G$ with $S\in\sylp{G}$, and 
is \emph{exotic} otherwise. 

We say that two subgroups $P,Q\le{}S$ are \emph{$\calf$-conjugate} if they 
are isomorphic in the category $\calf$.  Let $P^\calf$ be the set of all 
subgroups of $S$ which are $\calf$-conjugate to $P$.

\begin{Defi} \label{d:subgroups}
Fix a prime $p$, a finite $p$-group $S$, and a saturated fusion system $\calf$
over $S$.  For each subgroup $P\le{}S$,
\begin{enuma} 

\item $P$ is \emph{$\calf$-centric} if $C_S(Q)=Z(Q)$ for each $Q\in 
P^\calf$;

\item $P$ is \emph{normal in $\calf$} (denoted $P\nsg\calf$) if $P\nsg{}S$, and 
every morphism $\varphi\in\homf(Q,R)$ in $\calf$ extends to a morphism 
$\widebar{\varphi}\in\homf(PQ,PR)$ such that $\widebar{\varphi}(P)=P$; and  

\item $P$ is \emph{strongly closed} in $\calf$ if for each $Q\le P$ and each 
$\varphi\in\homf(Q,S)$, $\varphi(Q)\le P$.

\item The maximal normal $p$-subgroup of a saturated fusion system 
$\calf$ is denoted $O_p(\calf)$. This is defined since for $A,B\le S$ 
normal in $\calf$, $AB$ is also normal in $\calf$.

\item For any $\varphi\in\Aut(S)$, $\9\varphi\calf$ denotes the
fusion system over $S$ defined by
        $$ \Hom_{\9\varphi\calf}(P,Q) =
        \varphi\circ\homf(\varphi^{-1}(P),\varphi^{-1}(Q))
        \circ\varphi^{-1} $$
for all $P,Q\le{}S$.

\end{enuma}
\end{Defi}

\subsection{Essential and critical subgroups}

\leavevmode\noindent

We next define the \emph{essential} subgroups in a fusion system $\calf$, 
and describe how their automorphisms and those of $S$ generate $\calf$.

\begin{Defi} \label{d:ess-crit}
\begin{enuma} 

\item If $p$ is a prime and $H<G$ are finite groups, then $H$ is 
\emph{strongly $p$-embedded} in $G$ if $p\big\vert|H|$, and 
$H\cap\9xH$ has order prime to $p$ for each $x\in G{\sminus}H$. 

\item Let $S$ be a finite $p$-group.  A subgroup $P$ of $S$ is 
\emph{critical} if $P<S$, $P$ is centric in $S$, and there are subgroups 
$G_0$ and $G$ of $\Out(P)$ such that
	\[ \Out_S(P) \le G_0 < G \le \Out(P)\,, \]
$G_0$ is strongly $p$-embedded in $G$, and $\Out_S(P)\in\sylp{G}$.

\item If $\calf$ is a saturated fusion system over a finite $p$-group $S$, 
then a subgroup $P$ of $S$ is \emph{$\calf$-essential} if $P<S$, $P$ is 
$\calf$-centric and fully normalized in $\calf$, and $\outf(P)$ contains a 
strongly $p$-embedded subgroup.  We let $\EE_\calf$ denote the set of 
$\calf$-essential subgroups of $S$.
\end{enuma}
\end{Defi}

Note that by definition, if $P$ is critical in $S$, then $\alpha(P)$ is 
critical in $S$ for each $\alpha\in\Aut(S)$.

In the situation of Definition \ref{d:ess-crit}(b), $O_p(G)=1$ since 
$G$ contains a strongly $p$-embedded subgroup (cf. \cite[Proposition 
A.7(c)]{AKO}), and so $\Out_S(P)\cap{}O_p(\Out(P))\le O_p(G)=1$. Thus our 
definition of critical subgroup is equivalent to that of \cite[Definition 
3.1]{OV2}.

We refer to \cite[\S\,6.4]{Sz2}, and also to \cite[Proposition A.7]{AKO}, 
for some of the other properties of strongly $p$-embedded subgroups.  

The next proposition, first shown by Puig, is a version of Alperin's 
fusion theorem for saturated fusion systems. It explains the importance of 
essential subgroups.

\begin{Prop}[{\cite[Proposition 1.10(a,b)]{O-split}}] \label{p:AFT-E}
For any saturated fusion system $\calf$ over a finite $p$-group $S$, each 
morphism in $\calf$ is a composite of restrictions of automorphisms in 
$\autf(S)$, and of automorphisms in $O^{p'}(\autf(P))$ for $P\in\EE_\calf$.
\end{Prop}

The next lemma is an immediate consequence of the definitions:  critical 
subgroups in $S$ are the ones which can be essential in a saturated fusion 
system over $S$. 

\begin{Prop}[{\cite[Proposition 3.2]{OV2}}] \label{ess=>crit}
If $\calf$ is a saturated fusion system over a finite $p$-group $S$, and 
$P\in\EE_\calf$, then $P$ is a critical subgroup of $S$.  
\end{Prop}

The next three propositions give some necessary conditions for a subgroup 
to be critical, and hence necessary conditions for it to be essential.  

\begin{Prop}[{\cite[Lemma 3.4]{OV2}}] \label{p:QcharP}
Fix a prime $p$, a finite $p$-group $S$, a subgroup $P\le{}S$, and a
subgroup $\Theta$ characteristic in $P$. Assume there is
$g\in{}N_S(P){\sminus}P$ such that
\begin{enuma} 
\item $[g,P]\le{}\Theta\cdot\Phi(P)$, and
\item $[g,\Theta]\le\Phi(P)$.
\end{enuma}
Then $c_g\in{}O_p(\Aut(P))$, and hence $P$ is not critical. 
\end{Prop}

The next proposition is a consequence of Bender's classification \cite[Satz 
1]{Bender} of groups with strongly 2-embedded subgroups.

\begin{Prop}[{\cite[Proposition 3.3(a,c,d)]{OV2}}] \label{p:critical}
Let $S$ be a finite $2$-group, and let $P\le S$ be a critical subgroup. Set
$S_0=N_S(P)/P\cong\Out_S(P)$. Then the following hold.
\begin{enuma}
\item Either $S_0$ is cyclic, or $Z(S_0)=\{g\in{}S_0\,|\,g^2=1\}$. 
If $Z(S_0)$ is not cyclic, then $|S_0|=|Z(S_0)|^m$ for $m=1$, $2$, or $3$. 

\item Set $|S_0|=2^k$.  Then $\rk(P/\Phi(P))\ge2k$. If $k\ge2$, then 
$\rk([s,P/\Phi(P)])\ge2$ for all $1\ne{}s\in{}S_0$.

\item Assume $Z(S_0)\cong{}(C_2)^n$ with $n\ge2$, and fix
$1\ne{}s\in{}Z(S_0)$.  Then $\rk([s,P/\Phi(P)])\ge{}n$.
\end{enuma}
\end{Prop}

We also need the following refinement of the last proposition. 

\begin{Prop} \label{p:critical2}
Let $S$ be a finite $2$-group, and let $P\le S$ be a critical subgroup. 
Let $k$ be such that $2^k=|N_S(P)/P|=|\Out_S(P)|$, and let 
$\Phi(P)=P_0<\cdots<P_r=P$ be a sequence of subgroups 
characteristic in $P$.  Then there is some $1\le i\le r$ such that 
$\rk(P_i/P_{i-1})\ge2k$, and such that if $k\ge2$, then 
$\rk([s,P_i/P_{i-1}])\ge2$ for each $1\ne s\in\Out_S(P)$.  
\end{Prop}

\begin{proof}  Let $\Gamma\le\Out(P)$ be such that 
$\Out_S(P)\in\syl2\Gamma$ and $\Gamma$ has a strongly $2$-embedded 
subgroup.  Set $K_i=C_\Gamma(P_i/P_{i-1})\nsg\Gamma$: the kernel of the 
$\Gamma$-action on $P_i/P_{i-1}$.  Set $K=\bigcap_{i=1}^rK_i$ (so 
$K\nsg\Gamma$). Then $K\le O_2(\Gamma)$ by Lemma \ref{l:mod-Fr} and 
$O_2(\Gamma)=1$ (see \cite[Proposition A.7(c)]{AKO}), so $K=1$.

Fix an involution $t\in\Out_S(P)$.  Choose some $g\in\Gamma$ which does not 
commute with $t$, and choose $i$ such that $[g,t]\notin{}K_i$. Then $K_i$ 
has odd order since all involutions in $\Gamma$ are conjugate to $t\notin 
K_i$ (see \cite[6.4.4(i)]{Sz2}). Also, $\Gamma/K_i$ has at least two 
distinct involutions:  the images of $t$ and $gtg^{-1}$. If 
$\Out_S(P)$ is cyclic or generalized quaternion, it now follows that  
$\Gamma/K_i$ has a strongly $2$-embedded subgroup (the centralizer of one 
of the involutions).  Otherwise, by Bender's theorem (see \cite[Theorem 
6.4.2]{Sz2}), each strongly 2-embedded subgroup of $\Gamma$ contains 
$O_{2'}(\Gamma)$ and hence contains $K_i$, and so $\Gamma/K_i$ still has a 
strongly $2$-embedded subgroup by Lemma \ref{str.emb.->>}(b).  So in either 
case, the proposition follows from \cite[Lemma 1.7]{OV2}, applied with 
$V=P_i/P_{i-1}$. 
\end{proof}

\subsection{Reduced fusion systems and tame fusion systems}

\leavevmode\noindent

Recall that a saturated fusion system is called ``exotic'' if it is not 
isomorphic to $\calf_S(G)$ for any finite group $G$ with $S\in\sylp{G}$. 
In \cite{AOV1}, we described how we could restrict attention to a 
smaller class of saturated fusion systems which we call \emph{reduced 
fusion systems}, and still ``detect'' any exotic fusion systems (reduced or 
not) which reduce to them.  To make this more 
precise, we list here the main results of \cite{AOV1}: Theorem 
\ref{t:reduce} below.  We first need some more definitions.

Let $\calf$ be a saturated fusion system over a finite $p$-group $S$.  The 
\emph{focal subgroup} of $\calf$ is the subgroup
        \[ \foc(\calf) \defeq 
        \Gen{s^{-1}t\,\big|\,s,t\in{}S \textup{ and $\calf$-conjugate} }
        = \gen{[\autf(P),P] \,|\, P\le{}S }, \]
where the last two subgroups are equal by Proposition \ref{p:AFT-E} (Alperin's 
fusion theorem). The \emph{hyperfocal subgroup} of $\calf$ is the subgroup
        \[ \hyp(\calf) = \Gen{[O^p(\autf(P)),P] \,\big|\, P\le{}S }. \]
Equivalently, in the definition of $\hyp(\calf)$, we can restrict to 
automorphisms of order prime to $p$.  It is not hard to see that the image 
of the focal subgroup in $S/\hyp(\calf)$ is precisely its commutator 
subgroup $[S,S]\hyp(\calf)/\hyp(\calf)$ (cf. \cite[Lemma I.7.2]{AKO}).  Hence 
$\hyp(\calf)$ is a proper 
subgroup of $S$ if and only if $\foc(\calf)$ is a proper subgroup.  

As an immediate consequence of Proposition \ref{p:AFT-E} and the 
definition of $\foc(\calf)$, we have:

\begin{Prop} \label{foc(F)=<->}
For any saturated fusion system $\calf$ over a finite $p$-group $S$, 
	\[ \foc(\calf) = \Gen{[\autf(P),P]\,\big|\, 
	P\in\EE_\calf\cup\{S\} }. \]
\end{Prop}

By \cite[Theorems 4.3 \& 5.4]{BCGLO2}, there is a unique saturated fusion 
subsystem $O^p(\calf)\subseteq\calf$ over $\hyp(\calf)$ such that 
$\Aut_{O^p(\calf)}(P)\ge O^p(\autf(P))$ for all $P\le\hyp(\calf)$; and a 
unique saturated fusion subsystem $O^{p'}(\calf)\subseteq\calf$ over $S$ 
minimal with respect to the property that 
$\Aut_{O^{p'}(\calf)}(P)\ge O^{p'}(\autf(P))$ for all $P\le S$. This leads 
to the following definition.

\begin{Defi} \label{d:reduced}
Let $\calf$ be a saturated fusion system over a finite $p$-group 
$S$. Then 
\begin{enuma} 

\item $\calf$ is \emph{reduced} if $O_p(\calf)=1$ and 
$O^p(\calf)=O^{p'}(\calf)=\calf$; 

\item $\calf$ is \emph{decomposable} if there are subgroups $1\ne 
S_i\nsg S$ and saturated fusion systems $\calf_i$ over $S_i$ ($i=1,2$) such 
that $S=S_1\times S_2$ and $\calf=\calf_1\times\calf_2$, and $\calf$ is 
\emph{indecomposable} otherwise; and

\item $\calf$ is \emph{simple} if there are no proper normal subsystems 
$\cale\nsg\calf$ with $1\ne\cale\subsetneqq\calf$ (see \cite[Definition 
I.6.1]{AKO} for the definition of a normal fusion subsystem).

\item The \emph{reduction} of $\calf$ is the fusion system $\red(\calf)$ 
obtained by first setting $\calf_0=C_\calf(O_p(\calf))/Z(O_p(\calf))$, and 
then letting $\red(\calf)$ be the limiting term of the sequence 
$\calf_0\supseteq\calf_1\supseteq\calf_2\supseteq\dots$, where 
$\calf_{i+1}=O^{p'}(O^p(\calf_i))$ for all $i\ge0$.

\end{enuma}
\end{Defi}

Here, $C_\calf(Q)\subseteq\calf$ denotes the centralizer fusion system of 
$Q\le{}S$: the largest fusion subsystem of $\calf$ in which $Q$ is central 
(cf. \cite[Definition I.5.3]{AKO} or \cite[Definition 4.26(i)]{Craven}).  As 
the names suggest, the reduction of any saturated fusion system is reduced 
\cite[Proposition 2.2]{AOV1}.  

Simple fusion systems are always reduced and indecomposable, but the 
converse need not be true. For example, when $p=2$, the fusion system of 
the wreath product $A_6\wr A_5$ is reduced and indecomposable but not 
simple. However, a reduced fusion system which has no proper nontrivial 
strongly closed subgroups is simple (see the proof of Proposition 
\ref{p:FS(G)-red}(d) below).

We next explain how any exotic fusion system $\calf$ can be ``detected'' 
via its reduction $\red(\calf)$.  The key to doing this is the idea of a 
\emph{tame} fusion system. For 
any finite group $G$, $BG\pcom$ denotes the $p$-completion of the 
classifying space of $G$, and $\Out(BG\pcom)$ is the group of homotopy 
classes of homotopy equivalences from the space $BG\pcom$ to itself.

\begin{Defi} \label{d:tame}
A saturated fusion system $\calf$ over a finite $p$-group is 
\emph{tamely realized} by a finite group $G$ if 
\begin{itemize} 

\item $\calf$ is isomorphic to $\calf_T(G)$ (where 
$T\in\sylp{G}$); and 
\item the natural homomorphism $\kappa_G\:\Out(G)\Right3{}\Out(BG\pcom)$ is 
split surjective.
\end{itemize}
The fusion system $\calf$ is \emph{tame} if it is tamely realized by some 
finite group.
\end{Defi}

In fact, in the definition in \cite{AOV1}, we replace the group 
$\Out(BG\pcom)$ by one which is defined purely algebraically, as a certain 
group of outer automorphisms of the centric linking system for $G$ over 
$S$.  We give the above definition here to avoid a long discussion about 
linking systems and their automorphisms.  The two are equivalent by 
\cite[Theorem B]{BLO1}.  

The following theorem was shown in \cite[Theorems A, B, \& C]{AOV1}.  It 
is what motivated Definition \ref{d:tame}.

\begin{Thm}\label{t:reduce}
Let $\calf$ be a saturated fusion system over a finite $p$-group.
\begin{enuma}
\item If $\red(\calf)$ is tame, then $\calf$ is also tame, and in 
particular is realizable as the fusion system of a finite group.

\item If $\calf$ is reduced and each indecomposable factor 
of $\calf$ is tame, then $\calf$ is tame.  

\item If  $\calf$ is reduced and not tame, then there is an exotic fusion 
system $\til\calf$ such that $\red(\til\calf)\cong\calf$.
\end{enuma}
\end{Thm}

Theorem \ref{t:reduce} helps to explain why we only look at reduced, 
indecomposable fusion systems. Points (a) and (b) say that for 
every exotic fusion system $\calf$, the reduction $\red(\calf)$ is not 
tame, and at least one of its indecomposable factors is also not tame.  In 
other words, each exotic fusion system is detected by some reduced, 
indecomposable fusion system which is not tame.

Theorem \ref{t:reduce} shows the importance of 
determining whether a given reduced fusion system is tame.  In general, 
rather than comparing $\Out(G)$ with $\Out(BG\pcom)$, it is much simpler to 
compare $\Out(G)$ with a certain group $\Out(\calf)$ of outer 
automorphisms of $\calf$.

\begin{Defi} \label{d:Out(F)}
For any saturated fusion system $\calf$ over a finite $p$-group $S$, let 
$\Aut(\calf)$ be the group of those $\alpha\in\Aut(S)$ such that 
$\9\alpha\calf=\calf$ (the ``fusion preserving automorphisms'').  
Set $\Out(\calf)=\Aut(\calf)/\Aut_{\calf}(S)$.  
\end{Defi}

When $\calf=\calf_S(G)$ for some finite group $G$ with $S\in\sylp{G}$, 
there is a natural homomorphism from $\Out(G)$ to $\Out(\calf)$ defined 
by restriction (each class in $\Out(G)$ contains automorphisms of $G$ which 
normalize $S$).  By \cite[\S\S\,1.3 \& 2.2]{AOV1}, this map
factors as the composite of homomorphisms 
	\[ \Out(G) \Right4{\kappa_G} \Out(BG\pcom) 
	\Right4{\mu_G} \Out(\calf)\,. \] 

Since $\Out(\calf)$ is in general easier to describe than 
$\Out(BG\pcom)$, the simplest way to prove tameness is usually by showing 
that $\mu_G\circ\kappa_G$ is split surjective and that $\mu_G$ is 
injective.  The following proposition, which is a special case of 
\cite[Proposition 4.2]{AOV1}, suffices in all cases considered in this 
paper for showing that $\Ker(\mu_G)=1$.  

\begin{Prop} \label{p:newKer(mu)}
Fix a finite group $G$ and $S\in\syl2{G}$, and set $\calf=\calf_S(G)$. 
Assume that at most one subgroup $P\in\EE_\calf$ has noncyclic center. Then 
$\Ker(\mu_G)=1$. 
\end{Prop}

\begin{proof} Let $\EE_\calf^0$ be the set of 
those $P\in\EE_\calf$ such that $C_{Z(P)}(\autf(P))< C_{Z(P)}(\Aut_S(P))$. 
If $P\in\EE_\calf$ and $Z(P)$ is cyclic, then each element of odd order in 
$\autf(P)$ acts trivially on $Z(P)$. Then 
$C_{Z(P)}(\Aut_S(P))=C_{Z(P)}(\autf(P))$ (recall $P$ is fully normalized in 
$\calf$), and hence $P\notin\EE_\calf^0$.  

Thus $|\EE_\calf^0|\le1$ under our hypotheses. By 
\cite[Proposition 4.2(d)]{AOV1}, $\Ker(\mu_G)=1$ if 
$\EE_\calf^0=\emptyset$, so assume $\EE_\calf^0=\{P\}$. Then by the same 
result, for each $\alpha\in\Ker(\mu_G)$, there is an element 
$g_P\in C_{Z(P)}(\Aut_S(P))=Z(S)$ with the 
property that $\alpha=1$ if and only if $g_P\in g\cdot 
C_{Z(P)}(\autf(P))$ for some $g\in C_{Z(S)}(\autf(S))$. Note that 
$P\nsg S$ and $\autf(S)$ normalizes $P$ (by the uniqueness of $P$), so 
by \cite[Proposition 4.2(a,c)]{AOV1} (with $P\le S$ in the role of 
$Q\le P$ in point (c)), $g_P\equiv1$ modulo 
$C_{Z(P)}(\autf(S))=C_{Z(S)}(\autf(S))$. Thus $g_P\in 
C_{Z(S)}(\autf(S))$, so $\alpha=1$ by the above remarks, applied with 
$g=g_P$. 
\end{proof}

\subsection{Criteria for detecting reduced fusion systems}

\leavevmode\noindent

We now list some conditions on a finite $p$-group $S$ or on a fusion 
system $\calf$ over $S$ which are necessary for $\calf$ to be reduced (or 
sufficient for $\calf$ to not be reduced).  We begin with a simple criterion 
for detecting normal $p$-subgroups.

\begin{Prop}[{\cite[Proposition I.4.5]{AKO}}] \label{Q<|F}
Let $\calf$ be a saturated fusion system over a finite $p$-group $S$, and 
fix $Q\nsg S$.  Assume, for each $P\in\EE_\calf\cup\{S\}$, that $Q\le{}P$ 
and $\alpha(Q)=Q$ for each $\alpha\in\autf(P)$.  Then $Q\nsg\calf$.
\end{Prop}

The next lemma is the starting point for deciding whether or not 
$O^p(\calf)=\calf$.  

\begin{Lem}[{\cite[Corollary I.7.5]{AKO}}] \label{l:foc=S}
For any saturated fusion system $\calf$ over a finite $p$-group $S$, 
$\calf=O^p(\calf)$ if and only if $\foc(\calf)=S$.  If $\autf(S)$ is a 
$p$-group and $O^p(\calf)=\calf$, then 
	\[ S = \Gen{ [\autf(P),P] \,\big|\, P\in\EE_\calf}. \]
\end{Lem}

\begin{proof} See, e.g., \cite[Corollary I.7.5]{AKO} for a proof of the 
first statement.  The second follows from that, Proposition 
\ref{foc(F)=<->}, and \cite[Theorems 5.1.1(i) \& 5.1.3]{Gorenstein} 
($Q[S,S]=S$ implies $Q=S$).  
\end{proof}

We next look at cases where there are very few essential subgroups.

\begin{Lem} \label{l:1crit}
Let $S$ be a nontrivial finite $p$-group. For any saturated fusion system 
$\calf$ over $S$, if $|\EE_\calf|\le1$, then $O_p(\calf)\ne1$.  If 
$\outf(S)=1$ and $\EE_\calf$ contains exactly one $S$-conjugacy class, then 
$O^p(\calf)\subsetneqq\calf$.  In either case, $\calf$ is not reduced.
\end{Lem}

\begin{proof} If $\EE_\calf=\emptyset$, then 
$S\nsg\calf$ by Proposition \ref{Q<|F}, while if $\EE_\calf=\{P\}$, 
then $P\nsg\calf$ by the same proposition.  
In either case, $O_p(\calf)\ne1$, so $\calf$ is not reduced.  If 
$\outf(S)=1$ and $\EE_\calf=\calp$ for some $S$-conjugacy class $\calp$, 
then $\foc(\calf)$ is contained in $\gen{\calp,[S,S]}<S$, so 
$O^p(\calf)\subsetneqq\calf$, and again $\calf$ is not reduced. 
\end{proof}

In Sections \ref{s:128}--\ref{s:a12}, our main technique for checking that 
fusion systems are realizable is to list all reduced fusion systems over a 
given $2$-group $S$, and then match them with simple groups having Sylow 
$2$-subgroup $S$.  When doing this, it is important to know that the fusion 
systems of the groups in question are reduced, since this is not the case 
for all simple groups (e.g., not for $A_5$).  The following proposition, 
part of which is based on a theorem of Goldschmidt, gives some criteria for 
showing this.

\begin{Prop} \label{p:FS(G)-red}
Let $G$ be a finite nonabelian simple group. Choose $S\in\syl2{G}$, and set 
$\calf=\calf_S(G)$.  
\begin{enuma} 
\item In all cases, $O^2(\calf)=\calf$ and $\calf$ is indecomposable. 

\item If $S$ is nonabelian, and if $G$ is not isomorphic to 
$\PSU_3(2^n)$ ($n\ge2$) nor to $\Sz(2^{2n+1})$ ($n\ge1$), then 
$O_2(\calf)=1$. 

\item If $G$ and $S$ satisfy the hypotheses in (b), and $\Aut(S)$ is a 
$2$-group or $O^{2'}(\calf)=\calf$, then $\calf$ is reduced.

\item If $G$ is a known simple group and $\calf$ is reduced, then $\calf$ 
is simple. In particular, this is the case whenever $G$ and $S$ satisfy the 
hypotheses in (b), and $G$ is an alternating group, a sporadic simple 
group, a simple group of Lie type in defining characteristic $2$, or 
$\lie2F4(2)'$.

\end{enuma}
\end{Prop}

\begin{proof} \textbf{(a) }  By the focal subgroup theorem for groups (cf. 
\cite[Theorem 7.3.4]{Gorenstein}), $\foc(\calf)=S\cap[G,G]$.  Hence 
$\foc(\calf)=S$ since $G$ is simple, and $O^2(\calf)=\calf$ by Lemma 
\ref{l:foc=S}. 

Assume $\calf$ is decomposable: thus $\calf=\calf_1\times\calf_2$, where 
$\calf_i$ is over $S_i\ne1$ and $S=S_1\times S_2$. In particular, $S_1$ and 
$S_2$ are both strongly closed in $S$ with respect to $G$. So by 
\cite[Corollary A1]{Goldschmidt-2}, the normal closures of $S_1$ and 
$S_2$ commute with each other, which is impossible since each has normal 
closure $G$. Thus $\calf$ is indecomposable.

\smallskip

\noindent\textbf{(b) } Set $Q=O_2(\calf)$ for short. If $Q\ne1$, then $1\ne 
Z(Q)\nsg\calf$ (see \cite[Proposition I.4.4]{AKO}), and in particular, 
$Z(Q)$ is strongly closed with respect to $G$. But by \cite[Theorem 
A]{Goldschmidt-1}, under the given assumptions and since $G$ is simple, no 
nontrivial abelian subgroup of $S$ is strongly closed with respect to $G$, 
and thus $O_2(\calf)=Q=1$.

\smallskip

\noindent\textbf{(c) } If $\Aut(S)$ is a $2$-group, then 
$\autf(S)=\Inn(S)=\autf^0(S)$ in the notation of \cite[Theorem I.7.7]{AKO}, 
and hence $O^{2'}(\calf)=\calf$ by that theorem. Together with (a) and (b), 
this shows that $\calf$ is reduced under the above hypotheses. 

\smallskip

\noindent\textbf{(d) } Assume $G$ is a known simple group such that 
$\calf=\calf_S(G)$ is reduced but not simple, and let $1\ne\cale\nsg\calf$ 
be a proper normal subsystem over $T\nsg S$. If $T=S$, then 
$O^{2'}(\calf)\subseteq\cale\subsetneqq\calf$ (see \cite[Lemma 
1.26]{AOV1}), which is impossible since $\calf$ is reduced. Thus $1\ne 
T<S$, where $T$ is strongly closed (this is part of the definition 
of a normal subsystem \cite[Definition I.6.1]{AKO}). By \cite[Theorem 1]{Foote}, then there is a 
nontrivial abelian subgroup $Q\le T$ which is strongly closed in $\calf$, 
and $Q\nsg\calf$ by \cite[Corollary I.4.7(a)]{AKO}, contradicting the 
assumption that $\calf$ is reduced. Hence $\calf$ is simple.

The last statement is shown in \cite[\S\,16]{A-gfit}: in 16.3 
(simple groups of Lie type in characteristic 2 and $\lie2F4(2)'$), 16.5 
($A_n$), and 16.8 (sporadic simple groups).
\end{proof}

We refer to \cite{FF} and \cite[\S\,16]{A-gfit} for some similar results when 
$p$ is odd.

\bigskip

\section{Computer search criteria}
\label{s:magma}

We now list explicitly the criteria which we use to search for $2$-groups 
which could support reduced fusion systems, and to search for critical 
subgroups of a given $2$-group. Throughout this section, when $H\le G$ are 
finite groups, we let $\trf^G_H$ denote the transfer homomorphism from 
$G/[G,G]$ to $H/[H,H]$ (see, e.g., \cite[\S\,I.8]{AKO}). We will need the 
following application of these homomorphisms.

\begin{Lem} \label{l:usetrf2}
Fix a finite $p$-group $S$.  Assume there is $g\in{}S$ which satisfies
\begin{enuma}  
\item $g\notin[S,S]$;
\item $[\alpha,g]\in[S,S]$ for each $\alpha\in O^p(\Aut(S))$; and

\item $\trf_P^S([g])\in{}P/[P,P]$ is fixed by $O^p(\Aut(P))$ for each 
critical subgroup $P<S$.

\end{enuma}
Alternatively, assume there is $g\in{}S$ which satisfies
\begin{enumerate}[{\rm(a$'$) }]
\item $g\notin\Phi(S)$;
\item $[\alpha,g]\in\Phi(S)$ for each $\alpha\in O^p(\Aut(S))$; and

\item $\trf_P^S([g])\in{}P/\Phi(P)$ is fixed by $O^p(\Aut(P))$ for each 
critical subgroup $P<S$.

\end{enumerate}
In either case, every saturated fusion system over $S$ has a normal 
subsystem of index $p$, and hence there are no reduced fusion systems over 
$S$.  
\end{Lem}

\begin{proof}  We prove this for hypotheses (a)--(c). The proof for 
(a$'$)--(c$'$) holds by the same argument, upon replacing $Q/[Q,Q]$ (for 
$Q\le S$) by $Q/\Phi(Q)$.

Let $\calf$ be a saturated fusion system over $S$. For each $P,Q\le S$ and 
$\varphi\in\homf(P,Q)$, we let $\4\varphi$ denote the induced homomorphism from 
$P/[P,P]$ to $Q/[Q,Q]$. We claim that for each isomorphism 
$\varphi\in\isof(P,Q)$,
	\beqq \widebar\varphi(\trf_P^S([g]))=\trf_Q^S([g])\in Q/[Q,Q] \,. 
	\label{e:tt} \eeqq
Point \eqref{e:tt} holds by (b) (and the naturality 
properties of the transfer) when $\varphi=\alpha|_P$ for some 
$\alpha\in\autf(S)$. If $P,Q\le R<S$ where $R\in\EE_\calf$ (hence $R$ is 
critical in $S$ by Proposition \ref{ess=>crit}), and $\varphi=\alpha|_P$ 
for some $\alpha\in O^p(\autf(R))$, then 
	\[ \4\varphi(\trf_P^S([g])) = \4\varphi(\trf_P^R(\trf_R^S([g]))) 
	= \trf_Q^R(\4\alpha(\trf_R^S([g]))) = \trf_Q^R(\trf_R^S([g])) = 
	\trf_Q^S([g]), \]
where the third equality holds by (c).  Point \eqref{e:tt} now follows from 
Proposition \ref{p:AFT-E}:  each isomorphism $\varphi$ is a composite of 
restrictions of automorphisms of $S$ and of $\calf$-essential (hence 
critical) subgroups.

By \cite[Proposition I.8.4]{AKO}, there is an injective homomorphism
	\[ \trf_\calf\: S/\foc(\calf) \Right5{} S/[S,S] \,, \]
together with proper subgroups $P_1,\dots,P_m<S$ and homomorphisms 
$\varphi_i\in\homf(P_i,S)$ such that for $g\in{}S$, 
	\[ \trf_\calf([g]) = \prod_{[\alpha]\in\outf(S)}\4\alpha([g]) \cdot 
	\prod_{i=1}^m \4{\varphi_i}\bigl(\trf^S_{P_i}([g])\bigr). \]
Set $Q_i=\varphi_i(P_i)$, and let $\varphi'_i\in\isof(P_i,Q_i)$ be the 
restriction of $\varphi_i$.  By \eqref{e:tt}, if we set $k=|\outf(S)|$, 
then 
        \[ \trf_\calf([g]) = [g]^k \cdot \prod_{i=1}^m 
        \4{\incl}_{Q_i}^S\bigl(\4{\varphi'_i}(\trf_{P_i}^S([g]))\bigr)
        = [g]^k \cdot \prod_{i=1}^m \4{\incl}_{Q_i}^S(\trf_{Q_i}^S([g]))
        = \biggl[g^k \cdot \prod_{i=1}^m g^{[S:Q_i]}\biggr] \ne 1 \]
since $p{\nmid}k$ (and $p\big|[S:Q_i]$ for each $i$).  Thus
$g\notin\foc(\calf)$, so $\foc(\calf)<S$, and $\calf$ contains a normal 
subgroup of index $p$ by Lemma \ref{l:foc=S}.
\end{proof}

Let $S$ be a finite $p$-group. A normal subgroup $P\nsg S$ is called 
\emph{semicharacteristic} in $S$ if $P$ is normalized by $O^p(\Aut(S))$. If 
$P_1,P_2\le S$ are semicharacteristic subgroups, then so is $P_1\cap P_2$. 
We can thus define the \emph{semicharacteristic closure} of $Q\leq S$ to be 
the smallest subgroup $P\nsg S$ containing $Q$ which is semicharacteristic 
in $S$.

\begin{Prop} \label{p:search-criteria}
Assume $\calf$ is a reduced fusion system over a finite nonabelian 
$2$-group $S$. Assume also that $S$ is not isomorphic to $D_{2^n}$ 
($n\ge3$), $SD_{2^n}$ ($n\ge4$), or $C_{2^n}\wr{}C_2$ ($n\ge2$). Then $S$ 
satisfies the following conditions.
\begin{enuma}
\item \label{searchcrit:abindex2}
$S$ contains no abelian subgroup of index two.

\item \label{searchcrit:comScyclic}
$[S,S]$ is not cyclic.

\item \label{searchcrit:omega1im}
Let $A$ denote the image of $\Omega_1(Z(S))$ in $S/[S,S]$. Then either
$A=1$, or $\lvert A\rvert>2$ and $\Aut(S)$ is not a $2$-group.

\item \label{searchcrit:trf}
If $\Aut(S)$ is a $2$-group, then
	\begin{equation*}
	\bigcap_{M<S} \Ker\bigl[S/[S,S] \xrightarrow{\trf_M^S}
	  M/[M,M]\bigr] = 1\quad \text{and}\quad
	\bigcap_{M<S} \Ker\bigl[S/\Phi(S) \xrightarrow{\trf_M^S}
	  M/\Phi(M)\bigr] = 1,
	\end{equation*}
where the intersections are taken over all maximal subgroups $M<S$.

\item \label{searchcrit:numberofcrit}
$S$ has more than one critical subgroup, and $S$ has more than one conjugacy 
class of critical subgroups if $\Aut(S)$ is a $2$-group.

\item \label{searchcrit:focal}
For each critical subgroup $P<S$, let $Q_P$ denote the
semicharacteristic closure of $[N_S(P),P]$ in $P$. Then
$\Gen{[O^2(\Aut(S)),S],Q_P\big| \textup{$P$ critical}} = S$.

\item \label{searchcrit:normal}
Let $Q\nsg S$ be any normal subgroup which is semicharacteristic in
$S$, and which is contained in and semicharacteristic in each critical 
subgroup $P<S$. Then $Q=1$.

\item \label{searchcrit:transf} Let $K\leq S/[S,S]$ denote the subgroup 
consisting of those elements $x\in S/[S,S]$ which are fixed by 
$O^2(\Aut(S))$, and which are such that $\trf^S_P(x)\in P/[P,P]$ is fixed 
by $O^2(\Aut(P))$ for each critical subgroup $P<S$. Then $K=1$.

Similarly, let $K'\leq S/\Phi(S)$ denote the subgroup consisting of those 
elements $x\in S/\Phi(S)$ which are fixed by $O^2(\Aut(S))$, and which are 
such that $\trf^S_P(x)\in P/\Phi(P)$ is fixed by $O^2(\Aut(P))$ for each 
critical subgroup $P<S$. Then $K'=1$. 
\end{enuma} 
\end{Prop}

\begin{proof} By Proposition \ref{ess=>crit}, each $\calf$-essential 
subgroup is a critical subgroup of $S$.  This will be used throughout the 
proof.

Points \eqref{searchcrit:abindex2} and \eqref{searchcrit:comScyclic} follow 
from \cite[Proposition 5.2(a,b)]{AOV2}, \eqref{searchcrit:omega1im} follows 
from \cite[Corollary I.8.5]{AKO}, and \eqref{searchcrit:trf} and 
\eqref{searchcrit:transf} from Lemma \ref{l:usetrf2}. Since 
$\calf$-essential subgroups are critical, point 
\eqref{searchcrit:numberofcrit} follows from Lemma \ref{l:1crit}.

\smallskip

\noindent\textbf{(\ref{searchcrit:focal})}
Let $P_1,\dots,P_m$ be conjugacy class representatives for the critical 
subgroups of $S$, and let $Q_i=Q_{P_i}\le{}P_i$ be the semicharacteristic 
closure of $[N_S(P_i),P_i]$.  For each $i$, $Q_i$ is 
$O^2(\autf(P_i))$-invariant and $\Aut_S(P_i)$-invariant, hence 
$\autf(P_i)$-invariant.  So $Q_i\ge[O^{2'}(\autf(P_i)),P_i]$ since 
$O^{2'}(\autf(P_i))$ is generated by $\Aut_S(P_i)\in\syl2{\autf(P_i)}$ and 
its conjugates in $\autf(P_i)$. So if we set 
	\[ H= \Gen{[O^2(\autf(S)),S],Q_i\,\big|\,1\le i\le m}\,, \]
then
	\begin{align*} 
	\foc(\calf) &= \Gen{[\autf(S),S],[O^{2'}(\autf(P_i)),P_i]\,\big|\, 
	i=1,\ldots,m} \\
	&\le [S,S]\cdot\Gen{[O^2(\autf(S)),S],Q_i\,\big|\, i=1,\ldots,m} = 
	[S,S]H 
	\end{align*}
where the first equality follows from Proposition \ref{p:AFT-E}.  
Since $\calf$ is reduced, $\foc(\calf)=S$ by Lemma \ref{l:foc=S}.  Thus 
$[S,S]H=S$, so $H=S$ by \cite[Theorems 5.1.1(i) \& 5.1.3]{Gorenstein}.

\smallskip

\noindent\textbf{(\ref{searchcrit:normal})}
Assume $Q\nsg S$ is normal and semicharacterisitic in $S$, and contained in 
and semicharacteristic in each critical subgroup $P<S$.  Then if $P=S$ or 
$P\in\EE_\calf$, $Q$ is normalized by the action of 
$O^2(\autf(P))\cdot\Aut_S(P)=\autf(P)$.  Hence $Q\nsg\calf$ by 
Proposition \ref{Q<|F}, and $Q=1$ since $\calf$ is reduced.
\end{proof}

We now turn to the criteria used in our computer searches to determine 
(possibly) critical subgroups of a finite 2-group. Note that these are the 
actual criteria used in our computer program, not necessarily the 
optimal conditions which must hold when $P$ is critical. 

\begin{Prop} \label{p:critical-criteria}
Let $S$ be a finite $2$-group, and let $P\le S$ be a critical
subgroup. Let $S_0=N_S(P)/P$. Then the following conditions hold.
\begin{enuma}
\item Let $Z'=\Gen{ x\in Z_2(S)\,\big|\, |[x,S]|\le2}$. Then either 
$Z'\le P$, or there exists an involution $h\in S$ such that 
$P=C_S(h)$, $[N_S(P):P]=2$ and $[h,Z_2(S)]\neq 1$. 

\item
$P\neq S$ and $C_S(P)\leq P$.

\item
Either $S_0$ is cyclic or $Z(S_0)=\Omega_1(S_0)$. 

\item If $Z(S_0)$ is not cyclic, then $\lvert S_0\rvert = \lvert 
Z(S_0)\rvert^m$ for $m=1$, $2$ or $3$. 

\item Let $k$ be such that $|S_0|=2^k$. Then $\rk(P/\Phi(P))\geq 2k$. If 
$k\geq 2$, then for all $1\ne s\in S_0$, $\rk([s,P/\Phi(P)])\geq 2$.

\item
If $Z(S_0)\cong (C_2)^n$ with $n\geq 2$ then $\rk([s,P/\Phi(P)])\geq n$
for all $s\in Z(S_0)$, $s\neq 1$. 

\item
Let $\Theta=1$, $\Theta=Z(P)$ or $\Theta=Z_2(P)$. If $g\in N_S(P)$
satisfies $[g,P]\leq \Theta\cdot\Phi(P)$ and $[g,\Theta]\leq\Phi(P)$
then $g\in P$. 

\item
$\Aut(P)$ is not a $2$-group and $\Out_S(P)\cap O_2(\Out(P))=1$.

\item Let $k$ be as in \textup{(e)}. 
There exists a composition factor $M$ of the $\F_2[\Out(P)]$-module
$P/\Phi(P)$ with $\dim(M)\geq 2k$. If $k\geq 2$, then $M$ can be
chosen such that $\dim([s,M])\geq 2$ for all $s\in S_0$, $s\ne1$. 

\item
All involutions in $S_0$ are conjugate in
$N_{\Out(P)}(S_0)$. 
\end{enuma}
\end{Prop}

\begin{proof} Points (b) and (h) hold by definition of critical subgroups.  
Point (a) follows from \cite[Lemma 3.6(a)]{OV2}, points (c) and (d) from 
Proposition \ref{p:critical}(a), (e) from Proposition \ref{p:critical}(b), 
(f) from Proposition \ref{p:critical}(c), (g) from Proposition 
\ref{p:QcharP}, and (j) from \cite[Proposition 3.3(b)]{OV2}.  Point (i) is 
shown in Proposition \ref{p:critical2}.
\end{proof}

\begin{Defi} \label{d:p.crit.}
A subgroup $P$ of a finite $2$-group $S$ is \emph{potentially critical} if 
it satisfies conditions \textup{(a)--(j)} in Proposition 
\ref{p:critical-criteria}.
\end{Defi}

By definition, each automorphism of $S$ sends the set of potentially 
critical subgroups of $S$ to itself.

\bigskip

\section{Fusion systems over $2$-groups with abelian direct factor}
\label{s:Q<|F}

In this section, we find conditions on a $2$-group $S_0$ which imply that 
there are no reduced fusion systems over $S=S_0\times A$ for any abelian 
2-group $A\ne1$. As will be seen, the assumption that a fusion system 
$\calf$ over $S$ has no normal subsystems of $2$-power index implies 
certain extra properties, which allow us to show that $A$, or some other 
direct factor in $S$, is normal in $\calf$.  Green correspondence plays a 
key role in the arguments we use, and the following elementary lemma will 
be useful.  We refer to \cite[\S\S\,3.10 \& 3.12]{Benson1} for more details 
on vertices and Green correspondence.

\begin{Lem} \label{vx-props}
Let $G$ be a finite group and fix $S\in\sylp{G}$. Let $k$ be a field of 
characteristic $p$, and let $V$ be a finitely generated $k[G]$-module. 
\begin{enuma} 

\item If $V$ is indecomposable, then for all $H\le G$, each vertex of each 
indecomposable direct summand of $V|_H$ is contained in a vertex of $V$. In 
particular, if $S\in\sylp{G}$ and $V|_S$ has a nonzero direct summand with 
trivial $S$-action, then $S$ is a vertex of $V$.

\item Let $H\le G$ be a subgroup that contains $N_G(S)$. Let 
$0\ne W\le V|_H$ be a direct summand of $V$ as a $k[H]$-module, and write 
$W=\bigoplus_{i=1}^nW_i$ where each $W_i$ is indecomposable as a 
$k[H]$-module. Assume $S$ is a vertex of $W_i$ for each $i$ (in particular, 
this holds if $S$ acts trivially on $W$). Let $V_i$ be the $k[G]$-Green 
correspondent of $W_i$. Then $V\cong \til{V} \oplus \bigoplus_{i=1}^n V_i$ 
for some $k[G]$-module $\til{V}$.

\item Let $H<G$ be as in (b), and assume in addition that $H$ is strongly 
$p$-embedded in $G$. Assume also that $V$ is indecomposable as a 
$k[G]$-module and has vertex $S$. Let $W$ be the $k[H]$-Green correspondent 
of $V$. Then $V|_H\cong W\oplus X$, where $X$ is projective as a 
$k[H]$-module and hence free as a $k[S]$-module. 

\end{enuma}
\end{Lem}

\begin{proof} \textbf{(a) } Let $P$ be a vertex of $V$.  Then there is a 
$k[P]$-module $W$ such that $V$ is a direct summand of $\Ind_P^G(W)$.  By 
the Mackey formula (cf. \cite[Theorem 3.3.4]{Benson1}), $(\Ind_P^G(W))|_H$ 
is a direct sum of modules induced up from $H\cap\9gP$ for elements 
$g\in{}G$.  Hence each indecomposable direct summand of $V|_H$ has vertex 
contained in $\9gP$ for some $g\in{}G$ (and $\9gP$ is also a 
vertex of $V$). 

If $S\in\sylp{G}$ and $V|_S$ has a nonzero indecomposable direct summand 
$W$ with trivial $S$-action, then $S$ is a vertex of $W$ (see, e.g., 
\cite[Remark 4.8.11(b)]{LP}), and hence also a vertex of $V$.

\smallskip

\noindent\textbf{(b) } Let $V=\bigoplus_{j\in J}\4V_j$ be the decomposition 
as a sum of indecomposable $k[G]$-modules. Let $J_0\subseteq J$ be the set 
of all $j\in J$ such that $\4V_j|_H$ has an indecomposable direct summand 
$U_j\le\4V_j$ with vertex $S$. By Green correspondence and (a), for each 
$j\in J_0$, $U_j$ is the $k[H]$-Green correspondent of $\4V_j$ as a 
$k[G]$-module, and is the only indecomposable direct summand of $\4V_j|_H$ with 
vertex $S$. 

By the Krull-Schmidt theorem applied to $V|_H$ (cf. \cite[Theorem 
1.4.6]{Benson1}), and since $W=\bigoplus_{i=1}^nW_i$ where each $W_i$ is 
indecomposable and has vertex $S$, there is an injective map 
$r\:\{1,2,\dots,n\}\Right2{}J_0$ such that $W_i\cong U_{r(i)}$ for each 
$i$. The lemma now follows upon setting $V_i=\4V_{r(i)}$ for $1\le i\le n$ 
and $\til{V}=\bigoplus_{j\in J\sminus\Im(r)}\4V_j$.

\smallskip

\noindent\textbf{(c) } By Green correspondence 
and since $V$ is indecomposable, $V|_H\cong W\oplus X$ 
where the vertex of each indecomposable direct summand of $X$ is contained in 
$\9gS\cap H$ for some $g\in G\sminus H$. Since $H$ is strongly 
$p$-embedded, this means that each such direct summand has vertex the 
trivial group, and 
hence that $X$ is projective as a $k[H]$-module. In particular, $X$ is free 
over $k[S]$ (see \cite[Theorem 5.24]{CR}). 
\end{proof}

We will need the following lemma on Green correspondance.

\begin{Lem} \label{l:Green.corr}
Fix a finite group $G$ with a strongly $2$-embedded subgroup $H<G$, and 
choose $S\in\syl2{H}\subseteq\syl2{G}$. Let $V$ be a finitely generated 
$\F_2[G]$-module, and assume $0\ne{}W\le{}V$ is a direct summand of $V|_H$ 
upon which $S$ acts trivially. Then there is an $\F_2[G]$-submodule $V^*\le V$
such that $V^*|_H\cong W\oplus X$ as 
$\F_2[H]$-modules, where $X$ is free as an $\F_2[S]$-module, and where 
either $X=0$, or $|S|=2$ and $\dim_{\F_2}(X)\ge4$, or $\dim_{\F_2}(X)\ge8$.
\end{Lem}

\begin{proof} By Lemma \ref{vx-props}(b), there is an $\F_2[G]$-submodule 
$V^*\le V$ (in fact, a direct summand) which is isomorphic to the sum of 
the Green correspondents of the indecomposable direct summands of $W$. So without 
loss of generality, we can assume that $V=V^*$ and $W$ are indecomposable 
as $\F_2[G]$- and $\F_2[H]$-modules, respectively, $S$ is a vertex of both, 
and $V$ is the Green correspondent of $W$. By Lemma \ref{vx-props}(c), 
$V|_H\cong W\oplus X$ for some $X$ that is free as an $\F_2[S]$-module.

Assume $X\ne0$. If the action of $H$ on $W$ could be extended to a linear 
action of $G$, then that would have vertex $S$ (Lemma \ref{vx-props}(a)) 
and hence be the Green correspondent of $W$, contradicting the uniqueness 
of Green correspondents. So the action of $H$ on $W$ does not extend to any 
linear action of $G$, and in particular, $W$ is not normalized by $G$. 

Set $K=O^{2'}(H)$: the normal closure of $S$ in $H$. Thus $H/K$ has odd 
order, and $W$ can be regarded as an $\F_2[H/K]$-module.

Set $N_0=C_G(V)\nsg G$: the kernel of the $G$-action on $V$. Since $X\ne0$ 
is free as an $\F_2[S]$-module, $N_0\cap{}S=1$, and so $N_0$ has odd order. 
Also, $HN_0<G$ since the action of $HN_0$ on $V$ normalizes $W$ while the 
action of $G$ on $V$ does not. So $HN_0/N_0$ is strongly $2$-embedded in 
$G/N_0$ by Lemma \ref{str.emb.->>}(b). Upon replacing $G$ by $G/N_0$ and 
$H$ by $HN_0/N_0$, we can assume that $G$ acts faithfully on $V$.

Let $k\supseteq\F_2$ be any finite extension, and set 
$\5V=k\otimes_{\F_2}V$ and $\5W=k\otimes_{\F_2}W$. Write 
$\5W=\bigoplus_{i=1}^m\5W_i$, where $\5W_i$ is irreducible as a 
$k[H/K]$-module for each $i$. Let $\5V_i$ denote the $k[G]$-Green 
correspondent of $\5W_i$. By Lemma \ref{vx-props}(c), $\5V_i|_H\cong 
\5W_i\oplus \5X_i$ where $\5X_i$ is free as a $k[S]$-module. Also, $\5V$ 
has a direct summand isomorphic to $\bigoplus_{i=1}^m\5V_i$ by Lemma 
\ref{vx-props}(b). By the Krull-Schmidt theorem, and since 
$\5W|_{\F_2[H/K]}\cong W^\ell$ as $\F_2[H/K]$-modules where 
$\ell=\dim_{\F_2}(k)$, $\5W_i|_{\F_2[H/K]}\cong W^{\ell_i}$ for some 
$1\le\ell_i\le\ell$. If, for some $i$, $\5W_i\cong \5V_i|_H$, then by Lemma 
\ref{vx-props}(b) applied to $\5V_i|_{\F_2[G]}$, the same must be true for 
each irreducible direct summand of $\5W_i|_{\F_2[H/K]}$, which contradicts 
our assumption that $V>W$. Thus 
$\dim_k(\5V_i)-\dim_k(\5W_i)=\dim_k(\5X_i)\ge|S|$ for each $i$, and so 
$\dim_{\F_2}(X)=\dim_{\F_2}(V/W)\ge\sum_{i=1}^m\dim_k(\5X_i)\ge m|S|$. This 
proves the lemma when $m\ge2$ for some finite extension $k\supseteq\F_2$.

We are left with the case where $V$ is a faithful, indecomposable 
$\F_2[G]$-module and $W$ is absolutely irreducible as an $\F_2[H/K]$-module 
(i.e., $k\otimes_{\F_2}W$ is irreducible as a $k[H/K]$-module for each 
finite extension $k\supseteq\F_2$, see \cite[Theorem 9.2]{Isaacs}).  There 
are two cases to consider.

\noindent\boldd{Case 1: \ $O_{2'}(G)\ne1$.}  Since $O_{2'}(G)$ is 
solvable by the odd order theorem \cite{FT}, there is $1\ne{}N\nsg G$ which 
is an elementary abelian $p$-group for some odd $p$.  By the Frattini 
argument, $N_G(SN)=N_G(S)\cdot{}N\le HN$, so $HN/N\ge N_{G/N}(SN/N)$.  We 
claim that $H\cap{}N\ne1$ and acts nontrivially on $W$.  Assume otherwise.  
Regard $W$ as an $\F_2[HN/N]$-module via the isomorphism $HN/N\cong 
H/(H\cap{}N)$.  Since $SN/N$ acts trivially on $W$, it is a vertex of $W$ 
by Lemma \ref{vx-props}(a). Let $V'$ be its $\F_2[G/N]$-Green correspondent 
(note that $HN/N\ge N_{G/N}(SN/N)$).  Then $V'$ is indecomposable as an 
$\F_2[G]$-module, and $W$ is a direct summand of $V'|_H$ with vertex $S$. 
Hence as an $\F_2[G]$-module, $V'$ has vertex $S$ by Lemma 
\ref{vx-props}(a) 
and is the Green correspondent of $W$ by \cite[Theorem 
3.12.2(i)]{Benson1}, so $V'\cong V$ by the uniqueness of Green 
correspondents.  This is impossible, since $G$ acts faithfully on $V$ but 
not on $V'$, and we conclude that $H\cap{}N$ acts nontrivially.  In 
particular, $H\cap{}N\ne1$.

Let $r>1$ be the multiplicative order of $2$ in $\F_p^\times$. Thus all 
irreducible $\F_2[C_p]$-modules with nontrivial action have dimension $r$.  
Since $V$ is indecomposable, and $V=C_V(N)\oplus[N,V]$ where the 
summands are $\F_2[G]$-submodules, $V|_N$ is a direct sum of irreducible 
$\F_2[N]$-modules with nontrivial action, 
each of which has dimension $r$.  Similarly, since $W$ is irreducible and 
$H\cap{}N\nsg H$ acts nontrivially, each irreducible direct summand of 
$W|_{H\cap{}N}$ has dimension $r$.  Thus $r|\dim(V)$, $r|\dim(W)$, and 
$r|\dim(X)$.  

The dimension of each absolutely irreducible $\F_2[H/K]$-module divides 
$|H/K|$ (cf. \cite[\S\S\,6.5 \& 15.5]{Serre}). In particular, $\dim(W)$ is 
odd.  Thus $r|\dim(W)$ is also odd, so $r\ge3$, and $r|S|\big|\dim(X)$ in 
this case.

\noindent\boldd{Case 2: \ $O_{2'}(G)=1$.}  By Bender's theorem 
\cite[Satz 1 \& Lemma 2.6]{Bender}, $O^{2'}(G)$ is isomorphic to 
$\PSL_2(2^n)$, $\Sz(2^{2n-1})$, or $\PSU_3(2^n)$ ($n\ge2$), and hence 
$Z(S)$ is elementary abelian of rank at least $2$.  If $\dim(X)<8$, then 
$|S|<8$ since $X$ is free as an $\F_2[S]$-module, and so $S\cong 
\klfour$, and $G\cong A_5$ since $|\Out(A_5)|=2$ \cite[(3.2.17)]{Sz1}.  But 
then $H\cong A_4$, $H/K\cong{}C_3$, and $W$ has nontrivial $H$-action since 
the Green correspondent of the trivial $\F_2[A_4]$-module is trivial.  So 
$\F_4\otimes_{\F_2}W$ is reducible, which contradicts 
our assumption that $W$ is absolutely irreducible.  
\end{proof}

We now prove four propositions, each showing that under certain (fairly 
restrictive) hypotheses on a 2-group $S_0$, there are no reduced fusion 
systems over $S_0\times A$ for any abelian 2-group $A\ne1$.  Stronger 
results of this type will be shown in a later paper. 

In the first proposition, we consider certain finite $2$-groups of 
nilpotence class $2$. It will be applied, for example, when $S=S_0\times A$ 
for $S_0$ of type $\SL_3(2^n)$ ($n\ge2$), $2\cdot\SL_3(4)$, or $\Sp_4(2^n)$ 
($n\ge2$), and $A\ne1$ is abelian.

\begin{Prop} \label{p:AxE1E2}
Let $S$ be a finite $2$-group containing normal abelian subgroups 
$P_1,P_2\nsg{}S$ such that 
\begin{enumi} 
\item $S=P_1P_2$ and $P_1\cap{}P_2=Z(S)$;
\item $[S,S]\cap\Phi(Z(S))=1$; 
\item $P_i=Z(S)\Omega_1(P_i)$ and $\rk(P_i/Z(S))\ge2$ for $i=1,2$; and 
\item for each $i=1,2$ and each $g\in{}P_i{\sminus}Z(S)$, $C_{S}(g)=P_i$.  
\end{enumi}
Then $P_1$ and $P_2$ are the only possible critical subgroups of $S$, and 
every elementary abelian subgroup of $S$ is contained in $P_1$ or $P_2$. If 
$\calf$ is a saturated fusion system over $S$ such that $O_2(\calf)=1$, 
then $\EE_\calf=\{P_1,P_2\}$, $P_1$ and $P_2$ are elementary abelian, 
$\rk(P_1)=\rk(P_2)$, and $[S,S]=Z(S)$.
\end{Prop}

\begin{proof}  Set $V_i=\Omega_1(P_i)$ for short ($i=1,2$). Thus 
$V_i\nsg{}S$ and is elementary abelian, and $P_i=V_iZ(S)$ by (iii).  Also, 
$S=V_1V_2Z(S)$, so $[S,S]\le V_1\cap{}V_2$ is elementary abelian.

If $g\in{}S{\sminus}(P_1\cup{}P_2)$, then $g=g_1g_2z$ for some 
$g_i\in{}V_i{\sminus}Z(S)$ and $z\in{}Z(S)$, so $g^2=[g_1,g_2]z^2\ne1$ 
since $[g_1,g_2]\notin\Phi(Z(S))$ by (ii) and (iv), while 
$z^2\in\Phi(Z(S))$. Thus all elements of order two in $S$ lie in 
$P_1\cup{}P_2$. 

If $Q\le S$ is elementary abelian, and $Q\nleq Z(S)=P_1\cap{}P_2$, then 
choose $g\in{}Q{\sminus}Z(S)$. We just showed that $g\in P_i$ for $i=1$ or 
$2$, and so $Q\le C_S(g)=P_i$.

Let $R\le{}S$ be a critical subgroup.  If $R\notin\{P_1,P_2\}$, then 
$R\nleq{}P_1$ and $R\nleq{}P_2$ since $R$ is centric in $S$.  If 
$x\in\Omega_1(Z(R))$, then $x\in{}P_i$ for some $i$ as shown above.  If 
$x\notin{}Z(S)$, then $R\le C_S(x)=P_i$, which is impossible.  Thus 
$\Omega_1(Z(R))\le Z(S)$, so $\Omega_1(Z(S))=\Omega_1(Z(R))$ is 
characteristic in $R$.  But for $g\in{}S{\sminus}R$, $[g,\Omega_1(Z(S))]=1$ 
and $[g,R]\le[S,S]\le V_1\cap V_2=\Omega_1(Z(S))$, which contradicts 
Proposition \ref{p:QcharP}.  We conclude that $P_1$ and $P_2$ are the only 
subgroups of $S$ which could be critical. 

\smallskip

\noindent\textbf{Step 1: } 
Let $\calf$ be a saturated fusion system over $S$ such that $O_2(\calf)=1$.  
Then $P_1$ and $P_2$ must both be $\calf$-essential by Lemma \ref{l:1crit}.  
Since $P_i=V_iZ(S)$ where $V_i$ is elementary abelian, 
$\Phi(P_1)=\Phi(Z(S))=\Phi(P_2)$.  So $\Phi(Z(S))\nsg\calf$ by Proposition 
\ref{Q<|F}, and since $O_2(\calf)=1$, $\Phi(Z(S))=1$.  Thus $P_i=V_i$ is 
elementary abelian for $i=1,2$.

For $x\in{}P_1{\sminus}Z(S)$, we have 
	$\rk(P_1/Z(S)) = \rk(S/P_2) \le \rk([x,P_2]) = \rk(P_2/Z(S))$, 
where the inequality holds by Proposition \ref{p:critical}(c), and the last 
equality since $C_{P_2}(x)=P_2\cap{}P_1=Z(S)$ by (iv) and (i).  The opposite 
inequality holds by a similar argument, so $\rk(P_1/Z(S))=\rk(P_2/Z(S))$ 
and $\rk(P_1)=\rk(P_2)$. 

\smallskip

\noindent\textbf{Step 2: }  It remains to prove that $Z(S)=[S,S]$. By 
(i), $[S,S]\le P_1\cap P_2=Z(S)$, and it remains to prove the 
opposite inclusion. Regard $Z(S)$ as an $\F_2[\outf(S)]$-module with 
submodule $[S,S]$.  Since $|\outf(S)|$ is odd, there is an 
$\F_2[\outf(S)]$-submodule $W\le Z(S)$ which is complementary to $[S,S]$; 
i.e., $Z(S)=W\times[S,S]$.  

For each $i=1,2$, set $G_i=\autf(P_i)$, $T_i=\Aut_S(P_i)\in\syl2{G_i}$, and 
$H_i=N_{G_i}(T_i)$. Then 
        \beqq T_i\cong S/P_i\cong P_{3-i}/Z(S)\cong (C_2)^r 
        \quad\textup{where}\quad
        r = \rk(P_1/Z(S)) = \rk(P_2/Z(S)) \ge2. \label{e:r=rk} \eeqq
Since $\autf(S)$ is generated by $\Inn(S)$ and 
automorphisms of odd order, each $P_i$ is normalized by $\autf(S)$.  By the 
extension axiom, the homomorphism 
	\[ \psi_i\: \autf(S)\Right4{}H_i\,, \] 
induced by restriction to $P_i$, is surjective. In particular, a subgroup 
of $P_i$ is normalized by $H_i$ if and only if it is normalized by 
$\autf(S)$.

Consider the $\F_2[\outf(S)]$-module $P_i/[S,S]$. Since $|\outf(S)|$ is odd, 
$Z(S)/[S,S]$ has a complement $M/[S,S]$ in $P_i/[S,S]$. Hence 
$P_i=W\times M$ as $\F_2[\autf(S)]$-modules, and in particular,
$W$ is a direct factor of $P_i$ as an $\F_2[H_i]$-module.  

Since $P_i$ is maximal among $\calf$-essential subgroups, each 
$\alpha\in\autf(P_i)$ which extends to $\autf(Q)$ for any $Q>P_i$ also 
extends to $\autf(S)$ (Proposition \ref{p:AFT-E}), and hence lies in 
$H_i=N_{G_i}(T_i)$.  
So by \cite[Proposition I.3.3(b)]{AKO} and since $P_i\in\EE_\calf$, $H_i$ 
is strongly $2$-embedded in $G_i$.

If $W\ne1$ and $U$ is an indecomposable direct factor of $W$ as an 
$\F_2[H_i]$-module, then by Lemma \ref{vx-props}(b,c), the $G_i$-Green 
correspondent $U^*$ of $U$ is isomorphic to a direct factor of $P_i$, and 
$U^*|_{H_i}\cong U\times X$ for some $\F_2[H_i]$-module $X$ such that 
$X|_{T_i}$ is free as an $\F_2[T_i]$-module. Since $T_i$ acts trivially on 
$Z(S)$ and on $P_i/Z(S)$, $P_i|_{T_i}$ contains no nontrivial free 
$\F_2[T_i]$-submodules (recall that $T_i\cong (C_2)^r$ for $r\ge2$ by 
\eqref{e:r=rk}). So $X=1$, and $U\cong U^*|_{H_i}$. 
 
Thus the Green correspondent of each indecomposable direct factor of $W$ is 
isomorphic to that direct factor (after restriction to $H_i$). So by Lemma 
\ref{vx-props}(b), there is $W_i\le P_i$ which is a direct factor of $P_i$ 
as an $\F_2[G_i]$-module and such that $W_i|_{H_i}\cong W$. Also, $W_i\le 
C_{P_i}(T_i)=Z(S)$, $W_i\cap[S,S]=W_i\cap[T_i,P_i]=1$ since $W_i$ is a 
direct factor as an $\F_2[G_i]$-module, and so 
$Z(S)=C_{P_i}(T_i)=[S,S]\times W_i$. 

\smallskip

\noindent\textbf{Step 3: } Fix $i=1,2$. By the version of Bender's theorem 
in \cite[Theorem 6.4.2]{Sz2}, $O_{2'}(G_i)$ is contained in every strongly 
$2$-embedded subgroup of $G_i$, and in particular, $O_{2'}(G_i)\le 
H_i=N_{G_i}(T_i)$. So $[O_{2'}(G_i),T_i]\le O_{2'}(G_i)\cap T_i=1$. Upon 
setting $\5G_i=O^{2'}(G_i)$, we get $[\5G_i,O_{2'}(G_i)]=1$, and thus 
$O_{2'}(\5G_i)\le Z(\5G_i)$. 

By Bender's theorem \cite[Satz 1]{Bender}, 
$\5G_i/O_{2'}(\5G_i)\cong\SL_2(2^r)$ where $r\ge2$ is as in \eqref{e:r=rk}. 
Then $\5G_i\cong\SL_2(2^r)$, since by \cite[p. 119, Satz IX]{Schur}, the 
Schur multiplier of $\SL_2(2^r)$ has order $2$ (if $r=2$) or $1$ (if 
$r\ge3$). Hence there are $T_i^*\in\syl2{G_i}$ such that 
$\gen{T_i,T_i^*}=\5G_i$ and $D_i\defeq N_{\5G_i}(T_i)\cap 
N_{\5G_i}(T_i^*)\cong\F_{2^r}^\times$. (Identify $\5G_i$ with $\SL_2(2^r)$ 
in such a way that $T_i$ is the group of upper triangular matrices with 
1's on the diagonal, and let $T_i^*$ be the group of lower triangular 
matrices with 1's on the diagonal.) Then $C_{P_i}(\5G_i)=C_{P_i}(T_i)\cap 
C_{P_i}(T_i^*)$ has index at most $2^{2r}$ in $P_i$. Since each faithful 
$\F_2[\5G_i]$-module has dimension at least $2r$ (see \cite[Lemma 
1.7(a)]{OV2}), $P_i/C_{P_i}(\5G_i)$ is $2r$-dimensional and irreducible. 

Set $V_i=P_i/C_{P_i}(\5G_i)$, regarded as an $\F_2[\5G_i]$-module, and set 
$K_i=\End_{\F_2[\5G_i]}(V_i)$. Thus $K_i$ is a finite field extension of $\F_2$, 
and $\4\F_2\otimes_{K_i}V_i$ is irreducible as an $\4\F_2[\5G_i]$-module (see, 
e.g., \cite[Theorem 26.6.4]{A-FGT}). By a theorem of Curtis (see 
\cite[Theorem 2.8.9]{GLS3}), $\dim_{K_i}(C_{V_i}(T_i))=1$, so 
$[K_i:\F_2]=\dim_{\F_2}(C_{V_i}(T_i))\ge r$. Thus 
$\dim_{K_i}(V_i)=2r/[K_i:\F_2]=2$ since $T_i$ acts nontrivially on $V_i$ 
(since its normal closure $\5G_i$ acts nontrivially), and 
$V_i$ is the natural $\F_2[\SL_2(2^r)]$-module (see \cite[Example 
2.8.10.b]{GLS3}). In particular, $C_{V_i}(D_i)=1$, 
so $C_{P_i}(D_i)=C_{P_i}(\5G_i)$, and 
$C_{P_i}(\5H_i)=C_{P_i}(\5G_i)$, where $\5H_i=N_{\5G_i}(T_i)=T_iD_i$.

Set $H_i^*=\psi_i^{-1}(\5H_i)\nsg\autf(S)$ (see Step 2). Set 
$Q=C_S(H_1^*H_2^*)$, so that 
	\[ Q = C_S(H_1^*)\cap C_S(H_2^*) = C_{P_1}(\5H_1) 
	\cap C_{P_2}(\5H_2) = C_{P_1}(\5G_1)\cap C_{P_2}(\5G_2)\,. \]
Then $Q$ is normalized by $\autf(S)$ since $H_1^*H_2^*\nsg\autf(S)$. For 
$i=1,2$, $Q$ is normalized by $H_i$ since $\psi_i$ is onto (Step 2), and 
$Q$ is normalized by $\5G_i$ since $Q\le C_{P_i}(\5G_i)$. Since 
$G_i=\5G_iN_{G_i}(T_i)=\5G_iH_i$ by the Frattini argument, $Q$ is 
normalized by $G_i=\autf(P_i)$. 
So $Q\nsg\calf$ by Proposition \ref{Q<|F} and since $\EE_\calf=\{P_1,P_2\}$, 
and hence $Q\le O_2(\calf)=1$. 

Set $K=C_{\autf(S)}(Z(S)/[S,S])$. Then $K\ge H_1^*H_2^*$ since, by Step 2, 
for $i=1,2$, $\5G_i$ (hence $H_i^*$) acts trivially on $W_i$ and 
$Z(S)=[S,S]W_i$. So $Q[S,S]\ge C_S(K)[S,S]\ge Z(S)$, and thus 
$[S,S]=Z(S)$ since $Q=1$. 
\end{proof}

\begin{Lem} \label{l:S0*xA*}
Let $S_0$ be a finite nonabelian $2$-group such that $Z(S_0)$ is cyclic. 
Let $A$ be a finite abelian $2$-group, set $S=S_0\times A$, and 
assume $\calf$ is a reduced fusion system over $S$. Set 
$A^*=[\autf(S),Z(S)]$; thus $\autf(S)$ normalizes $A^*$ and 
$C_{A^*}(\autf(S))=1$. Then $A^*\cong A$, $S=S_0\times A^*$, and 
$\Omega_1(A^*)=[\autf(S),\Omega_1(Z(S))]$. 
\end{Lem}

\begin{proof} Set $Z=Z(S)$, $Z_0=Z(S_0)$, and $\Gamma=\outf(S)$ for short. 
Thus $\Gamma$ has odd order and acts on $Z$, so
	\begin{align} 
	Z_0\times A = Z &= C_Z(\Gamma)\times[\Gamma,Z] = C_Z(\Gamma)\times 
	A^* \label{e:Z-decomp} \\
	\Omega_1(Z) &= C_{\Omega_1(Z)}(\Gamma)\times [\Gamma,\Omega_1(Z)] 
	\label{e:Om1(Z)-decomp}
	\end{align}
(see \cite[Theorem 5.2.3]{Gorenstein}, and note in particular that 
$C_{A^*}(\Gamma)=1$). Since $Z_0$ is cyclic, 
$|\Omega_1(Z_0)|=2$; and since 
each nontrivial normal subgroup of $S_0$ intersects nontrivially with 
$Z_0$, $\Omega_1(Z_0)\le[S_0,S_0]=[S,S]$. By \cite[Corollary I.8.5]{AKO}, 
and since $\foc(\calf)=S$ by Lemma \ref{l:foc=S}
(recall $O^2(\calf)=\calf$), $\Gamma$ acts with trivial centralizer on 
	\beq \Omega_1(Z)\big/ \bigl( \Omega_1(Z)\cap[S,S]\bigr) 
	= \Omega_1(Z)/\Omega_1(Z_0)\,. \eeq
Hence $\Omega_1(Z_0)=C_{\Omega_1(Z_0)}(\Gamma)=C_{\Omega_1(Z)}(\Gamma)=\Omega_1(C_Z(\Gamma))$. 
Since $[\Gamma,\Omega_1(Z)]\le\Omega_1(A^*)$, and both are complements to 
$\Omega_1(Z_0)$ in $\Omega_1(Z)$ by \eqref{e:Z-decomp} and 
\eqref{e:Om1(Z)-decomp}, we have $\Omega_1(A^*)=[\Gamma,\Omega_1(Z)]$. 

Since $\Omega_1(C_Z(\Gamma))=\Omega_1(Z_0)$, and $C_Z(\Gamma)\cap 
A^*=1=Z_0\cap A$ by \eqref{e:Z-decomp}, we also have $C_Z(\Gamma)\cap 
A=1=Z_0\cap A^*$ since neither intersection has any elements of order $2$.  
Thus $|A|=|A^*|$, and  hence
$Z=Z_0\times A^*$ and $A\cong Z/Z_0\cong A^*$. Also, 
$S_0\cap A^*=1$, so $S=S_0\times A^*$. 
\end{proof}

Recall that the \emph{rank} of a finite $p$-group is the largest rank of 
any of its abelian subgroups.

\begin{Lem} \label{l:SxA}
Let $S_0$ be a finite nonabelian $2$-group such that $Z(S_0)$ is cyclic. 
Let $A\ne1$ be a nontrivial finite abelian $2$-group, set $S=S_0\times A$, 
and assume $\calf$ is a reduced fusion system over $S$. Then for each $1\ne 
\5W\le \Omega_1(Z(S))$ which is normalized by $\autf(S)$, there is an 
$\calf$-essential subgroup $P\le S$ such that
\begin{enuma} 
\item $P=C_S(\Omega_1(Z(P)))$ and $\Omega_1(Z(P))$ is fully normalized in 
$\calf$;
\item $\autf(P)(\5W)\ne \5W$; and 
\item for all $Q\le S$ such that $|Q|>|P|$, and all $\beta\in\homf(Q,S)$, 
$\beta(\5W)=\5W$.
\end{enuma}
Set $W=[\autf(S),\Omega_1(Z(S))]\ne1$, and let $P$ be 
a subgroup satisfying {\rm(a)--(c)} with $\5W=W$. Set $V=\Omega_1(Z(P))$. 
\begin{enuma}\setcounter{enumi}{3}
\item If there is $\Gamma\le\autf(S)$ such that $C_W(\Gamma)=1$ 
and $\Gamma(P)=P$, then either \smallskip
\begin{enumerate}[\rm(i) ] 
\item $\rk(P/A)\ge6$, or
\item $\rk(C_V(N_S(P))/W)\ge3$ and $\Gamma$ acts nontrivially on $V/W$.
\end{enumerate}
\end{enuma}
\end{Lem}

\begin{proof} 
\noindent\textbf{(a,b,c) } Since $\calf$ is reduced, $\5W$ cannot be normal 
in $\calf$. Since $\5W\le Z(S)$, it is contained in every $\calf$-essential 
subgroup. Hence by Proposition \ref{Q<|F} (and since $\autf(S)$ normalizes 
$\5W$), there is some $P\in\EE_\calf$ and some $\alpha\in\autf(P)$ such 
that $\alpha(\5W)\ne{}\5W$. Choose $P$ to have maximal order among all such 
$\calf$-essential subgroups. Then (c) holds by Proposition \ref{p:AFT-E}. 

Set $V=\Omega_1(Z(P))$. Thus $\5W\le\Omega_1(Z(S))\le V$. If $V$ is not fully normalized, then 
there is $\varphi\in\homf(N_S(V),S)$ such that $\varphi(V)$ is fully 
normalized (cf. \cite[Lemma I.2.6(c)]{AKO}).  Also, $N_S(V)\ge N_S(P)>P$, so 
$\varphi(\5W)=\5W$ by (c), and $\varphi(P)$ is fully normalized (hence 
$\calf$-essential) since $P$ is. Upon replacing $P$ by $\varphi(P)$, we can 
assume $V$ is fully normalized (and hence fully centralized).  

Let $\alpha\in\autf(P)$ be as above: $\alpha(\5W)\ne \5W$.  Since $V$ is fully 
centralized, $\alpha|_V\in\autf(V)$ extends to some 
$\widebar{\alpha}\in\autf(C_S(V))$.  If $C_S(V)>P$, then 
$\overline\alpha(\5W)=\5W$ by (c), contradicting our assumption on 
$\alpha$.  Hence $C_S(V)=P$, and in particular, $N_S(V)=N_S(P)$.

\smallskip

\noindent\textbf{(d) } By Lemma \ref{l:S0*xA*}, we can assume that 
$A=[\autf(S),Z(S)]$ and hence is normalized by $\autf(S)$, and that 
$W=\Omega_1(A)\ne1$. Fix $V\le P\in\EE_\calf$ which satisfy conditions 
(a)--(c) for $\5W=W$. Let $\Gamma\le\autf(S)$ be such that $C_W(\Gamma)=1$ 
and $\Gamma(P)=P$. Assume $\rk(V/W)\le5$ (otherwise (i) holds). 

Set $G=\outf(P)$ and $T=\Out_S(P)\in\syl2{G}$, and let $H\le G$ be the 
subgroup generated by the classes of all $\beta\in\autf(P)$ which extend to 
$\calf$-morphisms between subgroups strictly containing $P$. By 
\cite[Proposition I.3.3(b)]{AKO} and since $P\in\EE_\calf$, $H<G$ and is 
strongly $2$-embedded in $G$. Since $H\ge T$ by definition, $H\ge N_G(T)$.

Regard $V=\Omega_1(Z(P))$ as an $\F_2[G]$-module. By (c), $W$ is an 
$\F_2[H]$-submodule of $V|_H$. Set $V_0=V\cap S_0$. Then $V=V_0\times W$ 
since $P=(P\cap S_0)\times A$, and $V_0$ is acted upon by 
$T=\Out_{S_0}(P)$. Hence $W$ is a direct factor of $V|_H$ as 
$\F_2[H]$-modules, since the $\F_2[T]$-linear projection $V\too W$ 
(with kernel $V_0$) can be made $\F_2[H]$-linear by averaging over the 
cosets in $H/T$. 

We are thus in the situation of Lemma \ref{l:Green.corr}. By that lemma, 
there is an $\F_2[G]$-submodule $V^*\le V$ such that $V^*|_H\cong 
W\times X$ for some $\F_2[H]$-module $X$ that is free as an 
$\F_2[T]$-module, and such that $X=1$, or $|T|=2$ and $\rk(X)\ge4$, or 
$\rk(X)\ge8$. Since $\rk(X)\le\rk(V/W)\le5$, the last condition is 
impossible.

Assume $X=1$ (i.e., $V^*|_H\cong W$); we show that (ii) holds. Since 
$\Gamma(P)=P$, we have $[\gamma|_P]\in N_H(T)$ for each $\gamma\in\Gamma$, 
and hence the action of $\Gamma$ on $V$ factors through a subgroup 
$\4\Gamma\le N_H(T)$. Then $V^*$ and $W$ are isomorphic as 
$\F_2[\4\Gamma]$-modules. If (ii) does not hold, then either $\4\Gamma$ 
acts trivially on $V/W$ or $\rk(C_V(T)/W)\le2$. In the latter case, 
$\4\Gamma$ acts trivially on $C_V(T)/W$ since it normalizes the subspace 
$\Omega_1(Z(S))/W\le C_V(T)/W$ of rank $1$, and since the action of 
$\4\Gamma$ factors through $\Gamma^*=\4\Gamma T/T\le N_H(T)/T$ of odd 
order. So in either case, $\Gamma$ acts trivially on $C_V(T)/W$. Also, 
$C_V(T)=V_1\times V_2$, where $V_1=C_{C_V(T)}(\Gamma^*)$ and 
$V_2=[\Gamma^*,C_V(T)]$ (cf. \cite[Theorem 5.2.3]{Gorenstein}). Then $W\ge 
V_2$ since $C_V(T)/W$ has trivial action, and $W=V_2$ since 
$C_W(\Gamma)=1$. Also, $V^*\le C_V(T)$ and $V^*\cap 
V_1=C_{V^*}(\4\Gamma)=1$ since $V^*\cong W$ as $\F_2[\4\Gamma]$-modules, so 
$V^*=V_2=W$ since they have the same rank. This is impossible, since 
$G(W)\ne W$ by (b) while $G(V^*)=V^*$. 

Now assume $\rk(X)\ge4$ and $|T|=|N_S(P)/P|=2$. Recall that we set 
$V_0=V\cap{}S_0$, so that $V=V_0\times{}W$ as $\F_2[T]$-modules. Thus 
$\rk(V_0)=\rk(V/W)\ge\rk(X)$, and so 
$4\le\rk(V_0)\le5$. We have $V^*|_H\cong W\times X$ where 
$X|_T\cong\F_2[T]^k$ for some $k$, and $k=2$ by the assumptions on ranks. 
Hence $\rk(V^*/C_{V^*}(T))=2$, and so 
$\rk(V_0/C_{V_0}(T))=\rk(V/C_V(T))\ge2$. Since $|T|=2$, 
	\[ 2 \le \rk(V_0/C_{V_0}(T)) = \rk([T,V_0]) \le 
	\rk(C_{V_0}(T)). \] 
If $P\nsg S$, i.e., if $|S/P|=2$, then 
$\Omega_1(Z(S_0))=C_{V_0}(T)$ has rank at least $2$, which is impossible 
since $Z(S_0)$ is assumed to be cyclic.

Thus $P\nnsg S$. Fix $x\in N_S(P)\sminus P$. Let $\til{P}\ne{}P$ be another 
subgroup $S$-conjugate to $P$ with the same normalizer $P\gen{x}$, and set 
$\til{V}=\Omega_1(Z(\til{P}))$ and $\til{V}_0=\til{V}\cap{}S_0$.  Then 
$V\cap{}\til{V}=\Omega_1(Z(P\gen{x}))=C_V(x)$, so 
$V_0\cap{}\til{V}_0=C_{V_0}(x)=C_{V_0}(T)$, and we just saw that 
$\rk(V_0/C_{V_0}(T))\ge2$.  If $\til{V}\le{}P$, then 
$[V,\til{V}]\le[V,P]=1$ and $\rk(V_0\til{V}_0)
=\rk(\til{V}_0)+\rk(V_0/C_{V_0}(T))\ge4+2=6$, 
so (i) holds.  If $\til{V}\nleq{}P$, then we can assume $x$ was chosen so 
that $x\in{}\til{V}$.  But then $[\til{P},x]=1$ and hence 
$V\cap{}\til{P}\le C_V(x)$, which is impossible since 
$|V/(V\cap\til{P})|\le|P\gen{x}/\til{P}|=2$ and $|V/C_V(x)|\ge4$. 
\end{proof}

\begin{Prop} \label{p:AxS-rk5}
Let $S_0$ be a finite $2$-group such that $\rk(S_0)\le5$, $Z(S_0)$ is 
cyclic, and $\Aut(S_0)$ is a $2$-group. Then for any finite abelian 
$2$-group $A\ne1$, there are no reduced fusion systems over $S_0\times A$.
\end{Prop}

\begin{proof} We can assume $S_0$ is nonabelian; otherwise the result is 
clear.  Set $S=S_0\times A$, and assume $\calf$ is a reduced fusion system 
over $S$.  Using Lemma \ref{l:S0*xA*}, we can assume that 
$A=[\autf(S),Z(S)]$ and hence is 
normalized by $\autf(S)$. Set $W=[\autf(S),\Omega_1(Z(S))]$; then 
$W=\Omega_1(A)$ and $C_W(\autf(S))=1$ by Lemma \ref{l:S0*xA*} again.

Fix $\Gamma\le\autf(S)$ of odd order such that $\Inn(S)\Gamma=\autf(S)$ 
(by the Schur-Zassenhaus theorem, cf. \cite[Theorem 6.2.1(i)]{Gorenstein}). 
Then $C_W(\Gamma)=C_W(\autf(S))=1$ since $[\Inn(S),W]=1$. For 
each $\gamma\in\Gamma$, $\gamma(A)=A$, so $\gamma$ induces an automorphism 
of $S/A\cong S_0$ which must be the identity since $\Aut(S_0)$ is a 
$2$-group by assumption. Hence $\gamma(P)=P$ for each $P\le S$ which 
contains $A$, and in particular, for each $P\in\EE_\calf$. Also, for each 
such $P$, $\Gamma$ acts trivially on $Z(P)/A$, and hence on 
$\Omega_1(Z(P))/W$.

Lemma \ref{l:SxA}(d) now implies that $\rk(S_0)\ge6$, which contradicts our 
assumptions.
\end{proof}

We now give two applications of Lemma \ref{l:SxA} in situations where 
$\Aut(S_0)$ is not a 2-group.  Recall that $\UT_n(q)\le\SL_n(q)$ denotes 
the subgroup of upper triangular matrices over $\F_q$ with 1's on the 
diagonal.

\begin{Prop} \label{2^k.J2}
Let $S_0$ be a Sylow $2$-subgroup of $J_2$.  Then for any finite abelian 
$2$-group $A\ne1$, there are no reduced fusion systems over $S_0\times A$.
\end{Prop}

\begin{proof}  We adopt the notation used in \cite[\S\,4]{OV2} and 
\cite[\S\,6]{O-rk4}. Let $(a\mapsto\4a)$ denote the field involution of 
$\F_4$ (thus $\4a=a^2$). We identify $S_0=\UT_3(4)\gen{\theta}$, where 
$\theta\mxthree1ab01c001\theta^{-1}=\mxthree1{\4c}{\4b}01{\4a}001^{-1}$ for each 
$a,b,c\in\F_4$, and $\theta^2=1$. For $1\le i<j\le3$ and 
$x\in\F_4$, let $e_{ij}^x\in\UT_3(4)$ be the elementary matrix with unique 
nonzero off-diagonal entry $x$ in position $(i,j)$, and set 
$E_{ij}=\{e_{ij}^x\,|\,x\in\F_4\}$. Set $A_1=E_{12}E_{13}$ and 
$A_2=E_{13}E_{23}$. These are the only subgroups of $S_0$ isomorphic 
to $(C_2)^4$, and each involution in $\UT_3(4)$ is in $A_1\cup A_2$ (see 
\cite[Lemma 4.1(b,c)]{OV2}).

Set $S=S_0\times A$, and let $\calf$ be a reduced fusion system over $S$. 
Using Lemma \ref{l:S0*xA*}, we can assume that $A=[\autf(S),Z(S)]$, and 
hence is normalized by $\autf(S)$. Set $W=[\autf(S),\Omega_1(Z(S))]$; then 
$W=\Omega_1(A)$ and $C_W(\autf(S))=1$ by Lemma \ref{l:S0*xA*} again. Let 
$P$ be as in Lemma \ref{l:SxA}(a--c) for $\5W=W$, and set 
$V=\Omega_1(Z(P))$. In particular, $P\in\EE_\calf$ and $P=C_S(V)$. Set 
$V_0=V\cap S_0$ and $P_0=P\cap S_0$. 

\iffalse
Let $V$ and $P$ be as in Lemma 
\ref{l:SxA}(a--c) for $\5W=W$. In particular, 
$P\in\EE_\calf$, $V=\Omega_1(Z(P))$, and $P=C_S(V)$. Set $V_0=V\cap S_0$ 
and $P_0=P\cap S_0$. 
\fi

If $V_0\nleq\UT_3(4)$, then there is $g\in{}V_0{\sminus}\UT_3(4)$ of order 
two, and $C_{\UT_3(4)}(g)\cong Q_8$ with center $\gen{e_{13}^1}$ by 
\cite[Lemma 4.1(a)]{OV2}. Hence 
$V_0\le\Omega_1(C_{S_0}(g))=\gen{g,\e13^1}$, with equality since 
$V_0=\Omega_1(Z(P_0))\ge Z(S_0)=\gen{\e13^1}$. So $P=\gen{g}\times 
C_{\UT_3(4)}(g)\times A$. Each class in $N_S(P)/P$ is represented by some 
$x\in\UT_3(4)$, $[x,g]\in Z(P)\cap\UT_3(4)=\gen{\e13^1}$ since $g\in Z(P)$ 
and $x\in\UT_3(4)$, and $[x,C_{\UT_3(4)}(g)]\le E_{13}\cap P=\gen{\e13^1}$. 
Hence $[N_S(P),P]=\gen{\e13^1}\le\Phi(P)$, so $P\notin\EE_\calf$ by 
Proposition \ref{p:QcharP} (with $\Theta=1$), a contradiction.  

Thus $V_0\le\UT_3(4)$. If $V_0\nleq{}E_{13}=Z(\UT_3(4))$, then there is 
$g\in (V_0\cap A_i)\sminus E_{13}$ for $i=1$ or $2$, since all elements of 
order $2$ in $\UT_3(4)$ lie in $A_1\cup A_2$. Hence $P\le C_S(g)=A_i\times 
A$, with equality since $P$ is centric in $S$. Set 
$\Gamma=\{\alpha\in\autf(S)\,|\,\alpha(P)=P\}$: a subgroup of index $2$ 
since $A_1,A_2$ are the only subgroups of $S_0$ isomorphic to $(C_2)^4$ (and they are 
$S$-conjugate). Then $\autf(S)=\Inn(S)\Gamma$, and 
$C_W(\Gamma)=C_W(\autf(S))=1$ since $[\Inn(S),W]=1$. So the hypothesis in 
Lemma \ref{l:SxA}(d) holds, which is impossible since $\rk(P/A)=4$ and 
$\rk(C_V(N_S(P))/W)=\rk(C_{A_i}(\UT_3(4)))=2$.

The only remaining case is that where $V_0=E_{13}$ and $P=\UT_3(4)\times 
A$. This is impossible by Lemma \ref{l:SxA}(d) again (applied with 
$\Gamma=\autf(S)$), and since 
$\rk(V_0)=2$ and $\UT_3(4)=A_1A_2$ is characteristic in $S_0$.
\end{proof}

\begin{Prop} \label{2^k.UT4(2)}
For any finite abelian $2$-group $A\ne1$, there are no reduced fusion 
systems over $\UT_4(2)\times A$.
\end{Prop}

\begin{proof} Set $S_0=\UT_4(2)$ and $S=S_0\times A$. We need to use the 
following properties of $S_0$ (see, e.g., \cite[Lemma C.4(a)]{O-rk4}): 
there is a unique abelian subgroup $B_0\le S_0$ of order 16, $B_0\cong 
(C_2)^4$, $B_0\nsg S_0$, and $S_0/B_0\cong \klfour$ acts on $B_0$ 
by permuting a basis freely. 

Assume $\calf$ is a reduced fusion system over $S$. Using Lemma 
\ref{l:S0*xA*}, we can assume that $A=[\autf(S),Z(S)]$ and hence 
is normalized by $\autf(S)$. 

Set $B=B_0\times A$. Let $P_1,P_2,P_3$ be the three subgroups of index $2$ 
in $S$ which contain $B$. Thus $P_i=B\gen{x_i}$, and $Z(P_i)=C_B(x_i)$, for 
some $x_1,x_2,x_3\in S_0\sminus B_0$.
We first claim that 
	\beqq \calp \defeq \bigl\{P\in\EE_\calf \,\big|\, P=C_S(V) 
	\textup{ where } V=\Omega_1(Z(P)) \bigr\}
	\subseteq\{B,P_1,P_2,P_3\} \,. \label{e:calp} \eeqq

Assume otherwise: then there is $P\in\calp$ such that $P\ngeq B$. Set 
$P_0=P\cap S_0$. If $P_0$ is elementary abelian, then $\rk(P_0)\le3$ since 
$B_0$ is the unique abelian subgroup of order 16 in $S_0$, so 
$\rk([x,P])=\rk([x,P_0])=1$ for $x\in N_S(P)\sminus P$ such that $x^2\in 
P$. By Proposition \ref{p:critical}(b), $|N_{S_0}(P_0)/P_0|=|N_S(P)/P|=2$. 
Since $P$ is centric in $S$, $P\not<B$. Hence either 
$P_0\le\gen{g,B_0}$ for some $g\in P_0\sminus B_0$, in which case 
$P_0=\gen{g,C_{B_0}(g)}$ and $B\le N_S(P)$; or there are $g,h\in P_0\sminus 
B_0$ with $gh^{-1}\notin B_0$, in which case $P_0\cap 
B_0\le C_{B_0}(\gen{g,h})=Z(S_0)$ and $Z_2(S_0)\le N_S(P)$. Thus 
$|N_{S_0}(P_0)/P_0|\ge4$ in all cases, which is a contradiction.

Assume $P_0$ is not elementary abelian, and set $V_0=\Omega_1(Z(P_0))$. 
Then $\rk(V_0)\le2$ since $B_0$ is the only abelian subgroup of $S_0$ of 
order $16$, $V_0\nleq B_0$ since $P_0=C_{S_0}(V_0)\ngeq B_0$, so 
$V_0=\gen{x,Z(S_0)}$ for some $x\in S_0\sminus B_0$ of order $2$. Then 
$P_0=C_{S_0}(V_0)=\gen{x,y,C_{B_0}(x)}$ for some $y\in C_{S_0}(x)\sminus 
B_0\gen{x}$, $\gen{x,C_{B_0}(x)}\cong (C_2)^3$, and hence $P_0\cong 
C_2\times D_8$ and $[N_S(P),P_0]=Z(S_0)=\Phi(P_0)$, which contradicts 
Proposition \ref{p:QcharP} (with $\Theta=1$). This proves \eqref{e:calp}.

Set
	\begin{align*} 
	V &= \Omega_1(B) & G &= \autf(B) & T &= \Aut_S(B)\in\syl2{G} \\
	W &= \Omega_1(A) & H &= N_G(T) & H_0 &= C_G(T) \,.
	\end{align*}
Regard $V$ as an $\F_2[G]$-module. Since each $\beta\in H=N_G(T)$ extends 
to some $\4\beta\in\autf(S)$ by the extension axiom, $W$ is an 
$\F_2[H]$-submodule of $V|_H$. For each $\alpha\in\autf(S)$, 
$\alpha(B)=B$ and $\alpha|_B\in H$ since $B$ is the unique abelian subgroup 
of index $4$ in $S$, and hence 
$W=[\autf(S),\Omega_1(Z(S))]=[H,\Omega_1(Z(S))]$ and 
$C_W(H)=C_W(\autf(S))=1$ by Lemma \ref{l:S0*xA*}. Also, $V=B_0\times W$ as 
$\F_2[T]$-modules, the $\F_2[T]$-linear projection $V\Right2{}W$ can be made 
$\F_2[H]$-linear by averaging over cosets in $H/T$, and hence $V=B_0^*\times W$ 
for some $\F_2[H]$-submodule $B_0^*$. Since $B_0^*|_T\cong B_0\cong\F_2[T]$ 
as $\F_2[T]$-modules, $B_0^*$ is indecomposable as an $\F_2[T]$-module, and 
hence also indecomposable as an $\F_2[H]$-module.

In particular, if $V$ is decomposable as an $\F_2[G]$-module, then $T$ 
acts trivially on all but one of its indecomposable direct factors. Let 
$\5W$ be an indecomposable direct factor with trivial $T$-action; then 
$\5W\le C_V(T)\le \Omega_1(Z(S))$. Also, $\5W|_H\cong \5W^*$ as 
$\F_2[H]$-modules for some $\5W^*\le W$, so $C_{\5W}(H)=1$, and hence 
$\5W=[H,\5W]\le[H,\Omega_1(Z(S))]=W$. Moreover, $\5W$ is normalized by 
$\autf(P)$ for each $P\in\{B,P_1,P_2,P_3,S\}$ since $B$ is characteristic 
in each of those subgroups. But this contradicts Lemma \ref{l:SxA}(a,b) and 
\eqref{e:calp}.

Thus $V$ is indecomposable as an $\F_2[G]$-module. Since $V|_H=B_0^*\times 
W$ where $T$ acts trivially on $W$ and $T\in\syl2G$, Lemma \ref{vx-props}(b) 
implies that $W$ 
is the $\F_2[H]$-Green correspondent of $V$ and is irreducible as an 
$\F_2[H/T]$-module. Set 
$\Gamma=\bigl\{\alpha\in\autf(S)\,\big|\, \alpha|_B\in H_0\bigr\}$; we 
already saw that each element of $H_0$ extends to an element of $\Gamma$. 
If $H_0$ acts nontrivially on $W$ (equivalently, if $\Gamma$ acts 
nontrivially on $W$), then $C_W(\Gamma)=C_W(H_0)=1$ ($C_W(H_0)$ is an 
$\F_2[H]$-submodule of $W$ since $H_0\nsg H$), and $\Gamma(P)=P$ for $P\in 
\{B,P_1,P_2,P_3\}$ by definition of $H_0$. This contradicts Lemma 
\ref{l:SxA}(d) and \eqref{e:calp}.

Thus $H_0$ acts trivially on $W$, while $H$ does not since 
$W=[H,\Omega_1(Z(S))]$. By definition, $H/H_0\ne1$ acts as a group of 
automorphisms of $T\cong \klfour$ of odd order, and hence has order $3$. So 
$\rk(W)=2$ since $W$ is irreducible. Also, $\F_4\otimes_{\F_2}W$ is 
decomposable as an $\F_4[H]$-module, and the two direct factors are Galois 
conjugate, so $\F_4\otimes_{\F_2}V$ contains the sum of their Green 
correspondents (Lemma \ref{vx-props}(b)) which also are Galois conjugate. 
Since $V$ is indecomposable, $\F_4\otimes_{\F_2}V$ is the direct sum 
of at most two indecomposable submodules, and hence is the sum of exactly 
two indecomposable modules which are Galois conjugate. In particular, 
$(\F_4\otimes_{\F_2}V)|_H$ is the direct sum of at least four 
indecomposable modules. But this is impossible, since $V|_H\cong 
B_0^*\times W$, and $\F_4\otimes_{\F_2}B_0^*\cong\F_4[T]$ as 
$\F_4[T]$-modules and hence is indecomposable. 
\end{proof}

%%%%%%%%%%%%%%%%%

\section{$2$-Groups of order at most 128}
\label{s:128}

\bigskip

Reduced fusion systems over $2$-groups of order at most $64$ have already 
been handled in earlier papers.  We begin by listing them.  Recall that 
$\UT_n(q)$ denotes the group of upper triangular $n\times n$ matrices over 
$\F_q$ with $1$'s on the diagonal.

\begin{Thm} \label{t:order64}
Let $\calf$ be a reduced, indecomposable fusion system over a nontrivial 
$2$-group $S$ of order $2^k\le64$.  Then one of the following holds:
\begin{enuma}  
\item $S\cong{}D_{2^k}$ for some $3\le k\le6$, and $\calf$ is isomorphic to 
the fusion system of $\PSL_2(q)$ where $q\equiv2^k\pm1$ (mod $2^{k+1}$).

\item $S\cong{}SD_{2^k}$ for some $4\le k\le6$, and $\calf$ is isomorphic 
to the fusion system of $\PSL_3(q)$ where $q\equiv2^{k-2}-1$ (mod 
$2^{k-1}$).

\item $S\cong{}C_4\wr{}C_2$, and $\calf$ is isomorphic to the fusion system 
of $\SL_3(5)$.

\item $S\cong{}\UT_4(2)$, and $\calf$ is isomorphic to the fusion system of 
$\SL_4(2)\cong A_8$ or of $\SU_4(2)\cong\PSp_4(3)$.

\item $S\cong{}\UT_3(4)$, and $\calf$ is isomorphic to the fusion system of 
$\PSL_3(4)$.

\item $S$ is of type $M_{12}$, and $\calf$ is isomorphic to the fusion 
system of $M_{12}$ or of $G_2(3)$.
\end{enuma}
In all cases, $\calf$ is tame. Also, the fusion system of each of the 
groups listed above is simple. 
\end{Thm}

\begin{proof} If $S$ is dihedral, semidihedral, or a wreath product 
$C_{2^n}\wr C_2$, then by \cite[\S\,4.1]{AOV1} or \cite[Proposition 
3.1]{AOV2}, $\calf$ is isomorphic to the fusion system of $\PSL_2(q)$ or 
$\PSL_3(q)$ for appropriate odd $q$, and we are in one of the cases (a), 
(b), or (c). If not, then $|S|=64$ by \cite[Theorem 5.3]{AOV2}, and so $S$ 
is isomorphic to $\UT_4(2)$ or $\UT_3(4)$ or is of type $M_{12}$ by 
\cite[Theorem 5.4]{AOV2}. The reduced fusion systems over these three 
groups are listed in \cite[Propositions 5.1, 6.4, \& 4.2]{O-rk4}, and we 
are in the situation of (d), (e), or (f).

Tameness for all of these fusion systems is shown in \cite[Theorem 
C]{BMO2}, except for that of $M_{12}$.  When $\calf=\calf_S(G)$ for 
$G=M_{12}$, then by \cite[Proposition 4.3(a,b)]{O-rk4}, the composite 
$\mu_G\circ\kappa_G$ is an isomorphism from $\Out(G)$ to $\Out(\calf)$.  
By \cite[Proposition 3.2]{AOV2}, there are at most two $\calf$-essential 
subgroups, of which only one ($R$ in the notation of \cite{AOV2} and $A_+$ 
in \cite[\S\,4]{O-rk4}) has noncyclic center.  Hence $\Ker(\mu_G)=1$ by 
Proposition \ref{p:newKer(mu)}, $\kappa_G$ is an isomorphism, and $\calf$ 
is tame. 

By \cite[Theorem A]{O-rk4}, the fusion systems of all of the groups listed 
above are simple.
\end{proof}

As part of the proof of Theorem \ref{t:order64}, we have also shown:

\begin{Lem} 
If $G$ is isomorphic to the sporadic simple group $M_{12}$, 
then (for $p=2$) $\kappa_G$ and $\mu_G$ are both isomorphisms. In 
particular, $G$ tamely realizes its fusion system.
\end{Lem}

\newcommand{\mgn}[1]{\textbf{\##1}}

\newcommand{\mgnr}[1]{\smallskip\noindent\mgn{#1} \textbf{:} }

We now turn to groups of order $128$, where we begin our use of a computer 
search to identify 2-groups which potentially could support a reduced 
fusion system.

\begin{Thm} \label{t:order128}
Let $\calf$ be a reduced, indecomposable fusion system over a group $S$ of 
order $128$.  Then one of the following holds:
\begin{enuma}  
\item $S\cong{}D_{128}$, and $\calf$ is isomorphic to the fusion system of 
$\PSL_2(q)$ where $q\equiv\pm127$ (mod $256$).

\item $S\cong{}SD_{128}$, and $\calf$ is isomorphic to the fusion system of 
$\PSL_3(q)$ where $q\equiv31$ (mod $64$).

\item $S\cong{}C_{8}\wr{}C_2$, and $\calf$ is isomorphic to the fusion 
system of $\SL_3(9)$.

\item $S\cong{}D_8\wr{}C_2$, and $\calf$ is isomorphic to the fusion system 
of $A_{10}$ or of $\PSL_4(3)$.

\item $S$ is of type $M_{22}$, and $\calf$ is isomorphic to the fusion 
system of $M_{22}$, $M_{23}$, $\textup{McL}$, or $\PSL_4(5)$.

\item $S$ is of type $J_2$, and $\calf$ is isomorphic to the fusion system 
of $J_2$ or $J_3$.
\end{enuma}
In all cases, $\calf$ is tame. Also, the fusion system of each of the 
groups listed above is simple. 
\end{Thm}

\begin{proof} If $S$ is dihedral, semidihedral, or a wreath product 
$C_{8}\wr C_2$, then by \cite[\S\,4.1]{AOV1} or \cite[Proposition 
3.1]{AOV2}, $\calf$ is isomorphic to the fusion system of $\PSL_2(q)$ or 
$\PSL_3(q)$ for appropriate odd $q$, and we are in one of the cases (a), 
(b), or (c). If not, then it satisfies conditions (a)--(h) in Proposition 
\ref{p:search-criteria}. By a computer search based on those criteria, this 
leaves the six possibilities for $S$ listed below, where $\#(-)$ 
denotes the Magma/GAP identification number.

Tameness for the fusion systems in (a,b,e,f) is shown in \cite[Propositions 
4.3--5]{AOV1}.  Tameness for the fusion system of $A_{10}$ was shown in 
\cite[Proposition 4.8]{AOV1}, and for the other fusion systems in (c,d) in 
\cite[Theorem C]{BMO2}. All of these fusion systems are simple by 
\cite[Theorem A]{O-rk4}.

\mgnr{928}  $S\cong{}D_8\wr{}C_2$.  By \cite[Theorem A]{O-rk4}, $\calf$ 
must be as in case (d).  

\mgnr{931}  $S$ is of type $M_{22}$.  Any reduced fusion system over $S$ is 
as in (e) by \cite[Theorem 5.11]{OV2}. (If $\calf$ is isomorphic to the 
fusion system of $G\cong\PSSL_3(4)$ or $\PGGL_3(4)$, then 
$\foc(\calf)=S\cap[G,G]<S$ by the focal subgroup theorem \cite[Theorem 
7.3.4]{Gorenstein}.)

\mgnr{934}  $S$ is of type $J_2$.  By \cite[Theorem 4.8]{OV2}, $\calf$ must 
be isomorphic to the fusion system of $J_2$ or of $J_3$, and so we are in 
the situation of case (f).  (If $\calf$ is isomorphic to the fusion system 
of $G\cong \PSL_3(4)\sd{}\gen{\theta}$ or $\PGL_3(4)\sd{}\gen{\theta}$, as 
defined in that theorem, then $\foc(\calf)=S\cap[G,G]<S$ by the focal 
subgroup theorem \cite[Theorem 7.3.4]{Gorenstein}.)

\mgnr{1411} This group has two normal elementary abelian subgroups 
$P_1,P_2$ of rank $5$, where $S=P_1P_2$ and $Z\defeq Z(S)=P_1\cap 
P_2=[S,S]$. Also, for $i=1,2$ and $x_i\in P_i\sminus Z$, $C_S(x_i)=P_i$, 
and thus $S$ satisfies the hypotheses of Proposition \ref{p:AxE1E2} and of 
Lemma \ref{type1411} below. Let $z_0\in Z$ be as in Lemma 
\ref{type1411}.

Let $\calf$ be a saturated fusion system over $S$ with $O_2(\calf)=1$.  By 
Proposition \ref{p:AxE1E2}, $\EE_\calf=\{P_1,P_2\}$. For $i=1,2$, by 
Lemma \ref{l:HxA5} and since $\Aut_S(P_i)\cong \klfour$, we have 
$\autf(P_i)=\Delta_i\times H_i$ for some $\Delta_i\cong A_5$ and some $H_i$ 
of odd order.  By Lemma \ref{l:F2[A5]-mod}(b), and since $\rk([x,P_i])=2$ 
for $x\in S\sminus P_i$ and $[S,P_i]=C_{P_i}(S)$ has rank $3$, 
$C_{P_i}(\Delta_i)=C_{P_i}(N_{\Delta_i}(\Aut_S(P_i)))$ has rank $1$ (and is 
contained in $Z=Z(S)$). For $\alpha\in N_{\Delta_i}(\Aut_S(P_i))\cong A_4$ 
of order $3$, $\alpha$ extends to an automorphism of $S$ by the extension 
axiom, so $C_Z(\alpha)=\gen{z_0}$ by Lemma \ref{type1411} and hence 
$C_{P_i}(\Delta_i)=\gen{z_0}$. By the same lemma, $z_0\in C_S(\autf(S))$. 
So $\gen{z_0}\nsg\calf$ by Proposition \ref{Q<|F}, which contradicts the 
assumption that $O_2(\calf)=1$.  In particular, there are no reduced fusion 
systems over $S$.

\mgnr{2011}  $S\cong{}D_8\times{}D_{16}$.  By \cite[Theorem B]{O-split}, 
$\calf$ is decomposable.

\mgnr{2013}  $S\cong{}D_8\times{}SD_{16}$.  By \cite[Theorem 
B]{O-split}, $\calf$ is decomposable.
\end{proof}

The following lemma was needed in the above proof (see group 
\#1411).

\begin{Lem} \label{type1411}
Let $S$ be a group of order $128$ generated by two normal elementary 
abelian subgroups $P_1,P_2\nsg S$ of rank $5$, where $Z(S)=P_1\cap 
P_2=[S,S]$ has rank $3$. Assume also that $[x_1,x_2]\ne1$ for $x_i\in 
P_i\sminus Z(S)$ ($i=1,2$). 
Then there exists a unique involution $z_0\in Z(S)$ which is not a 
commutator. Moreover, for each $\alpha\in\Aut(S)$, $\alpha(z_0)=z_0$, 
and $C_{Z(S)}(\alpha)=\gen{z_0}$ if for $i=1$ or $i=2$, 
$\alpha$ normalizes $P_i$ and induces 
an automorphism of order $3$ on $P_i/Z$.
\end{Lem}

\begin{proof} Set $Z=Z(S)$. 
Fix representatives $g_1,g_2,g_3\in P_1\sminus Z$ and $h_1,h_2,h_3\in 
P_2\sminus Z$ for the nontrivial cosets in $P_1/Z$ and $P_2/Z$. 
For each $i,j\in\{1,2,3\}$, $1\ne[g_{i},h_{j}]\in Z$, and these nine 
involutions generate $[S,S]=Z$. Since $|Z\sminus1|=7$, there must be two 
distinct pairs of indices $(i,j)$ and $(k,\ell)$ such that 
$[g_{i},h_{j}]=[g_{k},h_{\ell}]$. If $i=k$, then 
$[g_{i},h_{j}h_{\ell}]=[g_{i},h_{j}][g_{i},h_{\ell}]=1$, 
which is impossible since $h_{j}h_{\ell}\notin Z$. So $i\ne k$, $j\ne 
\ell$ by a similar argument, and without loss of generality, we can assume 
that $(i,j)=(1,1)$ and $(k,\ell)=(2,2)$. Set 
$z_1=[g_{1},h_{1}]=[g_{2},h_{2}]$, and also 
	\[ z_2=[g_{1}g_{2},h_{1}]=[g_{2},h_{1}h_{2}]
	\quad\textup{and}\quad
	z_3=[g_{1},h_{1}h_{2}]=[g_{1}g_{2},h_{2}]. \]
Thus $[g_{1},P_2]=\gen{z_1,z_3}$, $[g_{2},P_2]=\gen{z_1,z_2}$, 
$[g_{3},P_2]=[g_{1}g_{2},P_2]=\gen{z_2,z_3}$, and 
$Z=[S,S]=\gen{z_1,z_2,z_3}$. Set $z_0=z_1z_2z_3$: the unique involution in 
$Z$ that is not a commutator. Then $z_0$ is fixed by each 
$\alpha\in\Aut(S)$. If $\alpha$ acts with order $3$ on $P_1/Z$, then 
it normalizes $P_2$ since $P_1$ and $P_2$ are the only elementary abelian 
subgroups of rank $5$, so $\alpha$ permutes cyclically the subgroups 
$[g_i,P_2]$ for $i=1,2,3$, and hence 
$C_Z(\alpha)=\gen{z_0}$. A similar argument applies if $\alpha$ acts with 
order $3$ on $P_2/Z$.
\end{proof}

\bigskip

\section{Groups of order 256}
\label{s:256}

\bigskip

The following notation is used in this section to describe certain 
semidirect products $H\sd{}K$. A superscript $\lambda$ (for $\lambda\in\Z$) 
over the ``$\rtimes$'' means that one generator of $K$ acts on $H$ via 
$(g\mapsto g^\lambda)$, while a superscript ``$t$'' means that a generator 
acts by exchanging two factors (or central factors) of $H$. Thus, for 
example, in (d) below, $(C_8\times C_8)\sd{-1,t}(\klfour)$ means that one of the 
factors $C_2$ acts by inverting $C_8\times C_8$, while the other acts by 
exchanging the two $C_8$'s. 

Whenever we list potentially critical subgroups of $S$ (Definition 
\ref{d:p.crit.}), they were found using computer computations based 
on the criteria in Proposition \ref{p:critical-criteria}.

\begin{Thm} \label{t:order256}
Let $\calf$ be a reduced, indecomposable fusion system over a group $S$ of 
order $256$.  Then one of the following holds:
\begin{enuma}  
\item $S\cong{}D_{256}$, and $\calf$ is isomorphic to the fusion system of 
$\PSL_2(q)$ where $q\equiv\pm255$ (mod $512$).

\item $S\cong{}SD_{256}$, and $\calf$ is isomorphic to the fusion system of 
$\PSL_3(q)$ where $q\equiv63$ (mod $128$).

\item $S\cong(Q_{16}\times_{C_2}Q_{16})\sd{t}C_2$, and $\calf$ is 
isomorphic to the fusion system of $\PSp_4(7)$.

\item $S\cong(C_8\times{}C_8)\sd{-1,t}(\klfour)$, and $\calf$ is 
isomorphic to the fusion system of $G_2(7)$. 

\item $S$ is of type $\Ly$, and $\calf$ is isomorphic to the fusion 
system of Lyons' sporadic group.

\item $S$ is of type $\Sp_4(4)$, and $\calf$ is isomorphic to the fusion 
system of $\Sp_4(4)$. 

\end{enuma}
In all cases, $\calf$ is tame.  Also, the fusion system of each of the 
groups listed above is simple.
\end{Thm}

\begin{proof} If $S$ is dihedral or semidihedral, then by 
\cite[\S\,4.1]{AOV1}, $\calf$ is isomorphic to the fusion system of 
$\PSL_2(q)$ or $\PSL_3(q)$ for appropriate odd $q$, and we are in one of 
cases (a) or (b). If not, then since a wreath product of the form 
$C_{2^n}\wr C_2$ cannot have order $2^8$, $S$ satisfies conditions (a)--(h) 
in Proposition \ref{p:search-criteria}. By a computer search based on those 
criteria, $S$ must be one of the $18$ groups listed below, where 
$\#(-)$ denotes the Magma/GAP identification number.

All of the fusion systems listed in the theorem are tame by \cite[Theorem 
C]{BMO2}, except for the fusion system of Lyons's sporadic group, whose 
tameness is shown separately (see group \#6665 below). If $\calf$ is the 
fusion system of $\Sp_4(4)$, then $\calf$ is simple by Proposition 
\ref{p:FS(G)-red}(d). All of the other fusion systems listed above are 
simple by \cite[Theorem A]{O-rk4}. 

\smallskip

\noindent \mgn{12955}, \mgn{12957}, \mgn{12965}, \mgn{15421}, \mgn{26833}, 
\mgn{26835}, \mgnr{55683} In the first three cases, 
$S\cong{}D_{16}\times{}D_{16}$, $D_{16}\times\SD_{16}$, or 
$\SD_{16}\times\SD_{16}$, respectively.  In the last four cases, $S\cong 
D_8\times S_2$ where $S_2\cong C_4\wr C_2$, $D_{32}$, $\SD_{32}$, or 
$\klfour\times D_8$, respectively.  All of these groups satisfy the 
hypotheses of \cite[Theorem B]{O-split}, and hence every reduced fusion 
system over any of them is decomposable.

\mgnr{5298}  $S\cong(C_8\times{}C_8)\sd{-1,t}(\klfour)$.  By \cite[Proposition 
4.2(c)]{O-rk4}, every reduced fusion system over $S$ is isomorphic to the fusion 
system of $G_2(7)$.

\mgnr{5352}  $S\cong(C_8\times{}C_8)\sd{3,t}(\klfour)$.  By \cite[Proposition 
4.2(a)]{O-rk4}, there are no reduced fusion systems over $S$.

\mgnr{6331} Here, $S=V_1V_2\gen{x}$, where $V_1,V_2\nsg S$ are elementary 
abelian of rank $5$, $\5Z\defeq Z(V_1V_2)=V_1\cap V_2=[V_1,V_2]$ has rank 
$3$, $[v_1,v_2]\ne1$ for $v_i\in V_i\sminus \5Z$, $x^2=1$, and 
$[x,V_i]\nleq \5Z$ for $i=1,2$. In particular, $V_1V_2$ satisfies the 
hypotheses of Lemma \ref{type1411}, and $S/V_i\cong D_8$ for $i=1,2$. The 
group $S$ is, in fact, a Sylow 2-subgroup of $2\cdot M_{22}$, although we 
do not use that here.

For $i=1,2$, choose $u_i\in[x,V_i]\sminus \5Z$. Thus $[x,u_i]=1$ since 
$x^2=1$, and $u_{3-i}V_i$ generates the center of $S/V_i\cong D_8$. So the 
conjugation morphisms 
$c_x$ and $c_{u_{3-i}}$ commute in $\Aut(V_i)$. If $[x,\5Z]=1$, then 
for $v_2\in V_2\sminus\5Z\gen{u_2}$, 
$[\9xv_2,u_1]={}^x[v_2,u_1]=[v_2,u_1]$, so $[u_2,u_1]=[[x,v_2],u_1]=1$, 
which contradicts the above remarks about $V_1$ and $V_2$. Thus 
$[x,\5Z]\ne1$, and since $\rk(\5Z)=3$ and $x^2=1$, $[x,\5Z]<C_{\5Z}(x)$ 
have rank $1$ and $2$, respectively. Since each nontrivial coset in 
$S/V_i$ is represented (up to conjugacy) by an element of $V_{3-i}$, or by 
$x$, or by some $g$ with $g^2\in u_{3-i}V_i$, we have $\rk(C_{V_i}(g))\le3$ 
for $g\in{}S{\sminus}V_i$ ($i=1,2$). Also, 
$C_{V_i}(\gen{x,u_{3-i}})=C_{\5Z}(x)$ has rank $2$, and so 
	\beqq \textup{$V\le S$ elementary abelian of rank $5$ \quad 
	$\implies$ \quad $V=V_1$ or $V_2$.} \label{e:Vi_unique} \eeqq

For $i=1,2$, define 
	\[ Q_i=V_i\gen{x} \qquad\textup{and}\qquad 
	R_i=N_S(Q_i)=V_i\gen{x,u_{3-i}}\,. \]
By computer computations, there 
are five $S$-conjugacy classes of potentially critical subgroups in 
$S$: $V_1V_2$, subgroups $S$-conjugate to the 
$Q_i$, and the $R_i$. (It is easy to see that these are the only 
potentially critical subgroups containing $V_1$ or $V_2$; the hard part is 
to show that there are no others.) 

Assume there is a reduced fusion system $\calf$ over $S$.  Then the 
following hold. 
\begin{enumi}

\item For $i=1,2$, $V_i$ is normalized by $\autf(V_1V_2)$ and by 
$\autf(S)=\Inn(S)$:

By (1), $\alpha(V_i)=V_i$ for each $\alpha\in
\autf(V_1V_2)$ of odd order. This also holds for $\alpha\in
\Aut_S(V_1V_2)$ since $V_i\nsg S$. The Sylow axiom implies that
$\Aut_S(V_1V_2)$ is a Sylow $2$-subgroup of $\autf(V_1V_2)$, so
$\alpha(V_i)=V_i$ for all $\alpha\in \autf(V_1V_2)$.

Each $\alpha\in\autf(S)$ restricts to an element of $\autf(V_1V_2)$, hence 
normalizes $V_1$ and $V_2$, and normalizes each of the subgroups in the 
chain 
	\[ \Phi(S) = \5Z\gen{u_1,u_2} < V_1\gen{u_2} < V_1V_2 < S \,. \]
Since each of these has index $2$ in the following, $\autf(S)$ is a 2-group 
by Lemma \ref{l:mod-Fr}, and hence $\autf(S)=\Inn(S)$ by the Sylow axiom.

\item For $i=1,2$, if $R_i\in\EE_\calf$, then $Q_i\notin\EE_\calf$, and 
$Z(R_i)=C_{\5Z}(x)=Z(S)$ is centralized by $\autf(R_i)$:

To see the last claim, note first that $Z(R_i)=C_{\5Z}(x)\cong \klfour$. If $\alpha\in\autf(R_i)$, then $\alpha(Z(R_i))=Z(R_i)$, and 
$\alpha|_{Z(R_i)}=\alpha|_{Z(S)}$ extends to an automorphism in $\autf(S)$ 
by the extension axiom.  Since $\autf(S)=\Inn(S)$ by (i), 
$\alpha|_{Z(R_i)}=\Id$.  

Consider the following chain of subgroups characteristic in $R_i$:
	\[ [R_i,R_i] = [u_{3-i},V_i]\gen{u_i} < [R_i,R_i]Z(R_i) < V_i < R_i. \]
Each of the first two inclusions is of index $2$, while $R_i/V_i\cong 
\klfour$. By Lemma \ref{l:mod-Fr}, the kernel of the natural 
homomorphism from $\Aut(R_i)$ to $\Aut(R_i/V_i)\cong\Sigma_3$ is a 
$2$-group, and since $\outf(R_i)$ has a strongly 2-embedded subgroup, it 
must be isomorphic to $\Sigma_3$. Hence $Q_i=V_i\gen{x}$ is 
$\calf$-conjugate to $V_i\gen{u_{3-i}}\nsg S$, so $Q_i$ is not fully 
normalized in $\calf$, and hence is not in $\EE_\calf$.

\item For $i=1,2$, either $R_i\in\EE_\calf$ or $Q_i\in\EE_\calf$, but not 
both:  

If neither $R_i\notin\EE_\calf$ nor $Q_i\notin\EE_\calf$, then 
$V_{3-i}\nsg\calf$ by Proposition \ref{Q<|F}, since $V_{3-i}$ is 
characteristic in $Q_{3-i}$ and $R_{3-i}$ by \eqref{e:Vi_unique}, and is 
normalized by $\autf(V_1V_2)$ and $\autf(S)$ by (i).  By (ii), $R_i$ and 
$Q_i$ cannot both be in $\EE_\calf$.

\item Either $R_1\in\EE_\calf$ or $R_2\in\EE_\calf$:

By \eqref{e:Vi_unique}, for $i=1,2$, $V_i$ is characteristic of 
index two in $Q_i$, so $[\autf(Q_i),Q_i]\le V_i$. Recall that 
$\autf(S)=\Inn(S)$ by (i). Hence if $R_1\notin\EE_\calf$ and 
$R_2\notin\EE_\calf$, then 
$\Gen{[\autf(P),P]\,\big|\,P\in\EE_\calf}\le V_1V_2<S$, so 
$O^2(\calf)\ne\calf$ by Lemma \ref{l:foc=S}, contradicting the assumption 
that $\calf$ is reduced.

\item If $R_1,R_2\in\EE_\calf$, then $O_2(\calf)\ne1$: 

By (ii), $\EE_\calf\subseteq\{R_1,R_2,V_1V_2\}$, and $\autf(R_i)$ 
centralizes $Z(S)$ for $i=1,2$. By Lemma \ref{type1411}, there is 
$1\ne z_0\in C_{\5Z}(\autf(V_1V_2))\le C_{\5Z}(\autf(S))$, and in 
particular, $z_0\in Z(S)$. Hence $1\ne\gen{z_0}\nsg\calf$ 
by Proposition \ref{Q<|F}, and $O_2(\calf)\ne1$.

\item If $Q_i\in\EE_\calf$ ($i=1$ or $2$), then there is $1\ne W_i\le Z(S)$ 
that is normalized by $\autf(V_i)$:

By Lemma \ref{l:G-V5}, and since $2\big||\autf(V_i)|$ and 
$\autf(V_i)<\Aut(V_i)$ by the Sylow axiom, there is a proper subgroup 
$1\ne{}W_i<V_i$ which is normalized by $\autf(V_i)$.  As an 
$\F_2[S/V_i]$-module, $V_i/W_i$ surjects onto $C_2$ with the trivial 
action. Hence $W_i\le[S,V_i]=\5Z\gen{u_i}$ since $\5Z\gen{u_i}$ has index 
$2$ in $V_i$.

We have seen that $\Phi(Q_i)=[x,V_i]$ and $Z(Q_i)=C_{V_i}(x)$ have rank $2$ 
and $3$, respectively, and $[x,V_i]<C_{V_i}(x)$ since $c_x$ has order $2$ 
in $\Aut(V_i)$. So by Lemma \ref{l:mod-Fr} again, applied to the chain 
$\Phi(Q_i)<Z(Q_i)<V_i<Q_i$, each $\Id\ne\alpha\in\autf(Q_i)$ of odd order 
acts irreducibly on $V_i/Z(Q_i)\cong \klfour$. Hence $W_iZ(Q_i)=Z(Q_i)$ or 
$V_i$. Since $W_iZ(Q_i)\le\5Z\gen{u_i}<V_i$, we conclude that $W_i\le 
Z(Q_i)=C_{V_i}(x)$.

Since $W_i$ is normalized by $\Aut_S(V_i)\le\autf(V_i)$, it is normal in 
$S$. The cosets $xV_i$ and $u_{3-i}xV_i$ are conjugate in $S/V_i\cong D_8$, 
and hence $W_i\le C_{V_i}(\gen{x,u_{3-i}})=C_{\5Z}(x)=Z(S)$.

\item If $Q_i,R_{3-i}\in\EE_\calf$ for $i=1$ or $i=2$, then 
$O_2(\calf)\ne1$:

By (vi), there is $1\ne W_i\le Z(S)$ that is normalized by $\autf(V_i)$, 
and hence by $\autf(Q_i)$. Also, $W_i$ is normalized by $\autf(V_1V_2)$ by 
(i), and is centralized by $\autf(R_{3-i})$ by (ii). Since each 
$\calf$-essential subgroup is $S$-conjugate to $V_1V_2$, $Q_i$, or 
$R_{3-i}$ by (iii) (and since $\autf(S)=\Inn(S)$), $W_i\nsg \calf$ by 
Proposition \ref{Q<|F}, and hence $O_2(\calf)\ne1$.

\end{enumi} 
By (iii) and (iv), either $R_1$ and $R_2$ are both essential, or $R_i$ and 
$Q_{3-i}$ are essential for $i=1$ or $2$. Hence $O_2(\calf)\ne1$ by (v) or 
(vii), and $\calf$ is not reduced.

\mgnr{6661}  $S\cong(Q_{16}\times_{C_2}Q_{16})\sd{t}C_2$.  By 
\cite[Proposition 5.6]{O-rk4}, every reduced fusion system over $S$ is 
isomorphic to the fusion system of $\PSp_4(7)$.

\mgnr{6662}  $S\cong(SD_{16}\times_{C_2}SD_{16})\sd{t}C_2$.  By 
\cite[Proposition 5.6]{O-rk4}, there are no reduced fusion systems over 
$S$.

\mgnr{6665}  $S$ is of type $\Ly$.  Let $G$ be Lyons's group, assume 
$S\in\syl2{G}$, and set $\calf=\calf_S(G)$.  By \cite[Proposition 
6.6]{O-rk4}, each reduced fusion system over $S$ is isomorphic to $\calf$.  
By the same proposition, there is exactly one $\calf$-essential subgroup 
with noncyclic center, and $\Out(\calf)=1$.  Hence 
$\mu_G\:\Out(BG_2^\wedge)\Right2{}\Out(\calf)$ is injective by 
Proposition \ref{p:newKer(mu)}, so $\Out(BG_2^\wedge)=1$, and $\calf$ is 
tame. In fact, since $\Out(G)=1$ by \cite[Proposition 5.8]{Lyons}, 
$\kappa_G$ and $\mu_G$ are both isomorphisms.

\mgnr{6666}  $S$ is the nonsplit extension of $\UT_3(4)$ by the group 
$\gen{\phi,\tau}$ of field and graph automorphisms (denoted 
$S^*_{\phi,\tau}$ in \cite[p. 71]{O-rk4}).  By \cite[Proposition 
6.6]{O-rk4}, there are no reduced fusion systems over $S$.

\mgnr{8935} In this case, $S=P_1P_2$, where $P_1$ and $P_2$ are elementary 
abelian subgroups of rank $6$ such that $Z(S)=P_1\cap P_2=[S,S]$ has rank 
$4$. By Proposition \ref{p:S4(4)} below, every reduced fusion system over 
$S$ is isomorphic to the fusion system of $\Sp_4(4)$. 

\noindent\mgn{53366}, \mgnr{53380}  $S\cong{}\klfour \times S_0$, where $S_0$ 
has type $M_{12}$ or $S_0\cong \UT_4(2)$.  By Proposition \ref{p:AxS-rk5} 
or \ref{2^k.UT4(2)}, respectively, there are no reduced fusion systems over 
$S$. (When $S_0$ has type $M_{12}$, $\Aut(S_0)$ is a 2-group by 
\cite[Proposition 3.2]{AOV2}.)

\noindent\mgnr{55676}  $S\cong{}\klfour \times\UT_3(4)$. By Proposition 
\ref{p:AxE1E2}, applied with $P_i=\klfour\times{}A_i\cong{}(C_2)^6$ where 
$A_1,A_2<\UT_3(4)$ are the two (normal) elementary abelian subgroups of 
rank $4$, there are no reduced fusion systems over $S$. 

\smallskip

This finishes the proof of the theorem.
\end{proof}

Within the proof of Theorem \ref{t:order256}, we have also shown:

\begin{Lem} 
If $G$ is isomorphic to the sporadic simple group of Lyons, 
then (for $p=2$) $\kappa_G$ and $\mu_G$ are both isomorphisms. In 
particular, $G$ tamely realizes its fusion system.
\end{Lem}

It remains to describe fusion systems over the group \#8935, which is a 
Sylow 2-subgroup of $\Sp_4(4)$.

\begin{Prop} \label{p:S4(4)}
Let $S$ be a group of order $256$ with subgroups $P_1,P_2\nsg 
S$ such that
	\beqq P_1\cong P_2\cong (C_2)^6,~ 
	P_1\cap{}P_2=Z(S)=[S,S]\cong{}(C_2)^4 \,. \label{e:S4(4)-sylow} \eeqq
Then every saturated fusion system $\calf$ over 
$S$ such that $O_2(\calf)=1$ is isomorphic to the fusion system of 
$\Sp_4(4)$.
\end{Prop}

\begin{proof} By comparing orders, we see that $S=P_1P_2$. Set $Z=Z(S)$ and 
$\4P_i=P_i/Z\cong \klfour$. Thus $S/Z=\4P_1\times\4P_2$. Choose subgroups 
$P_i^0<P_i$ ($i=1,2$) such that $P_i=Z\times P_i^0$, and set 
$\5S=P_1^0\*P_2^0$ (the free product). The commutator subgroup of 
$\5S/[\5S,[\5S,\5S]]$ is central in $\5S/[\5S,[\5S,\5S]]$ of order 
$16$ (isomorphic via the commutator pairing to the tensor product 
$P_1^0\otimes P_2^0$ when we regard $P_1^0$ and $P_2^0$ as $\F_2$-vector spaces), 
and hence the natural homomorphism from $\5S$ to $S$ is surjective with 
kernel $[\5S,[\5S,\5S]]$.

Assume $S^*$ is another group of order $2^8$, with subgroups 
$P_1^*,P_2^*\le S^*$ which also satisfy \eqref{e:S4(4)-sylow}. Choose 
subgroups $P_i^{*0}<P_i^*$ complementary to $Z(S^*)$, and set 
$\5S^*=P_1^{*0}\*P_2^{*0}$. Then $S^*\cong\5S^*/[\5S^*,[\5S^*,\5S^*]]$ by 
the above argument. Any pair of isomorphisms 
$\psi_i\in\Iso(P_i^0,P_i^{*0})$ ($i=1,2$) extends to an isomorphism 
$\5\psi\in\Iso(\5S,\5S^*)$ and hence to an isomorphism 
$\psi\in\Iso(S,S^*)$. The conditions \eqref{e:S4(4)-sylow} thus determine 
$S$ uniquely up to isomorphism. 

Let $\Aut^0(S/Z)<\Aut(S/Z)$ be the subgroup of those automorphisms which 
normalize the set $\{\4P_1,\4P_2\}$. The above argument, when applied with 
$S^*=S$ and $\{P_1^{*0},P_2^{*0}\}=\{P_1^0,P_2^0\}$, shows that the 
projection $S\to S/Z$ induces a (split) surjection 
	\[ \Psi\: \Aut(S) \Right5{} \Aut^0(S/Z) \cong 
	\bigl(\Aut(\4P_1)\times\Aut(\4P_2)\bigr)\sd{}C_2 \cong 
	\Sigma_3 \wr C_2\,. \]
(Note that $\Psi(\Aut(S))\le\Aut^0(S/Z)$ since by Proposition 
\ref{p:AxE1E2}, $P_1$ and $P_2$ are the unique maximal elementary abelian 
subgroups of $S$.) It follows that 
	\beqq \textup{for each $\Gamma\in\syl3{\Aut(S)}$, 
	$\Psi(N_{\Aut(S)}(\Gamma))=\Aut^0(S/Z)$} \label{e:5.2a} \eeqq
by the Frattini argument applied to the normal subgroup 
$\Ker(\Psi)\Gamma\nsg\Aut(S)$.

Let $\calf$ be a saturated fusion system over $S$ such that $O_2(\calf)=1$. 
By Proposition \ref{p:AxE1E2}, $\EE_\calf=\{P_1,P_2\}$. Since $|\outf(S)|$ 
is odd and $\Ker(\Psi)$ is a $2$-group by Lemma \ref{l:mod-Fr},
	\beqq \outf(S)\cong{}(C_3)^r \quad\textup{for some $0\le r\le2$.}
	\label{e:|outS|=3^r} \eeqq

Fix $i=1,2$. By Lemma \ref{l:HxA5}, and since $\Aut_S(P_i)\cong S/P_i\cong 
\klfour$, $\autf(P_i)=\Delta_i\times H_i$ for some $\Delta_i\cong{}A_5$ and 
some $H_i$ of odd order. Also, $\rk([s,P_i])=2$ for $s\in{}S{\sminus}P_i$, 
and $[S,P_i]=[S,S]=Z=C_{P_i}(S)$ has rank $4$. So by Lemma 
\ref{l:F2[A5]-mod}(b), $P_i$ is indecomposable as an 
$\F_2[\Delta_i]$-module, $Z_i\defeq{}C_{P_i}(\Delta_i)$ has rank $2$, and 
the action of $\Delta_i\cong{}A_5$ on $P_i/Z_i\cong{}(C_2)^4$ is isomorphic 
to the canonical action of $\SL_2(4)$ on $\F_4^2$.  Thus for $g\in\Delta_i$ 
of order three, $C_{P_i/Z_i}(g)=1$, and hence $C_{P_i}(g) = Z_i$. In 
summary, and since $\Aut_S(P_i)\le\Delta_i$, 
	\[ C_{P_i}(g) = Z_i = C_{P_i}(\Delta_i) \le C_{P_i}(\Aut_S(P_i)) = 
	Z\,. \]

For $i=1,2$, the homomorphism
	\[ \autf(S) \Right7{} N_{\autf(P_i)}(\Aut_S(P_i)) \cong
	A_4 \times H_i \]
induced by restriction is surjective by the extension axiom, and its kernel 
is a $2$-group by Lemma \ref{l:mod-Fr} (or Lemma \ref{l:res.aut.}). Hence 
$r\ge1$ and $H_i\cong{}(C_3)^{r-1}$ (with $r$ as in \eqref{e:|outS|=3^r}).

Fix some $\Gamma\in\syl3{\autf(S)}$.  Choose $\gamma_1,\gamma_2\in\Gamma$ 
of order $3$ such that $\gamma_i|_{P_i}\in\Delta_i$.  Then 
$Z_i=C_{P_i}(\Delta_i)=C_{P_i}(\gamma_i)=C_Z(\gamma_i)$ has rank $2$ and is 
normalized by $\Gamma$.  If $Z_1=Z_2$, then $Z_1\nsg\calf$ by Proposition 
\ref{Q<|F}, which contradicts our assumption that $O_2(\calf)=1$. Thus 
$Z_1\ne Z_2$ and $\gen{\gamma_1}\ne\gen{\gamma_2}$, so 
$\Gamma=\gen{\gamma_1,\gamma_2}\cong (C_3)^2$ by \eqref{e:|outS|=3^r}, and 
$r=2$. In particular, $\autf(P_i)\cong A_5\times C_3$ for $i=1,2$. Also, 
upon regarding $Z$ as an $\F_2[\Gamma]$-module with $Z_i=C_Z(\gamma_i)$, we 
see that $Z=Z_1\times Z_2$ where the $Z_i$ are 
$\F_2[\Gamma]$-submodules. We will show that the choice of the pair 
$\bigl(\gen{\gamma_1},\gen{\gamma_2}\bigl)$ completely determines $\calf$. 

For $i=1,2$, choose any $\eta_i\in\Gamma$ of order $3$ such that 
$\eta_i|_{P_i}\in H_i$. Since $\eta_i|_{P_i}$ commutes with $\Aut_S(P_i)$, 
$\eta_i$ acts trivially on $S/P_i$, and hence trivially on $\4P_{3-i}$. So 
$\eta_i$ must act nontrivially on $\4P_i$ since $\Ker(\Psi)$ is a 
$2$-group. Also, since $\gamma_i|_{P_i}$ does not commute with $\Aut_S(P_i)$, 
$\gamma_i$ acts nontrivially on $S/P_i\cong\4P_{3-i}$, and it acts 
nontrivially on $\4P_i=P_i/Z$ since $C_{P_i}(\gamma_i)=Z_i<Z$. Thus 
$\gen{\eta_1}$, $\gen{\eta_2}$, $\gen{\gamma_1}$, and $\gen{\gamma_2}$ are 
the four distinct subgroups of order $3$ in $\Gamma$, and $\gen{\eta_1}$ 
and $\gen{\eta_2}$ are the unique subgroups of $\Gamma$ of order $3$ which 
induce the identity on $\4P_2$ and $\4P_1$, respectively. 

Fix $\omega\in\F_4{\sminus}\F_2$. For $i=1,2$, $C_{P_i}(\eta_i)=1$, since 
$C_{\4P_i}(\eta_i)=1$ and 
$C_Z(\eta_i)=C_Z(\gamma_1^{\gee_{i1}}\gamma_2^{\gee_{i2}})=1$. Hence we can 
make  $P_i$ into an $\F_4$-vector space by defining multiplication by 
$\omega$ to be $\eta_i$.  Thus 
$C_{\Aut(P_i)}(\eta_i)=\Aut_{\F_4}(P_i)\cong\GL_3(4)$.  Set 
	\[ G_i=\bigl\{\alpha\in\Aut_{\F_4}(P_i) \,\big|\, \alpha(Z_i)=Z_i\bigr\}
	\qquad\textup{and}\qquad
	Q_i=O_2(G_i)\,. \]
Then $Q_i$ is the group of all $\alpha\in\Aut_{\F_4}(P_i)$ which induce the 
identity on $Z_i$ and on $P_i/Z_i$, $Q_i\cong\F_4^2\cong (C_2)^4$, and 
$G_i/Q_i\cong\GL_2(4)\times C_3\cong A_5\times (C_3)^2$. Set 
$K_i=\Aut_S(P_i)\gen{\gamma_i|_{P_i}}\cong A_4$.  Thus $K_i<\Delta_i<G_i$ 
for $i=1,2$.  

Let $\Delta^*_i<G_i$ be \emph{any} subgroup such that 
$\Delta^*_i\cong A_5$ and $K_i<\Delta^*_i$. Then $Q_i\cap\Delta_i^*=1$, and 
$\Delta_i^*Q_i=\Delta_iQ_i=O^3(G_i)$ since 
$G_i/Q_i\cong A_5\times (C_3)^2$. So by Proposition \ref{p:splittings}, 
applied with $G=G_i$, $Q=Q_i$, $H=\Delta_i$, and $H_0=K_i$, we have 
$\Delta^*_i=\9\psi(\Delta_i)$ for some $\psi\in{}C_{Q_i}(K_i)$.  Since 
$C_{Q_i}(K_i)\le C_{Q_i}(\gamma_i)=1$ (recall $C_{P_i}(\gamma_i)=Z_i$), it 
follows that $\Delta^*_i=\Delta_i$.

Thus $\autf(P_1)$ and $\autf(P_2)$ are determined by the choice of 
$\Gamma=\gen{\gamma_1}\times\gen{\gamma_2}$. So by Proposition 
\ref{p:AFT-E}, $\calf$ is determined by this choice. We just saw that 
$\gen{\gamma_1}$ and $\gen{\gamma_2}$ are the two subgroups of order $3$ in 
$\Gamma$ which act nontrivially on both $\4P_1$ and $\4P_2$. If $\calf'$ is 
another saturated fusion system over $S$ with $O_2(\calf')=1$, and 
$\calf'$ is determined by 
$\Gamma'=\gen{\gamma'_1}\times\gen{\gamma'_2}$, then there is 
$\varphi\in\Aut(S)$ such that $\9\varphi\Gamma'=\Gamma$. Either 
$\9\varphi\gen{\gamma'_i}=\gen{\gamma_i}$ for $i=1,2$, or 
$\9\varphi\gen{\gamma'_i}=\gen{\gamma_{3-i}}$. By \eqref{e:5.2a}, there is 
$\psi\in N_{\Aut(S)}(\Gamma)$ which exchanges $\gen{\gamma_1}$ and 
$\gen{\gamma_2}$, and so either $\9\varphi\calf'=\calf$ or 
$\9{\psi\varphi}\calf'=\calf$. 

Fix $S^*\in\syl2{\Sp_4(4)}$, and set $\calf^*=\calf_{S^*}(\Sp_4(4))$.  By 
the Chevalley commutator formula (see \cite[Theorem 5.2.2]{Carter}), 
the two unipotent radical subgroups $P^*_1,P^*_2<S^*$, 
are both isomorphic to $\F_4^3\cong (C_2)^6$, and are such that 
$P^*_1\cap{}P^*_2=[S^*,S^*]=Z(S^*)\cong (C_2)^4$. Hence $S^*$ satisfies 
\eqref{e:S4(4)-sylow}, so $S^*\cong S$. Also, $O_2(\calf^*)=1$ by 
Proposition \ref{p:FS(G)-red}(b), so $\calf^*\cong\calf$ since 
there is up to isomorphism at most one such saturated fusion 
system over $S$. 
\end{proof}

\bigskip

\section{Groups of order 512}
\label{s:512}

\bigskip

Throughout this section again, whenever we list potentially critical 
subgroups of $S$, they were found using computer computations based 
on the criteria in Proposition \ref{p:critical-criteria}.

\begin{Thm} \label{t:order512}
Let $\calf$ be a reduced, indecomposable fusion system over a $2$-group $S$ of 
order $512$.  Then one of the following holds:
\begin{enuma}  
\item $S\cong{}D_{512}$, and $\calf$ is isomorphic to the fusion system of 
$\PSL_2(q)$ where $q\equiv\pm511$ (mod $1024$).

\item $S\cong{}SD_{512}$, and $\calf$ is isomorphic to the fusion system of 
$\PSL_3(q)$ where $q\equiv127$ (mod $256$).

\item $S\cong{}C_{16}\wr{}C_2$, and $\calf$ is isomorphic to the fusion 
system of $\SL_3(17)$.

\item $S\cong{}D_{16}\wr{}C_2$, and $\calf$ is isomorphic to the fusion 
system of $\PSL_4(7)$.

\item $S\cong{}SD_{16}\wr{}C_2$, and $\calf$ is isomorphic to the fusion 
system of $\SL_5(3)$.

\item $S\cong{}\UT_3(8)$, and $\calf$ is isomorphic to the fusion system of 
$\SL_3(8)$.

\item $S$ is of type $A_{12}$, and $\calf$ is isomorphic to the fusion 
system of $A_{12}$, or of $\Sp_6(2)$, or of $\Omega_7(3)$.

\item $S\cong(Q_8\wr C_2)\times_{C_2}Q_8$, and $\calf$ is isomorphic to the 
fusion system of $\PSp_6(3)$.

\item $S$ is of type \HS, and $\calf$ is isomorphic to the fusion system of 
the Higman-Sims sporadic group.

\item $S$ is of type \ON, and $\calf$ is isomorphic to the fusion system of 
O'Nan's sporadic group.

\end{enuma}
In all cases, $\calf$ is tame. 
Also, the fusion system of each of the groups listed above is 
simple.
\end{Thm}

\begin{proof} If $S$ is dihedral, semidihedral, or a wreath product 
$C_{16}\wr C_2$, then by \cite[\S\,4.1]{AOV1} or \cite[Proposition 
3.1]{AOV2}, $\calf$ is isomorphic to the fusion system of $\PSL_2(q)$ or 
$\PSL_3(q)$ for appropriate odd $q$, and we are in one of the cases (a), 
(b), or (c). If not, then it satisfies conditions (a)--(h) in Proposition 
\ref{p:search-criteria}. By a computer search based on those criteria, we 
are left with the $31$ possibilities for $S$ listed below, where 
$\#(-)$ denotes the Magma/GAP identification number.

The tameness of the fusion systems of the sporadic groups of O'Nan and 
Higman-Sims (\#58362 and \#60329 below) is shown in Proposition 
\ref{t:O'N-HS}. All of the other fusion systems listed 
in the theorem are tame by \cite[Theorem C]{BMO2}, or (for the fusion 
system of $A_{12}$) by \cite[Proposition 4.8]{AOV1}. 

By \cite[Theorem A]{O-rk4}, the fusion systems listed above in cases 
(a--e) are all simple. By Proposition \ref{p:FS(G)-red}(a,b,d), if 
$\calf$ is any of the fusion systems in (f--j), then $O^2(\calf)=\calf$, 
$O_2(\calf)=1$, $\calf$ is indecomposable, and $\calf$ is simple 
except possibly when $\calf$ is the fusion system of $\Omega_7(3)$ or 
$\PSp_6(3)$. In case (g), $\Aut(S)$ is a $2$-group (see the proof of 
Proposition \ref{p:A12}), so those fusion systems are simple by Proposition 
\ref{p:FS(G)-red}(c,d). By Proposition \ref{t:S_6(3)}, when $S$ is of type 
$\PSp_6(3)$, there is a unique saturated fusion system $\calf$ over $S$ 
such that $O_2(\calf)=1$, and since $O_2(O^{2'}(\calf))=1$ \cite[Lemma 
1.20(e) \& Proposition 1.25(b)]{AOV1}, $\calf=O^{2'}(\calf)$ must be 
simple by Proposition \ref{p:FS(G)-red}(c,d).

\smallskip\noindent \mgn{128270}, \mgn{128271}, \mgn{399715}, \mgn{399717}, 
\mgn{399770},
\mgn{399771} :  In these six cases, $S=S_1\times{}S_2$ where 
$S_1\cong{}D_{16}$ or $SD_{16}$, and $S_2\cong{}D_{32}$, $SD_{32}$, or 
$C_4\wr{}C_2$.  By \cite[Theorem B]{O-split}, each reduced fusion system 
over any of these groups is decomposable.

\smallskip\noindent \mgn{420360}, \mgn{420362}, \mgn{6480905}, 
\mgn{7998954}, \mgn{6480855} :  In these cases, $S\cong{}D_8\times{}S_0$, 
where $S_0\cong{}D_{64}$, $SD_{64}$, $\UT_4(2)$, $\UT_3(4)$, or is of type 
$M_{12}$.  By \cite[Theorem B]{O-split}, each reduced fusion system over 
$S$ is decomposable.

\smallskip\noindent \mgn{10483221}, \mgn{10483222}, \mgn{10493114} :  In 
these cases, $S\cong D_8\times S_0$ where $S_0\cong \klfour\times D_{16}$, 
$\klfour\times\SD_{16}$, or $(C_2)^3\times D_8$, respectively. By \cite[Theorem 
B]{O-split}, there are no reduced fusion systems over $S$. 

\mgnr{7606661}  $S\cong D_8\times D_8\times D_8$, and each reduced fusion 
system over $S$ is decomposable by \cite[Theorem C]{O-split}.  

\smallskip\noindent \mgn{7530050}, \mgn{7530054}, \mgn{7530055}, 
\mgn{10482003} :  In these cases, $S\cong{}\klfour\times{}S_0$ where 
$S_0\cong{}D_8\wr{}C_2$ or has type $M_{22}$ or $M_{12}:2$; or 
$S\cong{}(C_2)^3\times{}S_0$ where $S_0$ has type $M_{12}$.  In all 
cases, $|Z(S_0)|=2$ and $\rk(S_0)\le4$. Also, $\Aut(S_0)$ is a $2$-group 
by \cite[Corollary A.10(c)]{O-rk4}, \cite[Lemma 5.5]{OV2}, 
the proof of \cite[Proposition 4.3(c)]{O-rk4}, or \cite[Proposition 
3.2]{AOV2}, respectively. So by Proposition \ref{p:AxS-rk5}, there are no 
reduced fusion systems over $S$. 

\smallskip\noindent \mgn{7530088}, \mgn{10482065} : $S\cong \klfour\times 
S_0$ where $S_0$ has type $J_2$, or $S\cong (C_2)^3\times\UT_4(2)$. By 
Proposition \ref{2^k.J2} or \ref{2^k.UT4(2)}, respectively, there are no 
reduced fusion systems over $S$. 

\mgnr{10493307} Here, $S\cong{}(C_2)^3\times\UT_3(4)$.  By Proposition 
\ref{p:AxE1E2}, applied with $P_i=(C_2)^3\times{}A_i\cong{}(C_2)^7$ where 
$A_1,A_2<\UT_3(4)$ are the two (normal) elementary abelian subgroups of 
rank $4$, there are no reduced fusion systems over $S$.  

\mgnr{58362} $S$ is of type \ON.  By Proposition \ref{t:O'N-HS} below, 
every reduced fusion system over $S$ is isomorphic to the fusion system of 
O'Nan's group.  

\mgnr{60329} $S$ is of type \HS.  By Proposition \ref{t:O'N-HS} below, 
every reduced fusion system over $S$ is isomorphic to the fusion system of 
the Higman-Sims group.  

\mgnr{60809} $S\cong D_{16}\wr C_2$.  By \cite[Proposition 5.5(a)]{O-rk4}, every 
reduced fusion system over $S$ is isomorphic to $\calf_S(\PSL_4(7))$.

\mgnr{60833} $S\cong SD_{16}\wr C_2$.  By \cite[Proposition 5.5(b)]{O-rk4}, every 
reduced fusion system over $S$ is isomorphic to $\calf_S(\SL_5(3))$.

\mgnr{406983} $S$ is of type $A_{12}$.  By Proposition \ref{p:A12} below, 
each reduced fusion system over $S$ is isomorphic to the fusion system of 
one of the groups $A_{12}$, $\Sp_6(2)$, or $\Omega_7(3)$.

\mgnr{6407070} By computer computations, $S$ has two potentially critical 
subgroups $P_1$ and $P_2$, both normal, where $P_1\cong2^{1+6}_+$, 
$S/P_1\cong \klfour$, and $|P_2|=2^8$. (In terms of the Magma/GAP 
generators, $P_1=\gen{s_1,s_4,s_5,s_6,s_7,s_8,s_9}$ and 
$P_2=\gen{s_2,s_3,s_4,s_5,s_6,s_7,s_8,s_9}$.) Set $V=P_1/Z(P_1)$, regarded 
as a $6$-dimensional $\F_2[S/P_1]$-module. Then $[S/P_1,V]=C_{V}(S/P_1)$ 
has rank $3$, $[x,V]$ has rank $2$ for each $1\ne x\in S/P_1$, and 
$\bigcap_{1\ne x\in S/P_1}[x,V]\ne1$. 

Let $\calf$ be a reduced fusion system over $S$. If $P_1\in\EE_\calf$, then 
by Lemma \ref{l:HxA5}, there is $\Gamma\le\outf(P_1)$ such that 
$\Gamma\cong A_5$ and $\Out_S(P_1)\in\syl2{\Gamma}$. This in turn induces 
an action of $\Gamma$ on $V$ that contains the action of 
$S/P_1$ on $V$ as a Sylow 2-subgroup. But this is impossible by Lemma 
\ref{l:F2[A5]-mod}(b) and the above remarks, and hence 
$P_1\notin\EE_\calf$.

Thus $\EE_\calf \subseteq \{P_2\}$, contradicting Lemma \ref{l:1crit}.

\mgnr{7530110} $S$ is of type $\PSp_6(3)$.  By Proposition \ref{t:S_6(3)} 
below, every reduced fusion system over $S$ is isomorphic to the fusion 
system of $\PSp_6(3)$.

\mgnr{7540630} $S\cong \klfour \times S_0$, where $S_0$ is the group of order 
$128$ with Magma/GAP number $1411$ described in the proof of Theorem 
\ref{t:order128}.  In particular, there are two subgroups $P_1,P_2<S$ 
isomorphic to $(C_2)^7$, and they satisfy the hypotheses of Proposition 
\ref{p:AxE1E2}.  Since $[S,S]<Z(S)$, there are no reduced fusion systems 
over $S$ by that proposition.

\mgnr{10481201} $S\cong \UT_3(8)$.  By Proposition \ref{t:ut3(8)} below, 
every reduced fusion system over $S$ is isomorphic to $\calf_S(\SL_3(8))$. 
\end{proof}

It remains to handle some of the individual cases.  

\begin{Prop} \label{t:ut3(8)}
Fix $k\ge1$.  A saturated fusion system $\calf$ over $\UT_3(2^k)$ is 
reduced if and only if it is isomorphic to the fusion system of 
$\PSL_3(2^k)$. 
\end{Prop}

\begin{proof} Since this holds for $k=1$ by \cite[Proposition 4.3]{AOV1} 
($\UT_3(2)\cong D_8$), we assume $k\ge2$ from now on. Set $q=2^k$ and 
$S=\UT_3(q)$ for short. For each $1\le{}i<j\le3$, let $E_{ij}\le S$ be the 
subgroup of upper triangular matrices with 1's on the diagonal and nonzero 
off-diagonal entry only in position $(i,j)$. Thus $E_{13}=Z(S)=[S,S]$. Set 
$A_1=E_{12}E_{13}$, and $A_2=E_{13}E_{23}$.  Then $A_1\cong 
A_2\cong(\F_q)^2\cong (C_2)^{2k}$, and we regard these groups as 
$\F_q$-vector spaces. 

Fix a reduced fusion system $\calf$ over $S$. By Proposition 
\ref{p:AxE1E2}, $\EE_\calf=\{A_1,A_2\}$.  For each $i=1,2$, set 
$\Gamma_i=\autf(A_i)$. 

\newcommand{\vv}{\mathscr{V}}

We first examine $\Gamma_1=\autf(A_1)$. Fix a strongly $2$-embedded 
subgroup $H<\Gamma_1$ which contains $\Aut_S(A_1)$. Consider the following 
sets of subgroups of $A_1$:
	\[ \vv_0=\bigl\{C_{A_1}(T)\,\big|\,T\in\syl2{H}\bigr\} \,, \qquad
	\vv_1=\bigl\{C_{A_1}(T)\,\big|\,T\in\syl2{\Gamma_1}\sminus\syl2{H}\bigr\} 
	\,, \]
	\[ \vv=\vv_0\cup\vv_1 = 
	\bigl\{C_{A_1}(T)\,\big|\,T\in\syl2{\Gamma_1}\bigr\} \,. \]
In particular, $E_{13}=C_{A_1}(\Aut_S(A_1))\in\vv_0$. Since $\Gamma_1$ 
permutes the elements of $\vv$ transitively (since 
$\varphi(C_{A_1}(T))=C_{A_1}(\9\varphi T)$ for $\varphi\in\Gamma_1$ and 
$T\in\syl2{\Gamma_1}$),  $\rk(V)=\rk(E_{13})=k$ for each $V\in\vv$. 

For each $V_0\in \vv_0$ and each $V_1\in\vv_1$, $V_0=C_{A_1}(T_0)$ and 
$V_1=C_{A_1}(T_1)$ for some $T_0\in\syl2H$ and 
$T_1\in\syl2{\Gamma_1}{\sminus}\syl2H$, and $\gen{T_0,T_1}$ contains the 
strongly $2$-embedded subgroup $\gen{T_0,T_1}\cap{}H$. In particular, 
$O_2(\gen{T_0,T_1})=1$, so $\gen{T_0,T_1}$ acts faithfully on 
$A_1/C_{A_1}(\gen{T_0,T_1})$ (Lemma \ref{l:mod-Fr}). By \cite[Lemma 
1.7(a)]{OV2}, $\rk\bigl(A_1/C_{A_1}(\gen{T_0,T_1})\bigr)\geq 2k$, so 
$V_0\cap V_1=C_{A_1}(\gen{T_0,T_1})=1$, and $A_1=V_0\times V_1$. 

In particular, $\vv_0\cap\vv_1=\emptyset$. Also, when $V_0=E_{13}$, this 
shows that each $V\in\vv_1$ contains a representative for each coset of 
$E_{13}$ in $A_1{\sminus}E_{13}$. Since $\Aut_S(A_1)$ acts transitively on 
each such coset, we conclude that $\bigcup_{V\in\vv_1}V\supseteq 
A_1{\sminus}E_{13}$. Hence $V_0\cap(A_1{\sminus}E_{13})=\emptyset$ for each 
$V_0\in\vv_0$, so $\vv_0=\{E_{13}\}$. Since $\Gamma_1$ acts transitively on 
$\vv$ and $E_{13}\cap V=1$ for each $V\in\vv\sminus\{E_{13}\}$, we see that 
$V\cap{}W=1$ for each pair $V,W$ of distinct elements in 
$\vv$. Thus $A_1{\sminus}1$ is the disjoint union of the sets $V{\sminus}1$ 
for all $V\in\vv$, so $|\vv|=q+1$, and $\Aut_S(A_1)$ normalizes the set 
$\vv_1=\vv{\sminus}\{E_{13}\}$. 

Fix some $V\in\vv{\sminus}\{E_{13}\}$. Since $A_1=E_{13}\times V$, there is 
$\varphi_1\in\Aut(A_1)$ which induces the identity on $E_{13}$ and on 
$A_1/E_{13}$, and such that $\varphi_1(V)=E_{12}$. Set 
$\vv^*=\varphi_1(\vv)$. Thus $E_{13}=\varphi_1(E_{13})$ and 
$E_{12}=\varphi_1(V)$ are both in $\vv^*$.  Also, 
$[\varphi_1,\Aut_S(A_1)]=1$ since the group of automorphisms of $A_1$ which 
induce the identity on $E_{13}$ and on $A_1/E_{13}$ is abelian, so 
$\Aut_S(A_1)$ also normalizes the set $\vv^*{\sminus}\{E_{13}\}$. Since 
$|\vv^*\sminus\{E_{13}\}|=q=|\Aut_S(A_1)|$ and $E_{12}$ is normalized only 
by the identity in $\Aut_S(A_1)$, $\Aut_S(A_1)$ acts transitively on this 
set. Hence $\vv^*=\varphi_1(\vv)$ is precisely the set of all 
$1$-dimensional $\F_q$-linear subspaces of $A_1$.

Let $\Theta_i^\SL < \Theta_i^\GL < \Theta_i^\GGL < \Aut(A_i)$ be the subgroups 
$\SL_2(q)$, $\GL_2(q)$, 
and $\GGL_2(q)$, respectively, defined with respect to the canonical 
$\F_q$-vector space structure on $A_i$.  By the fundamental theorem of 
affine geometry \cite[Th\'eor\`eme 2.6.3]{Berger1}, 
$N_{\Aut(A_1)}(\varphi_1(\vv))=\Theta_1^{\GGL}$, and hence 
$\9{\varphi_1}(\Gamma_1)\le\Theta_1^{\GGL}$. All Sylow 2-subgroups of 
$\9{\varphi_1}(\Gamma_1)$ are $\Theta_1^\GGL$-conjugate to $\Aut_S(A_1)$ 
and hence contained in $\Theta_1^\SL$. So 
$O^{2'}(\9{\varphi_1}(\Gamma_1))=\Theta_1^\SL$ since $\SL_2(q)$ is 
generated by any two of its Sylow 2-subgroups.

By a similar argument, there is $\varphi_2\in\Aut(A_2)$ which induces the 
identity on $E_{13}$ and on $A_2/E_{13}$, and such that 
$\Theta_2^{\SL}\le\9{\varphi_2}(\Gamma_2)\le\Theta_2^\GGL$ and 
$O^{2'}(\9{\varphi_2}(\Gamma_2))=\Theta_2^{\SL}$. Let 
$\varphi\in\Aut(S)$ be the unique automorphism such that 
$\varphi|_{A_i}=\varphi_i$ for $i=1,2$.  (Note that $\varphi$ has the form 
$\varphi(g)=g\chi(g)$ for some $\chi\in\Hom(S,Z(S))$.)  Upon replacing 
$\calf$ by $\9\varphi\calf$, we can assume that 
$\Theta_i^\SL\le\autf(A_i)\le\Theta_i^\GGL$ and 
$O^{2'}(\autf(A_i))=\Theta_i^{\SL}$ for each $i=1,2$. 

Set $G=\PSL_3(q)$ for short, and identify $S=\UT_3(q)$ with its image in $G$. 
Fix a generator $\lambda\in\F_q^\times$.  
Let $\beta_1,\beta_2\in\Aut_G(S)$ be conjugation by the classes of the diagonal 
matrices $\diag(\lambda,\lambda^{-1},1)$ and 
$\diag(1,\lambda,\lambda^{-1})$ respectively. Then 
$\Aut_{G}(S)=\text{Inn}(S)\gen{\beta_1,\beta_2}$ and for 
$a,b,c\in\F_q$, 
	\[ \beta_1\left(\mxthree1ab01c001\right) = 
	\mxthree1{\lambda^2a}{\lambda b}01{\lambda^{-1}c}001
	\qquad\textup{and}\qquad
	\beta_2\left(\mxthree1ab01c001\right) = 
	\mxthree1{\lambda^{-1}a}{\lambda b}01{\lambda^2c}001. \]
Thus $\beta_1|_{A_2}\in\Theta_2^\SL\le\autf(A_2)$. 
By the extension axiom, there is 
$\beta'_1\in\autf(S)$ such that $\beta'_1|_{A_2}=\beta_1|_{A_2}$. Then 
$\beta'_1|_{A_1}\in\autf(A_1)$, so $\beta'_1|_{A_1}\in\Theta_1^\GGL$ by our 
assumptions, and $\beta_1|_{A_1}\in\Theta_1^\GGL$ by construction.  Set 
$\tau=\beta_1^{-1}\beta'_1$; then $\tau$ induces the identity on $A_2$ and 
hence on $S/A_2$ (since $C_S(A_2)=A_2$) and $\tau|_{A_1}\in\Theta_1^\GGL$.  
It follows that $\tau|_{A_1}\in\Aut_{A_2}(A_1)$, and hence that 
$\tau\in\Inn(S)$. So $\beta_1\in\autf(S)$.  

By a similar argument, $\beta_2\in\autf(S)$.  Hence 
$\Aut_{G}(S)\le\autf(S)$.  Also, 
$\Aut_{G}(A_i)=\Theta_i^\SL\gen{\beta_i|_{A_i}}$ for $i=1,2$ (the subgroup 
of index $(3,q-1)$ in $\Theta_i^\GL$), so 
$O^{2'}(\autf(A_i))=\Theta_i^\SL\le\Aut_G(A_i)\le\autf(A_i)$.  Since 
$\EE_\calf=\{A_1,A_2\}$, we have $\calf_S(G)\subseteq\calf$ by Proposition 
\ref{p:AFT-E} (Alperin's fusion theorem). In addition, the $A_i$ are 
minimal $\calf$-centric subgroups, so $\Aut_G(P)\ge O^{2'}(\autf(P))$ for 
each $P\in\calf^c$. (If $P\in\calf^c{\sminus}\{A_1,A_2\}$, then 
$O^{2'}(\autf(P))=\Aut_S(P)$ by Proposition \ref{p:AFT-E} and since there 
are no larger essential subgroups.) So by \cite[Lemma I.7.6(a)]{AKO}, 
$\calf$ contains $\calf_S(G)$ as a subsystem of odd index in the sense of 
\cite[Definition I.7.3]{AKO}. Thus $\calf=\calf_S(G)$ since 
$O^{2'}(\calf)=\calf$ ($\calf$ is reduced). 

Conversely, $\calf_S(G)$ is simple by Proposition \ref{p:FS(G)-red}(d).
\end{proof}

The $2$-groups of type \ON\ and \HS\ will be handled together.

\begin{Prop} \label{t:O'N-HS}
\begin{enuma} 
\item If $S$ is a group of order 512 of type \ON, then every reduced fusion 
system over $S$ is isomorphic to the fusion system of O'Nan's group, and is 
tame.

\item If $S$ is a group of order 512 of type \HS, then every reduced fusion 
system over $S$ is isomorphic to the fusion system of the Higman-Sims group, 
and is tame.
\end{enuma}
\noindent If $G$ is either of the groups $\ON$ or $\HS$, 
then $\kappa_G$ and $\mu_G$ are both isomorphisms, and hence $G$ tamely 
realizes its fusion system.
\end{Prop}

\begin{proof}  We let $S^\ON$ and $S^\HS$ be 2-groups of type \ON\ or \HS, 
respectively, and set $S=S^\ON$ or $S^\HS$ when we do not need to 
distinguish them.  By \cite{O'Nan} or \cite{Alperin}, these groups have a 
presentation with generators $v_1,v_2,v_3,s,t$, where 
	\[ A=\gen{v_1,v_2,v_3}\cong (C_4)^3\,, \] 
and with additional relations 
	\begin{align*} 
	v_1{}^t&=v_3^{-1} & v_2{}^t&=v_2^{-1} & v_3{}^t&=v_1^{-1} & t^2&=1 \\ 
	v_1{}^s&=v_2 & v_2{}^s&=v_3 & v_3{}^s&=v_1v_2^{-1}v_3 & s^t&=s^{-1} 
	\end{align*}
(in both cases), and 
	\[ s^4 =  \begin{cases} 
	v_1v_3 & \textup{if $S=S^\ON$} \\
	1 & \textup{if $S=S^\HS$.}
	\end{cases} \] 
Thus $S$ is an extension of $A$ by $D_8$.  Also, $A$ is the unique subgroup 
of $S$ isomorphic to $(C_4)^3$, and hence is characteristic in all subgroups 
of $S$ which contain it.  

By computer computations, there are three conjugacy classes of potentially 
critical subgroups in each case, with representatives
	\begin{align*}  
	P_1&= A\gen{s^2,t} =C_S(v_1v_2^2v_3^{-1}) \\
	P_2^\ON&= \gen{s^2v_1,t,v_1v_3,v_1^2,v_2^2} \cong 
	C_4\times_{C_2}2^{1+4}_+ \\
	P_2^\HS&=\gen{sv_1v_2,t,v_1v_3,v_1^2,v_2^2} \\
	P_3&= A\gen{st,s^2} = C_S(v_1^2v_2^2)\,.
	\end{align*} 
Note that $P_1,P_3\nsg S$ (they have index two), while $N_S(P_2)=P_1$ when 
$S=S^\ON$, and $N_S(P_2)=P_2\gen{s}$ when $S=S^\HS$.  Since $P_1$ and $P_3$ 
are the only potentially critical subgroups of index $2$ in $S$ (and are 
not isomorphic to each other), they are both characteristic in $S$.

Set $w=v_1v_3$ and $z=w^2=v_1^2v_3^2\in{}Z(S)\cong C_2$.  By direct 
computation, $\Out(S)\cong{}(C_2)^4$, with explicit generators 
$[\chi_i]$ ($i=1,2,3,4$), where the $\chi_i\in\Aut(S)$ act as 
follows.  
	\beq \renewcommand{\arraystretch}{1.4} 
	\begin{array}{|l||c|c|c|c|c|c|c|} \hline 
	\chi & \chi(v_1) & \chi(v_2) & \chi(v_3) & \chi^\ON(s) & 
	\chi^\ON(t) & \chi^\HS(s) & \chi^\HS(t) \\ 
	\hline\hline 
	\chi_1 & v_1z & v_2z & v_3z & s & t & s & t \\ \hline 
	\chi_2 & v_1 & v_2 & v_3 & v_1^{-1}v_2^2v_3s & t 
	& w^{-1}s & zt \\ \hline
	\chi_3 & v_1 & v_2 & v_3 & v_1^{-1}v_2^{-1}s & t 
	& ws & w^{-1}t \\ \hline
	\chi_4 & v_1^{-1} & v_2^{-1} & v_3^{-1} & v_3s & v_2t & zs & t \\ \hline 
	\end{array} \eeq
Note that $\chi_i(P_j)=P_j$ for all $i,j$, except for the case 
$\chi_4(P_2^\ON) = {}^s(P_2^\ON) \neq P_2^{\ON}$.  Also, 
$\Out(P_2^\ON)\cong{}C_2\times\Sigma_6$, 
$\Out(P_2^\HS)\cong{}\klfour\times\Sigma_4$, and $\Out(P_1)$ and $\Out(P_3)$ 
are described by the following table (where the second factor in 
$\Out(P_3)$ has trivial center):
	\[ \renewcommand{\arraystretch}{1.4} 
	\begin{array}{|l||c|c|c|} \hline
	P & H^1(P/A;A) & N_{\Aut(A)}(\Aut_P(A))/\Aut_P(A) & \Out(P) \\ \hline\hline
	P_1 & \Z/4 & C_2\times\Sigma_4 & D_8\times\Sigma_4 \\ \hline
	P_3 & (\Z/2)^3 & C_2\times\Sigma_4 & 
	\klfour\times((C_2)^4\sd{}\Sigma_3) \\ \hline 
	\end{array} \]

Fix a reduced fusion system $\calf$ over $S$.  Then 
\begin{itemize} 
\item $\outf(S)=1$ since $\Out(S)$ is a 2-group; 

\item $P_2\in\EE_\calf$ by Proposition \ref{Q<|F}, since $A$ is 
characteristic in $P_1$ and $P_3$ and $A\not\nsg\calf$; and 

\item $P_3\in\EE_\calf$ since $Z(S)$ is characteristic in $P_1$ and in 
$P_2$ ($Z(S)\le Z(P_2)\le Z(P_1)=\gen{v_1v_2^2v_3^{-1}}\cong{}C_4$) and 
$Z(S)\not\nsg\calf$. 

\end{itemize}

\boldd{If $S=S^\ON$,} then since $P_2\in\EE_\calf$ and 
$\Out_S(P_2)\cong{}\klfour$, we have $\outf(P_2)\cong A_5\times H$ for some 
$H$ of odd order by Lemma \ref{l:HxA5}. In 
particular, there is $\xi\in N_{\autf(P_2)}(\Aut_S(P_2))$ of order three.  
By the extension axiom, $\xi$ extends to an element 
$\widehat\xi\in\autf(P_1)$ (recall $P_1=N_S(P_2)$).  Thus $\autf(P_1)$ is 
not a $2$-group, and so $P_1\in\EE_\calf$ by Proposition \ref{p:AFT-E}.

\boldd{If $S=S^\HS$,} set $T=N_S(P_2)\nsg S$.  Then 
$(P_3\cap{}T)/[P_3,P_3]=\Omega_1(P_3/[P_3,P_3])$, so $P_3\cap{}T$ is 
characteristic of index two in $P_3$.  Thus $[\Aut(P_3),P_3]\le{}T$ and 
$[\Aut(P_2),P_2]\le{}T$, so $P_1\in\EE_\calf$ by Lemma \ref{l:foc=S}.

\boldd{In both cases,} set $A_0=\Omega_1(A)=\gen{v_1^2,v_2^2,v_3^2}$. 
Since $C_S(A_0)=A$, 
	\[ \Aut_S(A_0) \cong S/A\cong D_8 \qquad\textup{and}\qquad
	\Aut_{P_i}(A_0)\cong P_i/A\cong \klfour ~ \textup{ for $i=1,3$}\,. \]
So for $i=1,3$, by Lemma \ref{l:res.aut.} (applied with $G=P_i$ and $H=A$), 
restriction induces a homomorphism from $\outf(P_i)$ into the group 
$N_{\Aut(A_0)}(\Aut_{P_i}(A_0))/\Aut_{P_i}(A_0)\cong\Sigma_3$ with kernel a 
$2$-group. This restriction homomorphism is an 
isomorphism since $P_i\in\EE_\calf$, so $\autf(A_0)$ contains the 
normalizers in $\Aut(A_0)$ of $\Aut_{P_1}(A_0)$ and $\Aut_{P_3}(A_0)$; 
i.e., the two maximal parabolic subgroups in $\Aut(A_0)$.  Thus 
$\autf(A_0)=\Aut(A_0)\cong\GL_3(2)$. 

Now, $O_2(\Aut(A))=\bigl\{\alpha\in\Aut(A)\,\big|\,\alpha|_{A_0}=\Id\bigr\} 
\cong (C_2)^9$, so $\Aut(A)/O_2(\Aut(A))\cong\Aut(A_0)$.  Let $\calg$ be the 
set of all subgroups $\Gamma<\Aut(A)$ such that $\Gamma\cong\GL_3(2)$ and 
$\Aut_S(A)<\Gamma$.  Since $\Aut_S(A)\cap{}O_2(\Aut(A))=1$, 
$\autf(A)\cap{}O_2(\Aut(A))=1$ by the Sylow axiom, and hence 
$\autf(A)\in\calg$ by the extension axiom. By Proposition 
\ref{p:splittings}, applied with $G=\Aut(A)$, $Q=O_2(G)$, $H=\autf(A)$, and 
$H_0=\Aut_S(A)$, the group 
$C_{O_2(\Aut(A))}(\Aut_S(A))=\gen{\chi_1|_A,\chi_4|_A}\cong \klfour$ acts 
transitively on $\calg$ via conjugation.  For $\Gamma\in\calg$ and $\chi\in 
O_2(\Aut(A))$, ${}^\chi\Gamma=\Gamma$ if and only if 
$[\chi,\Gamma]\le\Gamma\cap{}O_2(\Aut(A))=1$, in which case 
$[\chi,\Aut(A)]=[\chi,O_2(\Aut(A))]=1$ since $O_2(\Aut(A))$ is abelian.  
Thus $\chi_4|_A$ normalizes all elements in $\calg$, while $\chi_1|_A\notin 
Z(\Aut(A))$ normalizes none of them.  Hence $|\calg|=2$, and its elements 
are exchanged by $\chi_1|_A$.

Now fix some $\Gamma\in\calg$.  We can assume, after replacing $\calf$ by 
$\9{\chi_1}\calf$ if necessary, that $\autf(A)=\Gamma$. We claim that there 
is a unique possibility for $\autf(P_1)$ whose elements restrict to 
elements of $\Gamma$, and exactly two such possibilities for $\autf(P_3)$. 
To see this, note that for $P=P_1$ or $P_3$, the image of $\outf(P)$ 
in $N_{\Aut(A)}(\Aut_P(A))/\Aut_P(A)$ is precisely
$N_\Gamma(\Aut_P(A))/\Aut_P(A)\cong\Sigma_3$ by the extension axiom. When 
$P=P_1$, $O_3\bigl(N_\Gamma(\Aut_P(A))/\Aut_P(A)\bigr)$ lifts 
to a cyclic subgroup of order $12$ in $\Out(P_1)$, and 
hence there is only one choice for the subgroup of order $3$ in 
$\outf(P_1)$. When $P=P_3$, it lifts to a subgroup isomorphic to 
$C_2\times A_4$, so there are four subgroups of order $3$, of which 
just two are normalized by $\Out_S(P_3)$. By direct computations, these two 
possibilities for $\outf(P_3)$ are exchanged by $\chi_2$, and fixed by 
$\chi_3$ and $\chi_4$.

\boldd{If $S=S^\ON$,} then $\Out(P_2)\cong C_2\times\Sigma_6$.  We already 
saw that $O^{2'}(\outf(P_2))\cong{}A_5$, and hence $\outf(P_2)\cong A_5$.  
Also, $\outf(P_2)\le O^2(\Out(P_2))\cong A_6$, $A_6$ contains $12$ 
subgroups isomorphic to $A_5$ (six which act fixing a point and six which 
act transitively), and they are all conjugate in $\Aut(A_6)$ (see 
\cite[(3.2.19)]{Sz1}). Since each of those has five Sylow 2-subgroups all 
lying in the same $A_6$-conjugacy class, and $A_6$ has $30$ subgroups 
isomorphic to $\klfour$ ($15$ in each of two classes), we see by counting 
that each $\klfour\le A_6$ is contained in exactly two subgroups isomorphic 
to $A_5$. Thus there are two possibilities for $\outf(P_2)$, and they are 
exchanged by $\chi_3$ and fixed by $c_s\circ\chi_4$.  (Recall that $\chi_4$ 
does not normalize $P_2$.)  We now conclude that there is (up to 
isomorphism) at most one reduced fusion system $\calf$ over $S^\ON$, and 
that $\Out(\calf)=\gen{[\chi_4]}\cong C_2$.  In particular, $\calf$ 
is isomorphic to the fusion system of O'Nan's simple group, which is 
reduced by Proposition \ref{p:FS(G)-red}(d). 

\boldd{If $S=S^\HS$,} then $\Out(P_2)\cong \klfour\times\Sigma_4$ contains 
exactly four subgroups of order three of which two are normalized by 
$\Out_S(P_2)$.  Those two are exchanged by $\chi_3$ and fixed by $\chi_4$.  
Thus there is, up to isomorphism, at most one reduced fusion system $\calf$ 
over $S^\HS$, and $\Out(\calf)=\gen{[\chi_4]}\cong C_2$.  In 
particular, $\calf$ isomorphic to the fusion system of the Higman-Sims 
simple group, which is reduced by Proposition \ref{p:FS(G)-red}(d). 

\boldd{In both cases,} since $P_3$ is the only $\calf$-essential subgroup 
with noncyclic center, $\Ker(\mu_G)=1$ for $G=\ON$ or $\HS$ by Proposition 
\ref{p:newKer(mu)}. When $G=\ON$, then by \cite[Lemma 11.2]{O'Nan}, 
$|\Out(G)|\le2$. Also, $G$ contains two 
conjugacy classes of subgroups isomorphic to $L_3(7)^*$ by \cite[Lemma 
10.6(iii)]{O'Nan}, where $L_3(7)^*$ denotes the 
extension of $\PSL_3(7)$ by its graph automorphism (cf. \cite[p. 
471]{O'Nan}). If $\mu_G\circ\kappa_G$ is not injective, then there is 
$\alpha\in\Aut(G){\sminus}\Inn(G)$ such that $\alpha|_S=\Id$, $\alpha$ 
exchanges the two $G$-conjugacy classes of subgroups isomorphic to 
$L_3(7)^*$ by \cite[Lemma 11.1]{O'Nan}, and hence exchanges the two 
$G$-conjugacy classes of cyclic subgroups of order 16 (\cite[Lemma 
10.13]{O'Nan}). (By \cite[Lemma 4.3(vi)]{O'Nan}, $G$ contains four classes 
of elements of order $16$. Each element of order $16$ in $S$ has the form 
$(sa)^{\pm1}$ for $a\in A$, $|\Aut_S(\gen{sa})|=4$ for each $a$, and thus 
there are only two classes of cyclic subgroups of order $16$.) So 
$N_{\Aut(G)}(A)=N_{\Inn(G)}(A)\gen{\alpha}$ satisfies the hypotheses of 
\cite[Lemma 11.3]{O'Nan} (where $V$ in \cite{O'Nan} corresponds to $A$ 
here), which is impossible by point (ii) in that lemma. Thus 
$\mu_G\circ\kappa_G$ is injective. Since $|\Out(G)|=2$ (cf. \cite{JW}), 
$\mu_G$ and $\kappa_G$ are isomorphisms, and $\calf=\calf_S(G)$ is tame. 

Now assume that $G=\HS$. Then $|\Out(G)|\ge2$ by the construction of $G$ in 
\cite{HS}, with equality by, e.g., \cite[pp. 10--11]{Lyons-out}. By 
\cite[Table IIIb]{Frame}, there is no element 
$\alpha\in\Aut(G)\sminus\Inn(G)$ such that $C_G(\alpha)$ contains a Sylow 
2-subgroup of $G$. Thus $\mu_G\circ\kappa_G$ is injective, and hence (since 
$\mu_G$ is injective and $|\Out(G)|=|\Out(\calf)|$) $\mu_G$ and 
$\kappa_G$ are both isomorphisms. So $\calf=\calf_S(G)$ is tame. 
\end{proof}

It remains only to consider $2$-groups of type $\PSp_6(3)$.

\begin{Prop} \label{t:S_6(3)}
Let $S$ be a group of order 512 of type $\PSp_6(3)$, and thus isomorphic to 
$(Q_8\wr C_2)\times_{C_2}Q_8$.  Then every saturated fusion system over $S$ 
with $O_2(\calf)=1$ is isomorphic to the fusion system of $\PSp_6(3)$.
\end{Prop}

\newcommand{\z}[1]{z_{#1}}
\newcommand{\tauu}{t}  
\newcommand{\QQ}{\mathbf{Q}}   %% quaternion subgroups
\newcommand{\EX}{\mathbf{B}}   %% extraspecial subgroups

\begin{proof} Let $\QQ_1,\QQ_2\le S$ be the the two quaternion factors in 
$Q_8\wr C_2$, and fix $\tauu\in S$ such that $\tauu^2=1$ and 
$\9\tauu\QQ_1=\QQ_2$. Let $\QQ_3\le S$ be the other factor, and set 
$R_0=\QQ_1\QQ_2\QQ_3\cong(Q_8)^3/C_2$. For $i=1,2,3$, let $\z{i}\in 
Z(\QQ_i)$ be the generator. Thus for distinct $i,j\in\{1,2,3\}$, 
$[\QQ_i,\QQ_j]=1$ and $\QQ_i\cap\QQ_j=1$. Also, $z_1z_2z_3=1$, 
$Z(R_0)=\gen{\z1,\z2,\z3}\cong \klfour$, and 
$Z(S)=\gen{\z1\z2}=\gen{\z3}$.

Choose $\gamma\in\Aut(R_0)$ of order $3$ that permutes the $\QQ_i$ 
cyclically, and such that $\gen{\gamma,c_{\tauu}}\cong\Sigma_3$. For 
$i=1,2,3$, choose $\eta_i\in\Aut(R_0)$ such that $\eta_i|_{\QQ_i}$ is an 
automorphism of order $3$ and $\eta_i|_{\QQ_j}=\Id$ for $i\ne 
j$. Assume also that we have done this in such a way that conjugation by 
$\gamma$ and $c_{\tauu}$ in $\Aut(R_0)$ permutes the set 
$\{\eta_1,\eta_2,\eta_3\}$. In particular, 
$\gen{\eta_1,\eta_2,\eta_3,\gamma,c_{\tauu}}\cong C_3\wr\Sigma_3$. Also, 
since $\eta_1\eta_2$ and $\eta_3$ both commute with $c_{\tauu}$, there are 
automorphisms $\eta_{12}^S,\eta_3^S\in\Aut(S)$ such that 
	\[ \eta_{12}^S|_{R_0}=\eta_1\eta_2, \quad \eta_3^S|_{R_0}=\eta_3, 
	\quad\textup{and}\quad \eta_{12}^S(\tauu)=\tauu=\eta_3^S(\tauu). \]

The groups $\Out(R_0)$ and $\Out(S)$ were described in \cite[Lemma 
7.4(b,c,e)]{OV2} (where they are denoted $R_0$ and $R_2$, respectively). 
Each automorphism of $R_0$ permutes the three subgroups $\QQ_iZ(R_0)$.  Let 
$\Aut^0(R_0)\nsg\Aut(R_0)$ be the subgroup of those elements which are the 
identity on $Z(R_0)$; equivalently, which normalize each $\QQ_iZ(R_0)$. Set 
$\Out^0(R_0)=\Aut^0(R_0)/\Inn(R_0)$. Then 
	\beqq \Out^0(R_0)\cong\Sigma_4\times\Sigma_4\times\Sigma_4
	\qquad\textup{and}\qquad
	\Out(R_0)\cong\Sigma_4\wr\Sigma_3. \label{e:6.4b} \eeqq
The homomorphism 
	\beq \Out(S) \Right4{\cong}
	N_{\Out(R_0)}(\Out_{S}(R_0))\big/\Out_{S}(R_0) \cong
	\Sigma_4\times\Sigma_4 \eeq
induced by restriction is an isomorphism (by the proof of \cite[Lemma 
7.4(e)]{OV2}). In particular,
	\beqq \gen{\eta_1,\eta_2,\eta_3,\gamma}\in\syl3{\Aut(R_0)} 
	\quad\textup{and}\quad \gen{\eta_{12}^S,\eta_3^S}\in\syl3{\Aut(S)}. 
	\label{e:6.4c} \eeqq

We also need to work with the extraspecial subgroups 
	\[ \EX_0 = \gen{\z1,\z2,\tauu}\times_{C_2}\QQ_{12} \cong 2^{1+4}_- 
	\qquad\textup{and}\qquad 
	\EX = \EX_0 \times_{C_2} \QQ_3\cong 2^{1+6}_+\,, \]
where $\QQ_{12}=C_{\QQ_1\QQ_2}(\tauu)$ is the ``diagonal'' subgroup in 
$\QQ_1\QQ_2$. Consider the following subgroups and inclusions in 
$\Out(\EX)$:
	\beqq \underset{\cong A_4}{\Out_S(\EX)\gen{[\eta_{12}^S|_{\EX}]}}
	\times \underset{\cong A_3}{\gen{[\eta_3^S|_{\EX}]}} \le 
	\underset{\cong \Sigma_5}{\Out(\EX_0)} 
	\times \underset{\cong \Sigma_3}{\Out(\QQ_3)} \le 
	\underset{\cong \Sigma_8}{\Out(\EX)} \,. \label{e:Out(X)} \eeqq
Here, $\Out(\EX_0)\cong\SO_4^-(2)\cong \Sigma_5$ by \cite[Exercises 
8.5(3) \& 7.7(5)]{A-FGT}, and $\Out(\EX)\cong\SO_6^+(2)\cong \Sigma_8$ by 
\cite[Exercises 8.5(3) \& 7.7(7)]{A-FGT}. We are regarding $\Out(\EX_0)$ and 
$\Out(\QQ_3)$ as subgroups of $\Out(\EX)$: those classes of automorphisms 
that are the identity on the other factor.

\smallskip

Fix a saturated fusion system $\calf$ over $S$ such that $O_2(\calf)=1$. By 
computer computations, $R_0$ 
and $\EX$ are the only potentially critical subgroups of $S$, so 
$\EE_\calf=\{R_0,\EX\}$ by Lemma \ref{l:1crit}. 

Set $M=\gen{\eta_{12}^S,\eta_3^S}\in\syl3{\Aut(S)}$. Choose 
$\varphi\in\Aut(S)$ such that $\9\varphi\autf(S)\cap M 
\in\syl3{\9\varphi\autf(S)}$. Upon replacing $\calf$ by $\9\varphi\calf$, 
we can assume that $\autf(S)\cap M \in\syl3{\autf(S)}$. Equivalently, since 
$\outf(S)$ has odd order, it must be a $3$-group, and hence 
$\autf(S)\le\Inn(S)\cdot M$. 

Set $\4\EX=\EX/\gen{\z3}$, and more generally $\4P=P/\gen{\z3}$ when 
$\z3\in P\le\EX$. By Lemma \ref{l:HxA5}, $\outf(\EX)=\Delta\times 
X$, where $\Delta\cong A_5$ and $X$ has odd order. By Lemma 
\ref{l:F2[A5]-mod}(a), and since $[\Out_S(\EX),\4\EX]\nleq 
C_{\4\EX}(\Out_S(\EX))$ (and $\rk([x,\4\EX])=2$ for $x\in S\sminus\EX$), 
we have $\4\EX=\4B_1\times\4B_2$, where $\4B_1=[\Delta,\4\EX]$ is 
$4$-dimensional and irreducible as an $\F_2[\Delta]$-module and 
$\4B_2=C_{\4\EX}(\Delta)$ is 2-dimensional. For $i=1,2$, let 
$B_i\le\EX$ be such that $z_3\in B_i$ and $B_i/\gen{z_3}=\4B_i$. Thus 
$\4B_2\le C_{\4\EX}(\Out_S(\EX))=\4{\QQ_3\gen{\z1}}$, so 
$B_2\le\QQ_3\gen{\z1}\cong Q_8\times C_2$, and 
$\z1\in[\Aut_S(\EX),\EX]\le B_1$. Hence $B_2\cong Q_8$. 
Also, $B_1=C_{\EX}(B_2)$ since $\4B_1\cap 
\4{C_{\EX}(B_2)}$ is a nontrivial $\F_2[\Delta]$-submodule of $\4B_1$. Hence 
$B_1\cong2^{1+4}_-$, so $\Out(B_1)\cong\Sigma_5$, and 
$\Delta$ has index two in $C_{\Out(\EX)}(\4B_2)\cong\Out(B_1)$. In particular, 
$[\eta_{12}^S|_{\EX}]\in\Delta\le\outf(\EX)$ and hence 
$\eta_{12}^S|_{\EX}\in\autf(\EX)$.

By the extension axiom, $\eta_{12}^S|_{\EX}\in\autf(\EX)$ extends to an 
element $\eta\in\autf(S)$. Then $\eta|_{\QQ_3}=\Id$, and since 
$\autf(S)\le\Inn(S)\cdot M$, we have
$\eta\in\Inn(S)\cdot\eta_{12}^S$. Thus $\eta_{12}^S\in\autf(S)$, and 
$\outf(S)=\gen{[\eta_{12}^S]}$ or $\gen{[\eta_{12}^S],[\eta_3^S]}$.

Now, $\gen{\z3}\nnsg\calf$ since $O_2(\calf)=1$. Since 
$\gen{\z3}=Z(S)=Z(\EX)\cong C_2$, there is by Proposition \ref{Q<|F} an 
element of $\autf(R_0)$ which does not fix $\z3$. Also, 
$Z(R_0)=\gen{\z1,\z2}\cong \klfour$, where $\z3=\z1\z2$, and 
$c_{\tauu}\in\Aut_S(R_0)$ exchanges $\z1$ and $\z2$. So 
$\autf(Z(R_0))\cong\Sigma_3$. By the extension axiom, each element of 
$\autf(Z(R_0))$ is the restriction of an element of $\autf(R_0)$. Hence 
there is $\gamma'\in\autf(R_0)$ of odd order such that 
$\gamma'|_{Z(R_0)}=\gamma|_{Z(R_0)}$. In particular,  
$\gamma'\notin\Aut^0(R_0)$.

Let $H_0<H_i<H\nsg\Aut(R_0)$ ($i=1,2,3$) be the subgroups 
	\[ H_0=O_2(\Aut(R_0))\,, \qquad H_i=H_0\gen{\eta_i}\,,\qquad 
	H=H_1H_2H_3\le \Aut^0(R_0). \]
By \eqref{e:6.4b}, $H/\Inn(R_0)\cong(A_4)^3$, and 
	\[ H/H_0=O_3(\Aut(R_0)/H_0)=(H_1/H_0)\times(H_2/H_0)\times(H_3/H_0) 
	\cong(C_3)^3. \]
Consider the homomorphism 
	\[ \Psi\:\Aut(R_0)\Right6{}\Aut(H/H_0) \cong \GL_3(3) \]
induced by conjugation. Then $\Ker(\Psi)=H$, and $\Im(\Psi)\cong 
C_2\wr\Sigma_3$ acts on the set 
$\{\eta_1^{\pm1}H_0,\eta_2^{\pm1}H_0,\eta_3^{\pm1}H_0\}$ 
as the group of signed permutations. Also, conjugation by 
$c_{\tauu}$ exchanges $\eta_1$ and $\eta_2$ and sends $\eta_3$ to itself. 
Since $\Psi(\gamma')\ne1$ has odd order and 
$\Gen{[c_{\tauu}]}\in\syl2{\outf(R_0)}$ by the Sylow axiom, 
$\Psi(\autf(R_0))=\gen{\Psi(c_{\tauu}),\Psi(\gamma')}\cong\Sigma_3$, and 
$\Psi(\gamma')$ permutes cyclically the three cosets $\eta_1H_0$, 
$\eta_2H_0$, and $\eta_3^\gee H_0$ for some $\gee=\pm1$.

Choose $\beta\in\Aut(S)$ such that
	\[ \beta|_{\QQ_1\QQ_2}=\Id\,,\qquad \beta(\tauu)=\tauu\,, \qquad 
	\beta(\QQ_3)=\QQ_3\,, \]
and such that $\beta$ induces an automorphism of order $2$ on 
$\QQ_3/\gen{\z3}$. Then conjugation by $[\beta|_{R_0}]$ in $\Out(R_0)$ 
fixes $[\eta_1]$ and $[\eta_2]$ and inverts $[\eta_3]$. So upon replacing 
$\calf$ by $\9\beta\calf$ and $\gamma'$ by $\beta\gamma'\beta^{-1}$ 
if necessary, we can assume that $\gee=1$ (i.e., 
$\Psi(\gamma')$ permutes cyclically the three cosets 
$\eta_iH_0\in{}H/H_0$). Since $\beta\in N_{\Aut(S)}(\Inn(S)M)$, we still 
have $\outf(S)=\gen{[\eta_{12}^S]}$ or $\gen{[\eta_{12}^S],[\eta_3^S]}$. In 
particular, $\eta_1\eta_2=\eta_{12}^S|_{R_0}\in\autf(R_0)$.

Set $\autf^*(R_0)=\autf(R_0)\cap{}H$ for short. Since $1\ne 
\eta_1\eta_2H_0\in\autf^*(R_0)H_0/H_0$, and each nontrivial 
subgroup of $H/H_0$ normalized by $\Psi(\gamma')$ contains 
the coset $\eta_1\eta_2\eta_3H_0$ (since 
$C_{H/H_0}(\Psi(\gamma'))=\gen{\eta_1\eta_2\eta_3H_0}$), there is 
$\eta_3'\in\autf^*(R_0)\cap\eta_3H_0$. Upon replacing $\eta_3'$ by some 
appropriate power of $\eta_3'$, if necessary, we can assume $|\eta_3'|=3$. 
Since $\gamma'$ normalizes $\autf^*(R_0)$, there are elements 
$\eta_i'\in\autf^*(R_0)\cap\eta_iH_0$ of order $3$ for $i=1,2$. Thus 
$\autf(R_0)\ge\Inn(R_0)\gen{\eta'_1,\eta'_2,\eta'_3,\gamma',c_{\tauu}}$, 
with equality because $H_0\cap\autf(R_0)=\Inn(R_0)$. 

Since $c_{\tauu}$ normalizes 
$\autf^*(R_0)=\Inn(R_0)\gen{\eta'_1,\eta'_2,\eta'_3}$, conjugation 
by $[c_{\tauu}]$ exchanges the classes $[\eta'_1]$ and $[\eta'_2]$ in 
$\outf(R_0)$ and centralizes $[\eta'_3]$. In particular, $\eta'_1\eta'_2$ 
and $\eta'_3$ both normalize $\Aut_S(R_0)$, and hence by the extension 
axiom, both extend to elements of $\autf(S)$.  Thus 
$\outf(S)=\gen{[\eta_{12}^S],[\eta_3^S]}$, and so 
$\eta_1\eta_2,\eta_3\in\autf(R_0)$ by restriction. Since 
$H_0\cap\autf(R_0)=\Inn(R_0)$, it follows that 
$[\eta'_1\eta'_2]=[\eta_1\eta_2]$ and $[\eta'_3]=[\eta_3]$ in 
$\Out^0(R_0)\cong(\Sigma_4)^3$. 

Recall that the $[\eta_i]$ and $[\eta'_i]$ ($i=1,2,3$) all lie in 
$H/\Inn(R_0)\cong(A_4)^3$. For each $i$, there are four subgroups of order 
$3$ in $H_i/\Inn(R_0)\cong A_4\times (C_2)^4$, and they generate a subgroup 
$H_i^*\cong A_4$. Thus $[\eta_i],[\eta'_i]\in{}H_i^*$ for $i=1,2$. Also, 
$H_1^*\cap{}H_2^*=1$ and $[\eta_1\eta_2]=[\eta'_1\eta'_2]$, and we conclude 
that $[\eta_i]=[\eta'_i]\in\Out(R_0)$ for each $i=1,2,3$. In particular, 
$\eta_1,\eta_2,\eta_3\in\autf(R_0)$. 

Set 
	\[ K=\Gen{[\eta_1],[\eta_2],[\eta_3]} = \autf^*(R_0)/\Inn(R_0)
	\le \outf(R_0). \]
Both $[\gamma]$ and $[\gamma']$ permute the elements $[\eta_i]$ cyclically 
under conjugation, and $\gamma|_{Z(R_0)}=\gamma'|_{Z(R_0)}$ by assumption. 
So $[\gamma^{-1}\gamma']\in C_{\Out(R_0)}(K)$, where $C_{\Out(R_0)}(K)=K$ 
by \eqref{e:6.4b} and since a 3-cycle is self-centralizing in 
$\Sigma_4$. Hence $\gamma\in\autf(R_0)$, and so 
	\[ \autf(R_0) =\Inn(R_0)\gen{\eta_1,\eta_2,\eta_3,\gamma,c_{\tauu}},
	\quad\textup{and}\quad 
	\autf(S)=\Inn(S)\gen{\eta_{12}^S,\eta_3^S}. \]
Also, by \eqref{e:Out(X)} and since $[\eta_3^S|_{\EX}]\in\Out(\EX)$ 
corresponds to a $3$-cycle in $\Sigma_8$, $\outf(\EX)\cong{}A_5\times{}C_3$ 
is the unique subgroup of this isomorphism type in $\Out(\EX)\cong\Sigma_8$ 
which contains $\Out_S(\EX)$, $[\eta_{12}^S|_{\EX}]$, and 
$[\eta_3^S|_{\EX}]$.

We have now shown that each saturated fusion system $\calf^*$ over $S$ with 
$O_2(\calf^*)=1$ is isomorphic to $\calf$.  In particular, if $\calf^*$ is 
the fusion system of $\PSp_6(3)$, then $O_2(\calf^*)=1$ by Proposition 
\ref{p:FS(G)-red}(b), and so $\calf\cong\calf^*$.
\end{proof}

\section{Fusion systems over a $2$-group of type $A_{12}$}
\label{s:a12}

\newcommand{\bbb}{\calb}

Throughout this section, we fix the following notation for certain elements 
of $A_{12}$:
	\begin{align*}  
	a_1&=(1\,2)(3\,4) &
	a_2&=(5\,6)(7\,8) &
	a_3&=(9\,10)(11\,12) \\
	b_1&=(1\,3)(2\,4) &
	b_2&=(5\,7)(6\,8) &
	b_3&=(9\,11)(10\,12) \\
	\mu_{12}&=(1\,2)(5\,6) &
	\mu_{23}&=(5\,6)(9\,10) &
	\tau&=(1\,5)(2\,6)(3\,7)(4\,8) 
	\end{align*}
and set $S=\Gen{a_1,a_2,a_3,b_1,b_2,b_3,\mu_{12},\mu_{23},\tau}\in 
\syl2{A_{12}}$.  We need to consider the following subgroups of $S$:
	\begin{align*}
	A&=\gen{a_1,a_2,a_3,b_1,b_2,b_3}\cong{}(C_2)^6 & 
	Q&=\gen{a_1,a_2,b_1b_2,\mu_{12},\tau}\cong 2^{1+4}_+ \\
	N_1&=A\gen{\mu_{12},\tau}\cong
	((C_2)^4\sd{}(\klfour))\times{}\klfour & H_1&=A\gen{\tau} \\
	N_2&=A\gen{\mu_{12},\mu_{23}} & H_2&=A\gen{\mu_{23}} \\
	N_3&=Q\gen{a_3,b_3,\mu_{23}} & N_{13}&=N_1\cap 
	N_3=Q\times\gen{a_3,b_3}.
	\end{align*}
Note that $A$, $Q$, $N_1$, $N_2$, and $N_3$ are all normal in $S$.

\begin{Lem} \label{l:a12:char}
Assume the above notation. Then $A$ is the unique elementary abelian 
subgroup of rank $6$ in $S$, and hence is characteristic in all subgroups 
which contain it. The subgroups $N_1$, $N_2$, $N_3$, and 
$N_{13}$ are all characteristic in $S$, and $N_{13}$ is characteristic in 
$N_1$ and in $N_3$.
\end{Lem}

\begin{proof} Assume $A$ is not unique: 
let $P<S$ be such that $P\ne A$ and $P\cong (C_2)^6$. Then $P\cap A\le 
C_A(PA/A)$, and hence $\rk(C_A(PA/A))\ge6-\rk(PA/A)$. It is straightforward 
to check that $S/A\cong D_8$, that $\rk(C_A(g))=4$ for $g\in S/A$ of order 
$2$, and that $\rk(C_A(V))=3$ if $V<S/A$ and $V\cong \klfour$. So 
there is no such subgroup $P$, and $A$ is the unique elementary abelian 
subgroup of rank $6$. In particular, $A$ is characteristic in every 
subgroup of $S$ that contains it.

If $N<S$ is such that $N\ge A$ and $N/A\cong \klfour$, then 
$N=N_1$ or $N_2$. Since $N_1\not\cong N_2$ ($[N_2,N_2]=Z(N_2)$ while 
$[N_1,N_1]\ne Z(N_1)$), both subgroups are characteristic in $S$. 

Set $Z=\gen{a_1a_2}=[N_1,N_1]\cap Z(N_1)$: a subgroup characteristic in 
$N_1$ and in $S$. Then $N_{13}/Z\cong (C_2)^6$, $S/N_{13}\cong \klfour$, 
$\rk(C_{N_{13}/Z}(g))\le4$ for each $g\in S\sminus N_{13}$, and 
$\rk(C_{N_{13}/Z}(S))=2$. So $N_{13}/Z$ is the unique elementary abelian 
subgroup of $S/Z$ of rank $6$, and $N_{13}$ is characteristic in $N_1$ and 
in $S$. Also, the three subgroups of $S$ containing $N_{13}$ with index $2$ 
($N_1$, $N_3$, and $N_{13}\gen{b_1\mu_{23}}$) are pairwise nonisomorphic 
--- they have commutator subgroups $\gen{a_1,a_2,b_1b_2}$, 
$\gen{a_1,a_2,a_3,\mu_{12}}$, and $\gen{a_1,a_2,a_3,b_1b_2\mu_{12}}$, 
respectively --- and hence all three are characteristic in $S$.

Since $N_{13}/Z(N_3)\cong{}(C_2)^5$ is the unique abelian subgroup of index 
two in $N_3/Z(N_3)$, $N_{13}$ is characteristic in $N_3$. 
\end{proof}

\begin{Lem} \label{Aut(N3)}
Assume the above notation.
\begin{enuma}

\item There is an automorphism $\nu_3\in\Aut(N_3)$ of order $3$ which takes 
values as follows:
	\[ \renewcommand{\arraystretch}{1.5}
	\begin{array}{|c||c|c|c|c|c|c|c|c|}
	\hline
	g & a_1 & a_2 & a_3 & b_3 & b_1b_2 & \mu_{12} & 
	\mu_{23} & \tau \\
	\hline\hline
	\nu_3(g) & a_2\mu_{12} & a_1\mu_{12} & a_3 & b_3 & 
	a_1a_2b_1b_2\tau & a_2 & \mu_{23} & b_1b_2 \\ \hline
	\end{array} \]
The action of $\nu_3$ on 
	\[ Q=\gen{a_1b_1b_2\mu_{12},a_1\tau}\times_{\gen{a_1a_2}} 
	\gen{a_1b_1b_2\tau,b_1b_2\mu_{12}} \cong Q_8\times_{C_2}Q_8 \]
has order three on each factor. Also, $\nu_3$ commutes in $\Aut(Q)$ with 
$c_{\mu_{23}}$, where $c_{\mu_{23}}$ exchanges the two quaternion factors.

\item $\Aut(N_3)/O_2(\Aut(N_3))\cong\Sigma_3$, generated by the classes of 
$\nu_3$ and $c_{b_1}$. 

\item If $\alpha\in\Aut(N_3)$ has order $3$, then 
$\alpha|_{Z(N_{13})}=\Id$, and $\alpha$ acts on $N_{13}/Z(N_{13})\cong 
Q/Z(Q)$ with $C_{N_{13}/Z(N_{13})}(\alpha)=1$. If in addition, 
$\alpha(Q)=Q$, then $\gen{\9\varphi\alpha}=\gen{\nu_3}$ for some 
$\varphi\in\Aut_Q(N_3)$.

\item Let $\Delta\le\Aut(N_3)$ be such that $\Delta>\Aut_S(N_3)$ and 
$\Delta/\Inn(N_3)\cong\Sigma_3$. Then there is $\varphi\in\Aut(S)$ such 
that $\9\varphi\Delta=\gen{\Aut_S(N_3),\nu_3}$.

\end{enuma}
\end{Lem}

\begin{proof} \noindent\textbf{(a) } The first two statements are 
easily checked. Also, $\bigl[\nu_3|_Q,c_{\mu_{23}}\bigr]=\Id$ in $\Aut(Q)$ 
since $\nu_3$ is a homomorphism and $\nu_3(\mu_{23})=\mu_{23}$. 

\smallskip

\noindent\textbf{(b) } Consider the chain of subgroups $P_0<P_1<P_2<N_3$, where 
	\[ P_0=\Phi(N_3)=\gen{a_1,a_2,a_3,\mu_{12}},\quad
	P_1=P_0Z(N_{13})=P_0\gen{b_3},\quad
	P_2=N_{13}=P_1\gen{b_1b_2,\tau}. \]
Since $N_{13}$ is characteristic in $N_3$ by Lemma \ref{l:a12:char}, each 
of the $P_i$ is characteristic in $N_3$. So by Lemma \ref{l:mod-Fr}, the 
kernel of the induced homomorphism 
$\Aut(N_3)\Right2{}\Aut(P_2/P_1)\cong\Sigma_3$ is contained in 
$O_2(\Aut(N_3))$. Also, the images of $\nu_3$ and $c_{b_1}$ in 
$\Aut(P_2/P_1)$ generate this group, so 
$\Aut(N_3)/O_2(\Aut(N_3))\cong\Sigma_3$. 

\smallskip

\noindent\textbf{(c) } Assume $\alpha\in\Aut(N_3)$ has order $3$. 
Since $\alpha$ normalizes the chain 
	\[ \underset{=\gen{a_1a_2}}{[N_{13},N_{13}]} < 
	\underset{=\gen{a_1a_2,a_3}}{Z(N_3)} < 
	\underset{=\gen{a_1a_2,a_3,b_3}}{Z(N_{13})} \]
of characteristic subgroups, $\alpha|_{Z(N_{13})}=\Id$ by Lemma 
\ref{l:mod-Fr}. 

Consider the subgroups $P_1=Z(N_{13})\gen{a_2,\mu_{12}}$ and 
$N_{13}=P_1\gen{b_1b_2,\tau}$. Since 
	\[ [\mu_{23},b_1b_2]=a_2  \quad\textup{and}\quad 
	[\mu_{23},\tau]=\mu_{12} \,, \]
$[\mu_{23},-]$ sends $N_{13}/P_1$ isomorphically to $P_1/Z(N_{13})$, and 
this isomorphism commutes with $\alpha$ since 
$\alpha(\mu_{23})\in\mu_{23}N_{13}$. So $\alpha$ induces an automorphism of 
order $3$ on $P_1/Z(N_{13})$ since it induces an automorphism of order $3$ 
on $N_{13}/P_1$ by the proof of (b). Thus $\alpha$ acts on 
$N_{13}/Z(N_{13})\cong Q/Z(Q)$ with $C_{N_{13}/Z(N_{13})}(\alpha)=1$.

Now assume in addition that $\alpha(Q)=Q$. Then $C_Q(\alpha)=Z(Q)$, so 
$\alpha$ and $\nu_3$ both act nontrivially on each of the quaternion 
factors of $Q$. They also commute with $c_{\mu_{23}}|_Q$ modulo 
$\Aut_{N_{13}}(Q)=\Inn(Q)$, where $c_{\mu_{23}}$ exchanges the two 
quaternion factors of $Q$. Thus $\alpha|_Q$ is 
congruent to $\nu_3|_Q$ or $\nu_3^{-1}|_Q$ modulo $\Inn(Q)$. Hence 
$\gen{\alpha|_Q}$ and $\gen{\nu_3|_Q}$ are Sylow 3-subgroups of 
$\Inn(Q)\gen{\alpha|_Q}$, so upon replacing $\alpha$ by $\alpha^{-1}$ if 
necessary, we can assume that $\9\varphi\alpha|_Q=\nu_3|_Q$ for some 
$\varphi\in\Aut_Q(N_3)$. We already showed that $\alpha|_{Z(N_{13})}=\Id$, 
and so $\9\varphi\alpha|_{N_{13}}=\nu_3|_{N_{13}}$. 

Thus $\9\varphi\alpha=\beta\circ\nu_3$ for some $\beta\in\Aut(N_3)$ such that 
$\beta|_{N_{13}}=\Id$. Then $\beta(\mu_{23})=g\mu_{23}$ for some 
$g\in{}Z(N_{13})$, and in particular, $\beta$ has order at most $2$ and 
commutes with $\nu_3$. Thus $\beta=\Id$ since $\beta\circ\nu_3$ has order 
$3$, and hence $\9\varphi\alpha=\nu_3$. 

\smallskip

\noindent\textbf{(d) } Set $T_3=\gen{a_3,b_3}\nsg S$; thus $N_{13}=Q\times 
T_3$.  For each $\chi\in\Hom(Q,T_3)$, set 
$Q_\chi=\{g\chi(g)\,|\,g\in{}Q\}$. Each subgroup of $N_{13}=Q\times T_3$ 
isomorphic to $Q$ is sent isomorphically to $Q$ by projection to the first 
factor, and hence is equal to $Q_\chi$ for some unique $\chi$.  

Let $\calq\supseteq\calq_0$ be the sets of subgroups of $N_{13}$ isomorphic 
to $Q$ which are normal in $N_3$ or in $S$, respectively. For example, 
$Q\in\calq_0$. Then 
	\[ \calq = \{Q_\chi\,|\,\chi\in\Hom_{\Aut_{N_3}(N_{13})}(Q,T_3)\} 
	\qquad\textup{and}\qquad
	\calq_0 = \{Q_\chi\,|\,\chi\in\Hom_{\Aut_{S}(N_{13})}(Q,T_3)\}, \]
where $\Hom_X(Q,T_3)$ denotes the set of group homomorphisms 
from $Q$ to $T_3$ which commute with the action of $X$.

If $\chi\in\Hom(Q,T_3)$ is such that $\nu_3(Q_\chi)=Q_\chi$, then 
$\chi(\nu_3(g))=\chi(g)$ for each $g\in Q$ (recall $\nu_3|_{T_3}=\Id$), and 
$\chi=1$ since $Q=[\nu_3,Q]$. Thus $Q$ is the only member of $\calq$ 
normalized by $\nu_3$. 

Fix $\Delta\le\Aut(N_3)$ such that $\Delta>\Aut_S(N_3)$ and 
$\Delta/\Inn(N_3)\cong\Sigma_3$. Choose $\alpha\in\Delta$ of order three. 
Since $\Aut(N_3)/O_2(\Aut(N_3))\cong\Sigma_3$ by (b), there is $\psi\in 
O_2(\Aut(N_3))$ such that $\gen{\9\psi\alpha}=\gen{\nu_3}$. Let $\chi$ be 
such that $Q_\chi=\psi^{-1}(Q)\in\calq$. Since $Q$ is the only member of 
$\calq$ normalized by $\nu_3$, $Q_\chi$ is the only member normalized by 
$\alpha$. Also, $c_{b_1}$ normalizes $\Inn(N_3)\gen{\alpha}$, 
so $\9{b_1}Q_\chi=Q_\chi$ by the uniqueness of $Q_\chi$, and 
hence $Q_\chi\nsg S$ and $Q_\chi\in\calq_0$. In particular, $\chi$ 
commutes with the action of $\Aut_S(N_{13})$.

We claim that there is $\varphi\in\Aut(S)$ such that $\varphi|_A=\Id$, 
$\varphi(\mu_{23})=\mu_{23}$, $\varphi(\mu_{12})=\mu_{12}\chi(\mu_{12})$, 
and $\varphi(\tau)=\tau\chi(\tau)$. Since $\chi$ commutes with the 
action of $c_{b_1}\in\Aut_S(N_{13})$ and $[b_1,T_3]=1$, we have 
$Q\cap{}A=[b_1,Q]\le\Ker(\chi)$. Thus $\varphi(g)=g\chi(g)$ for all 
$g\in{}Q$. Since $T_3\le Z(N_1)$ (where $N_1=AQ$), $\varphi|_{N_1}$ is a 
well defined automorphism. To see that $\varphi$ is well defined, it 
remains to check that $\chi(\9{\mu_{23}}g)=\9{\mu_{23}}\chi(g)$ for all 
$g\in Q$, and this holds since $\chi$ commutes with the action of 
$\Aut_S(N_{13})$. 

By construction, $\varphi(Q_\chi)=Q$. So upon replacing $\Delta$ by 
$\9\varphi\Delta$ and $\alpha$ by $\9\varphi\alpha$, we can assume that 
$\alpha(Q)=Q$. By (c), $\gen{\9\eta\alpha}=\gen{\nu_3}$ for some 
$\eta\in\Aut_Q(N_3)$. Thus $\alpha\equiv\nu_3^{\pm1}$ (mod $\Inn(N_3)$), so 
$\Delta=\gen{\Aut_S(N_3),\alpha}=\gen{\Aut_S(N_3),\nu_3}$. 
\end{proof}

\begin{Lem} \label{Aut(N2)}
Assume the above notation. Set $A_0=Z(N_2)=\gen{a_1,a_2,a_3}$, and let 
	\[ R_0\:\Aut(N_2)\Right4{}\Aut(A_0) \]
be the homomorphism induced by restriction. 
\begin{enuma} 

\item There is an automorphism $\nu_2\in\Aut(N_2)$ of order $3$ which takes 
values as follows:
	\[ \renewcommand{\arraystretch}{1.5}
	\begin{array}{|c||c|c|c|c|c|c|c|c|}
	\hline
	g & a_1 & b_1 & a_2 & b_2 & a_3 & b_3 & \mu_{12} & 
	\mu_{23} \\
	\hline\hline
	\nu_2(g) & a_2 & b_2 & a_3 & b_3 & a_1 & b_1 & \mu_{23} & 
	\mu_{12}\mu_{23} \\ \hline
	\end{array} \]

\item For each $\beta\in\Ker(R_0)$, $\beta$ induces the identity on 
$N_2/A_0$. Furthermore, $\Ker(R_0)=O_2(\Aut(N_2))$ is an elementary abelian 
$2$-group, and $\Im(R_0)=\gen{\nu_2|_{A_0},c_\tau|_{A_0}}\cong\Sigma_3$. 

\item Let $\Delta\le\Aut(N_2)$ be such that $\Delta>\Aut_S(N_2)$ and 
$\Delta/\Inn(N_2)\cong\Sigma_3$. Then there is $\xi\in\Aut(S)$ such that 
$\9\xi\Delta=\gen{\Aut_S(N_2),\nu_2}$ and $\xi(\tau)=\tau$.

\end{enuma}
\end{Lem}

\begin{proof} Throughout the proof, we set $K=\Ker(R_0)$.

\smallskip

\noindent\textbf{(a) } This is easily checked.

\smallskip

\noindent\textbf{(b) } Fix $\beta\in K$. 
Then $\beta(A)=A$ since $A$ is characteristic in $N_2$ by Lemma 
\ref{l:a12:char}. The commutators 
	\[ [\mu_{12},A]=\gen{a_1,a_2}, \qquad 
	[\mu_{23},A]=\gen{a_2,a_3}, \qquad\textup{and}\qquad
	[\mu_{12}\mu_{23},A]=\gen{a_1,a_3} \]
are all distinct, so for $x,y\in N_2$ such that $[x,A]=[y,A]$, we have 
$x\equiv y$ (mod $A$). The equalities $[\beta(x),A]=\beta([x,A])=[x,A]$ for 
$x\in N_2$ now show that $\beta$ induces the identity on $N_2/A$. Thus 
$\beta$ normalizes $C_A(\mu_{12})=A_0\gen{b_3}$, 
$C_A(\mu_{23})=A_0\gen{b_1}$, and $C_A(\mu_{12}\mu_{23})=A_0\gen{b_2}$, and 
hence induces the identity on $A/A_0$. 

Let $h_{12},h_{23}\in A$ be such that $\beta(\mu_{ij})=\mu_{ij}h_{ij}$. 
Then $h_{12}\in C_A(\mu_{12})=A_0\gen{b_3}$ since $(\mu_{12}h_{12})^2=1$, 
and $h_{23}\in A_0\gen{b_1}$ and $h_{12}h_{23}\in A_0\gen{b_2}$ by similar 
arguments. Hence $h_{12},h_{23}\in A_0$, and so $\beta$ induces the 
identity on $N_2/A_0$. Also, $\beta\in O_2(\Aut(N_2))$ by Lemma 
\ref{l:mod-Fr}.

Thus $K\le O_2(\Aut(N_2))$. 
Each $\alpha\in\Aut(N_2)$ permutes the subgroups $[x,A]$ for 
$x\in\{\mu_{12},\mu_{23},\mu_{12}\mu_{23}\}$; and thus permutes their 
pairwise intersections $\gen{a_i}$ for $i=1,2,3$. So 
$\Im(R_0)\cong\Sigma_3$, generated by $\nu_2|_{A_0}$ and $c_{\tau}|_{A_0}$, 
and $K=O_2(\Aut(N_2))$. Also, $K$ is elementary abelian, 
since each $\beta\in K$ has the form $\beta(g)=g\chi(g)$ for some 
$\chi\in\Hom(N_2,A_0)\cong (C_2)^{15}$ (and the resulting bijection 
$K\cong\Hom(N_2,A_0)$ is an isomorphism).

\smallskip

\noindent\textbf{(c) } Let 
$\Delta\le\Aut(N_2)$ be a subgroup such that $\Delta>\Aut_S(N_2)$ and 
$\Delta/\Inn(N_2)\cong\Sigma_3$. Thus $\Delta\cap K=\Inn(N_2)$. By 
Proposition \ref{p:splittings}, applied with $\Out(N_2)$, $K/\Inn(N_2)$, 
$\gen{\Out_S(N_2),[\nu_2]}$, and $\Out_S(N_2)$ in the roles of $G$, $Q$, 
$H$, and $H_0$, there is $\xi_0\in K$ such that 
$[\xi_0,c_\tau]\in\Inn(N_2)$ and $\9{\xi_0}\Delta=\gen{\nu_2,\Aut_S(N_2)}$. 
Since $K=C_K(\nu_2)\times[\nu_2,K]$ (see \cite[Theorem 5.2.3]{Gorenstein}) and 
both factors are normalized by $c_\tau$, we can choose $\xi_0\in[\nu_2,K]$.

Now, $[\nu_2,\Inn(N_2)]$ has an $\F_2$-basis 
$\bigl\{c_{b_1b_3},c_{b_2b_3},c_{\mu_{12}\mu_{23}},c_{\mu_{23}}\bigr\}$ 
permuted freely by $c_\tau$. Hence each element of 
$[\nu_2,K]/[\nu_2,\Inn(N_2)]$ which is centralized by $c_\tau$ lifts to an 
element of $[\nu_2,K]$ which commutes with $c_\tau$ in $\Aut(N_2)$. In 
particular, there is $\xi_1\equiv\xi_0$ (mod $[\nu_2,\Inn(N_2)]$) such that 
$\xi_1\in C_K(c_\tau)$. Since $\xi_1$ commutes with $c_\tau$ in 
$\Aut(N_2)$, it extends to an element $\xi\in\Aut(S)$ such that 
$\xi(\tau)=\tau$, and we still have 
$\9{\xi}\Delta=\gen{\nu_2,\Aut_S(N_2)}$. 
\end{proof}

The following proposition is essentially a special case of the main theorem 
in Ron Solomon's paper \cite{Solomon-A12}, where he lists the finite simple 
groups (more generally, the ``fusion-simple'' groups) with Sylow 
$2$-subgroups isomorphic to $S$.  (See also \cite[Theorem 
1.1]{Solomon-O7}.)  But our method of proof, based on analysis of the 
possible essential subgroups, is somewhat different. 

\begin{Prop} \label{p:A12}
Let $S$ be a $2$-group of type $A_{12}$. Then every 
reduced fusion system over $S$ is isomorphic to the fusion system of one of 
the groups $A_{12}$, $\Sp_6(2)$, or $\Omega_7(3)$.  The 
subgroups $N_1$, $N_2$, and $N_3$ are essential in all three fusion 
systems, $H_1$ is essential in the fusion system of $A_{12}$, and these are 
the only essential subgroups up to $S$-conjugacy.
\end{Prop}

\begin{proof} By a computer search, the only potentially critical subgroups 
of $S$ (up to conjugacy) are $N_1$, $N_2$, $N_3$, $H_1$, and $H_2$.  Of 
these, the $N_i$ are characteristic of index two in $S$ (Lemma 
\ref{l:a12:char}), while $N_S(H_i)=N_i$ for $i=1,2$.  Also, $\Aut(S)$ is a 
2-group by Lemma \ref{l:mod-Fr}, applied to the sequence 
$\Phi(S)<N_1\cap{}N_2\cap{}N_3<N_1\cap{}N_2<N_1<S$ of characteristic 
subgroups of $S$. 

Assume $\calf$ is a reduced fusion system over $S$.  Since $\Out(S)$ is a 
$2$-group, $\outf(S)=1$.  By Lemma \ref{l:foc=S}, 
	\beqq S = \Gen{[\autf(P),P] \,\big|\, P\in\EE_\calf}.
	\label{e:7.1a} \eeqq

\smallskip

\noindent\boldd{Step 1: \ $\EE_\calf\supsetneqq\{N_2,N_3\}$. } By Lemma 
\ref{l:a12:char}, $A$ is a characteristic subgroup of each 
subgroup of $S$ that contains $A$. Since $N_3$ is the only 
potentially critical subgroup that does not contain $A$, it must be 
essential, since otherwise $A\nsg\calf$ by Proposition \ref{Q<|F}.

Since $N_{13}$ is characteristic of index $2$ in $N_3$ by Lemma 
\ref{l:a12:char}, $[\autf(N_3),N_3]\le{}N_{13}\le N_1$.  Since $A$ is 
characteristic of index two in $H_2$ by Lemma \ref{l:a12:char}, 
$[\autf(H_2),H_2]\le A\le{}N_1$. So $N_2\in\EE_\calf$, since otherwise 
$[\autf(P),P]\le N_1$ for all $P\in\EE_\calf$, contradicting 
\eqref{e:7.1a}.

Since $\Aut(N_3)/O_2(\Aut(N_3))\cong\Sigma_3$ by Lemma \ref{Aut(N3)}(b), 
$\outf(N_3)\cong\Sigma_3$. By Lemma \ref{Aut(N3)}(d), there is 
$\varphi\in\Aut(S)$ such that 
$\9\varphi\autf(N_3)=\gen{\Aut_S(N_3),\nu_3}$. So upon replacing $\calf$ by 
$\9\varphi\calf$, we can assume that $\autf(N_3)=\gen{\Aut_S(N_3),\nu_3}$.

Since $\Aut(N_2)/O_2(\Aut(N_2))\cong\Sigma_3$ by Lemma \ref{Aut(N2)}(b), 
$\outf(N_2)\cong\Sigma_3$. By Lemma \ref{Aut(N2)}(c), there is 
$\xi\in\Aut(S)$ such that $\xi(\tau)=\tau$ and 
$\9\xi\autf(N_2)=\gen{\nu_2,\Aut_S(N_2)}$. Upon replacing $\calf$ by 
$\9\xi\calf$, we can assume that $\autf(N_2)=\gen{\nu_2,\Aut_S(N_2)}$. 
Also, $\xi$ normalizes $Q=\Gen{\tau,[\tau,S]}$, so 
$(\9\xi\nu_3)(Q)=Q$, and $\9\xi\nu_3\in\Inn(N_3)\gen{\nu_3}$ by Lemma 
\ref{Aut(N3)}(c). So we still have $\autf(N_3)=\gen{\Aut_S(N_3),\nu_3}$.

By a direct computation, $[\nu_3,N_3]=Q$, and hence 
	\beqq [\autf(N_2),N_2][\autf(N_3),N_3] = 
	\Gen{a_1,a_2,a_3,b_1b_2,b_2b_3,\mu_{12},\mu_{23},\tau} <S\,. 
	\label{foc(N2N3)} \eeqq
So by \eqref{e:7.1a}, $\EE_\calf\supsetneqq\{N_2,N_3\}$.

\smallskip

\noindent\boldd{Step 2: \ $H_2\notin\EE_\calf$, and at least one of $N_1$, 
$H_1$ lies in $\EE_\calf$. } Since $\nu_2\in\autf(N_2)$, $H_2$ is 
$\calf$-conjugate to $\nu_2^{-1}(H_2)=A\gen{\mu_{12}}$.  Since 
$A\gen{\mu_{12}}$ is normal in $S$, $H_2$ is not fully normalized in 
$\calf$, and hence is not $\calf$-essential. Thus at least one of the 
subgroups $N_1$ or $H_1$ is $\calf$-essential.  

\smallskip

\noindent\boldd{Step 3: \ $\autf(N_1)$ normalizes $Q$. } Set 
$T_3=\gen{a_3,b_3}$, so that $N_{13}=N_1\cap{}N_3=Q\times T_3$. Recall that 
$N_{13}$ is characteristic in $N_1$ by Lemma \ref{l:a12:char}.

By \cite[Proposition 3.2(b,c)]{O-split}, 
$\Out(N_{13})/O_2(\Out(N_{13}))\cong\Out(Q)\times\Aut(T_3)\cong 
(\Sigma_3\wr C_2)\times\Sigma_3$. The automorphisms $c_{b_1}$ and 
$c_{\mu_{23}}$ both act nontrivially on $Q/Z(Q)$, and only the second 
acts nontrivially on $T_3$. Hence 
$\Out_S(N_{13})=\gen{[c_{b_1}],[c_{\mu_{23}}]}$ embeds into the quotient 
group $\Out(N_{13})/O_2(\Out(N_{13}))$. Since $N_{13}\nsg S$, 
$\Out_S(N_{13})\in\syl2{\outf(N_{13})}$, and so $\outf(N_{13})\cap 
O_2(\Out(N_{13}))=1$. Hence $\outf(N_{13})$ also embeds into $(\Sigma_3\wr 
C_2)\times\Sigma_3$. 

Thus any two elements of odd order in $\outf(N_{13})$ commute.  So if 
$\alpha\in\autf(N_1)$ has odd order, then 
$\bigl[\alpha|_{N_{13}},\nu_3|_{N_{13}}\bigr]\in\Inn(N_{13})$.  Since 
$[N_{13},N_{13}]\le Q$, this shows that $\alpha$ normalizes 
$Q=[\nu_3,N_{13}]$. So $O^2(\autf(N_1))$ normalizes $Q$, 
$\Aut_S(N_1)\in\syl2{\autf(N_1)}$ normalizes $Q$ since $Q\nsg S$, and hence 
$\autf(N_1)$ normalizes $Q$. 

We claim that in fact, 
	\beqq \autf(N_1) = \bigl\{\alpha\in\Aut(N_1) \,\big|\, 
	\alpha|_A\in\autf(A),~ \alpha(Q)=Q \bigr\} .
	\label{autN1} \eeqq
Since $A$ is characteristic in $N_1$ by Lemma \ref{l:a12:char}, 
$\autf(N_1)$ normalizes $A$ as well as $Q$, and hence $\autf(N_1)$ 
is contained in the right hand side of \eqref{autN1}.  

If $\alpha\in\Aut(N_1)$ is such that $\alpha|_A\in\autf(A)$ and 
$\alpha(Q)=Q$, then $\alpha|_A$ normalizes $\Aut_{N_1}(A)$.  So by the 
extension axiom (and since $C_S(A)=A\le N_1$), there is 
$\beta\in\autf(N_1)$ such that $\beta|_A=\alpha|_A$.  Set 
$P=\gen{a_1,b_1,a_2,b_2,\tau,\mu_{12}}=Q\gen{b_1}$.  Then 
$\alpha^{-1}\beta$ induces the identity on $A$, hence on $P\cap{}A$, and it 
normalizes $P$ since it normalizes $Q$.  So by Lemma 
\ref{l:res.aut.}, and since $P\cap{}A$ is centric in $P$ and 
$\Aut_P(P\cap A)=\gen{c_\tau,c_{\mu_{12}}}$ 
permutes freely the basis $\{b_1,a_1b_1,b_2,a_2b_2\}$ of $P\cap{}A$, 
we have $(\alpha^{-1}\beta)|_P\in\Inn(P)$.  Let $x\in{}P$ be such that 
$(\alpha^{-1}\beta)|_P=c_x|_P$; then $x\in{}C_P(P\cap{}A)=P\cap{}A$ since 
$(\alpha^{-1}\beta)|_A=\Id$.  Hence $(\alpha^{-1}\beta)|_A=\Id=c_x|_A$, and 
since $N_1=AP$, we have $\alpha^{-1}\beta=c_x\in\Inn(N_1)$.  Thus 
$\alpha\in\autf(N_1)$, and this finishes the proof of \eqref{autN1}.

\smallskip

\noindent\boldd{Step 4: \ The group $\autf(A)$. } Set 
	\beqq \Gamma=\autf(A) \qquad\textup{and}\qquad 
	\Gamma_0=\gen{\Aut_S(A),\nu_2|_A} \cong \Sigma_4\,. \label{e:Gamma} 
	\eeqq
Thus $O_2(\Gamma_0)=\gen{c_{\mu_{12}},c_{\mu_{23}}}$, and $\Aut_S(A)\cong 
D_8$ is a Sylow 2-subgroup of $\Gamma$. We will show that 
	\beqq \Gamma \cong A_7,\quad \GL_3(2), \quad\textup{or}\quad
	(C_3)^3\sd{}\Sigma_4 \,. \label{e:3Gammas} \eeqq

By the Gorenstein-Walter theorem on groups with dihedral Sylow 2-subgroups 
\cite[Theorem 1]{GW}, $\Gamma/O_{2'}(\Gamma)$ is isomorphic to a subgroup 
of $\PGGL_2(q)$ which contains $\PSL_2(q)$ ($q$ odd), to $A_7$, or to 
$D_8$. In the first case, $|\PSL_2(q)|$ divides $|\Gamma|$ and 
$|\Gamma|$ divides $|\GL_6(2)|$, and since 
$16\nmid|\PSL_2(q)|$, we have $q\le9$. The third case is impossible, 
since $\Gamma_0/O_{2'}(\Gamma_0)\cong\Sigma_4$ is not a subquotient of 
$D_8$. Hence $\Gamma/O_{2'}(\Gamma)$ is isomorphic to one of the groups 
$\PGL_2(3) \cong \Sigma_4$, $\PGL_2(5) \cong \Sigma_5$, $\PSL_2(7) \cong 
\GL_3(2)$, $\PSL_2(9) \cong A_6$, or $A_7$.

Assume $O_{2'}(\Gamma)\ne1$.  By Lemma \ref{GL6(2)}, for some odd prime 
$p$, there is a normal elementary abelian $p$-subgroup $1\ne{}K\nsg\Gamma$, 
and $N_{\Aut(A)}(K)\ge\Gamma\ge\Gamma_0\cong\Sigma_4$. Also, $O_2(\Gamma_0)$ 
cannot centralize $K$, since with respect to the $\F_2$-basis 
$\{a_1,a_2,a_3,b_1,b_2,b_3\}$ of $A$, 
	\beqq 
	C_{\Aut(A)}\bigl(\gen{c_{\mu_{12}},c_{\mu_{23}}}\bigr) = 
	\bigl\{ \mxtwo{I}X0I \,\big|\, X \in M_3(\F_2) \bigr\} 
	\cong (C_2)^9. \label{cent2gp} \eeqq
Thus $O^2(N_{\Aut(A)}(K))\ge O^2(\Gamma_0)\ge O_2(\Gamma_0)$ contains 
involutions which do not centralize $K$, and from the list of normalizers 
in Lemma \ref{GL6(2)}, this is possible only if $K\cong (C_3)^3$ and 
$\Gamma\le N_{\Aut(A)}(K)\cong\Sigma_3\wr\Sigma_3$. Since 
$[\Gamma:\Gamma_0]$ is odd, $\Gamma=K\Gamma_0\cong (C_3)^3\sd{}\Sigma_4$.

Now assume $O_{2'}(\Gamma)=1$. We already showed that 
$\Gamma\cong\Sigma_4$, $\Sigma_5$, $A_6$, $A_7$, or $\GL_3(2)$, and it 
remains to eliminate the first three possibilities. Since $A$ is 
characteristic in all subgroups of $S$ which contain it (Lemma 
\ref{l:a12:char}), $C_P(\autf(P))\ge C_A(\Gamma)$ for all $P\le S$ 
containing $A$. Since $N_3$ is the only $\calf$-essential subgroup that 
does not contain $A$, and 
$C_{N_3}(\autf(N_3))\ge\gen{a_1a_2a_3}=C_A(\Gamma_0)\ge C_A(\Gamma)$ by a 
direct check, Proposition \ref{Q<|F} implies that $C_A(\Gamma)\nsg\calf$, 
and hence that
	\beqq C_A(\Gamma) \le O_2(\calf)=1\,. \label{e:CA(G)=1} \eeqq
In particular, $\Gamma>\Gamma_0$, and $\Gamma\not\cong\Sigma_4$.

We claim that
	\beqq \Gamma_0<\Gamma_1\le\Gamma, \quad \Gamma_1\cong \Sigma_5 
	\textup{ or } A_6 \quad\implies\quad C_A(\Gamma_1)
	=\gen{a_1a_2a_3} \,. \label{e:Gamma1} \eeqq
Fix $\Gamma_1$ as on the left side of \eqref{e:Gamma1}. Assume there is 
$H<\Gamma_1$ such that $H\cong A_5$, $\Gamma_1=\gen{\Gamma_0,H}$, and 
$\Gamma_0\cap H\cong A_4$. Under this assumption, since 
$C_A(\gen{\mu_{12},\mu_{23}})=\gen{a_1,a_2,a_3}=[\gen{\mu_{12},\mu_{23}},A]$, 
and since $\rk([x,A])=\rk([\mu_{12},A])=2$ for each involution $x\in H$, 
Lemma \ref{l:F2[A5]-mod}(b) implies that 
$C_A(H)=C_A(H\cap\Gamma_0)=C_A(\Gamma_0)=\gen{a_1a_2a_3}$, and hence that 
\eqref{e:Gamma1} holds. 

We now check that such an $H$ exists under the hypotheses of 
\eqref{e:Gamma1}. If $\Gamma_1\cong\Sigma_5$, then $H=O^2(\Gamma_1)$ 
satisfies the above conditions.  If $\Gamma_1\cong A_6$, then there is at 
least one pair of subgroups $\Gamma_0^*,H^*<\Gamma$ such that 
$\Gamma_0^*\cong\Sigma_4$, $H^*\cong A_5$, $\Gamma_1=\gen{\Gamma_0^*,H^*}$, 
and $\Gamma_0^*\cap H^*\cong A_4$; and we can take $\Gamma_0^*=\Gamma_0$ 
since the subgroups of $A_6$ isomorphic to $\Sigma_4$ are permuted 
transitively by $\Aut(A_6)$. (This last claim holds by 
\cite[(3.2.20)]{Sz1}, and since the two $A_6$-conjugacy classes of 
subgroups isomorphic to $\Sigma_4$ contain elements of order $3$ in 
different classes.) This finishes the proof of \eqref{e:Gamma1}. By 
\eqref{e:CA(G)=1} and \eqref{e:Gamma1}, $\Gamma\not\cong\Sigma_5$ and 
$\Gamma\not\cong A_6$, so \eqref{e:3Gammas} holds.

\smallskip

\noindent\boldd{Step 5: Three distinct fusion systems over $S$. } Set 
$\5a=a_1a_2a_3$ for short. Consider the automorphisms 
$\sigma_1,\sigma_2\in\Aut(S)$, defined by setting (for $i=1,2,3$): 
	\begin{align*}  
	\sigma_1(a_i)&=a_i & \sigma_1(b_i)&=a_ib_i & 
	\sigma_1(\mu_{12})&=\mu_{12} & \sigma_1(\mu_{23})&=\mu_{23} & 
	\sigma_1(\tau)&=\tau \\ 
	\sigma_2(a_i)&=a_i & \sigma_2(b_i)&=\5ab_i & 
	\sigma_2(\mu_{12})&=\mu_{12} & \sigma_2(\mu_{23})&=\mu_{23} & 
	\sigma_2(\tau)&=\tau. 
	\end{align*} 
For each $j=1,2$, $\sigma_j|_{N_2}$ commutes with $\nu_2$, and 
$\sigma_j|_{N_3}$ commutes with $\nu_3$ modulo $\Inn(N_3)$. Commutativity 
with $\nu_2$ is easily checked. Commutativity with $\nu_3$ modulo 
$\Inn(N_3)$ follows from Lemma \ref{Aut(N3)}(c), applied with 
${}^{\sigma_j}\nu_3$ in the role of $\alpha$. Hence we can replace $\calf$ 
by $\9{\sigma_1}\calf$ or $\9{\sigma_2}\calf$ without changing $\autf(N_2)$ 
or $\autf(N_3)$ (or $\Gamma_0$).

Consider the $\F_2$-basis $\bbb=\{b_1,a_1b_1,b_2,a_2b_2,b_3,a_3b_3\}$ for 
$A$, and set $\5\bbb=\bbb\cup\{\5a\}$.  Each element of $\Gamma_0<\Aut(A)$ 
permutes the elements of $\bbb$ and fixes $\5a$.

We claim that
	\beqq \textup{$H_1$ fully normalized in $\calf$
	\quad$\implies$\quad
	$\outf(H_1)\cong C_\Gamma(c_\tau)/\gen{c_\tau}$.} 
	\label{autH1} \eeqq
By Lemma \ref{l:res.aut.}, restriction induces a homomorphism 
$R\:\outf(H_1)\too{}C_\Gamma(c_\tau)/\gen{c_\tau}$, and $\Ker(R)$ is a 
$2$-group.  Also, $R$ is surjective by the extension axiom.  If $H_1$ is fully 
normalized, then $\Out_S(H_1)=\Gen{[c_{\mu_{12}}]}\in \syl2{\outf(H_1)}$, 
and so $R$ is injective.

\boldd{Assume $\Gamma\cong A_7$.}  Then $\Gamma$ contains two conjugacy classes of 
subgroups isomorphic to $\klfour$, represented by the two subgroups of this 
type in any Sylow 2-subgroup.  Of these, each subgroup in one of the 
classes contains a subgroup of order three in its centralizer, while those 
in the other class have normalizer isomorphic to $\Sigma_4$.  By 
\eqref{cent2gp}, $O_2(\Gamma_0)=\gen{c_{\mu_{12}},c_{\mu_{23}}}$ must be of 
the latter type, so $\Gamma_0=N_\Gamma(\gen{c_{\mu_{12}},c_{\mu_{23}}})$, 
and there is a unique subgroup $\Gamma_1<\Gamma$ such that 
$\Gamma_0<\Gamma_1\cong{}A_6$. By \eqref{e:Gamma1}, 
$\5a\in{}C_A(\Gamma_1)$, and hence lies in a $\Gamma$-orbit $X$ of length 
$|\Gamma/\Gamma_1|=7$. By the above remarks about subgroups of $A_7$, 
$\Gamma_0$ acts on $X$ with two orbits: of lengths $1$ and $6$. By a direct 
check, the action of $\Gamma_0$ on $A$ has exactly two orbits of length 
$6$: the orbits $\bbb$ of $b_1$ and $\sigma_2(\bbb)$ of $\5ab_1$. 
So after replacing $\calf$ by $\9{\sigma_2}\calf$ if necessary, we can 
assume that $X=\5\bbb$, and hence that 
$\Gamma$ is the group of even permutations of $\5\bbb$.  Since $\Gamma$ 
determines $\autf(N_1)$ by \eqref{autN1}, and since $H_1$ is not fully 
normalized (hence not $\calf$-essential) by \eqref{autH1} 
($C_\Gamma(c_\tau)/\gen{c_\tau}\cong C_2\times\Sigma_3$), $\Gamma$ 
determines $\calf$.  More precisely, there is at most one reduced fusion 
system $\calf$ over $S$ for which $\Gamma=\autf(A)$, $\autf(N_2)$, and 
$\autf(N_3)$ are as just described; and any other reduced fusion system 
over $S$ with $\autf(A)\cong A_7$ is isomorphic to it. Also, by Step 2, $N_1$ is 
$\calf$-essential in this case since $H_1$ is not.

\boldd{Assume $\Gamma\cong\GL_3(2)$.}  Since $\5a=a_1a_2a_3$ is fixed by 
$\Gamma_0$ but not by $\Gamma$, it must lie in an orbit $X$ of length 
$7=|\Gamma/\Gamma_0|$.   We claim that any transitive action of $\GL_3(2)$ 
on a set $X$ of $7$ elements (with isotropy subgroups isomorphic to 
$\Sigma_4$) is the group of all automorphisms of $(X\cup\{0\},+)$ under 
some group structure such that $(X\cup\{0\},+)\cong\F_2^3$. To see this, 
note that the isotropy subgroup must be the normalizer of some $P\cong 
\klfour$ (since it is maximal), and hence the stabilizer subgroup of a 1- or 
2-dimensional subspace in $\F_2^3$ (since $P$ is conjugate to one of two 
subgroups in any given Sylow $2$-subgroup). Hence the elements of $X$ are 
in natural correspondence with the elements of $\F_2^3\sminus\{0\}$ or 
$(\F_2^3)^*\sminus\{0\}$, where $(\F_2^3)^*$ denotes the dual vector space. 

Since the six elements of $X\sminus\{\5a\}$ are permuted transitively by 
$\Gamma_0$, we can assume (after replacing $\calf$ by $\9{\sigma_2}\calf$ 
if necessary) that they form the basis $\bbb$, and hence that $\Gamma$ 
permutes the set $\5\bbb$. Thus $\Gamma$ is the group of all automorphisms 
of $(\5\bbb\cup\{0\},+)$ for some group structure such that 
$(\5\bbb\cup\{0\},+)\cong(\F_2)^3$.  Since the fixed set of an automorphism 
is a subgroup, we have $\5a+b_i=a_ib_i$ for each $i=1,2,3$ (corresponding 
to automorphisms in $\gen{c_{\mu_{12}}|_A,c_{\mu_{23}}|_A}$).  Hence there 
are exactly two such structures on which $\Gamma_0$ acts via automorphisms: 
one in which $b_1+b_2=b_3$, and the other in which $b_1+b_2=a_3b_3$.  So we 
can assume, after replacing $\calf$ by $\9{\sigma_1}\calf$ if necessary, 
that $\Gamma$ is the group of automorphisms of $(\5\bbb\cup\{0\},+)$ where 
$b_1+b_2=b_3$. Again in this case, $\Gamma$ determines $\autf(N_1)$ by 
\eqref{autN1}, $H_1$ is not fully normalized (hence not $\calf$-essential) 
by \eqref{autH1} ($C_\Gamma(c_\tau)/\gen{c_\tau}\cong \klfour$), and thus 
$\Gamma$ determines $\calf$. Also, by Step 2, $N_1$ is $\calf$-essential 
since $H_1$ is not.

\boldd{If $\Gamma\cong (C_3)^3\sd{}\Sigma_4$,} then by Lemma \ref{GL6(2)}, 
applied 
with $K\cong (C_3)^3$, there is a decomposition $A=A_1\times A_2\times A_3$ 
such that $\rk(A_i)=2$ for each $i$ and each element of $\Gamma$ permutes the 
subgroups $A_1,A_2,A_3$. The subgroup 
$O_2(\Gamma_0)=\gen{c_{\mu_{12}},c_{\mu_{23}}}<\Aut_S(A)$ normalizes each 
$A_i$ and acts on it nontrivially, so there is a unique basis $\bbb_i$ of 
$A_i$ which is permuted by $O_2(\Gamma_0)$. Then 
$\bbb_1\cup\bbb_2\cup\bbb_3$ is a basis of $A$ which is permuted 
transitively by $\Gamma_0$. After replacing $\calf$ by $\9{\sigma_2}\calf$ 
if necessary, we can assume that $\bbb_1\cup\bbb_2\cup\bbb_3=\bbb$, and 
hence (possibly after permuting the indices) that $\bbb_i=\{b_i,a_ib_i\}$ 
for each $i$. Thus 
$\Gamma=\gen{\Aut_S(A),\nu_2|_A,\alpha_1,\alpha_2,\alpha_3}$, where the 
$\alpha_i$ are defined by setting 
	\beq \alpha_i(a_i)=b_i,\quad \alpha_i(b_i)=a_ib_i,\qquad 
	\alpha_i(a_j)=a_j,\quad \alpha_i(b_j)=b_j \tag{for $j\ne{}i$.} \eeq 
By \eqref{autN1}, $\autf(N_1)$ is uniquely determined by this choice of 
$\Gamma=\autf(A)$.

Since $C_{O_3(\Gamma)}(c_\tau)=\gen{\alpha_1\alpha_2,\alpha_3}$ is not 
isomorphic to $C_{O_3(\Gamma)}(c_{\mu_{12}})=\gen{\alpha_3}$, 
$c_\tau$ and $c_{\mu_{12}}$ are not conjugate in $\Gamma=\autf(A)$, and 
hence $H_1=A\gen{\tau}$ is not $\calf$-conjugate to $A\gen{\mu_{12}}$. 
So $H_1$ is fully normalized in $\calf$. By \eqref{autH1}, 
$\outf(H_1)\cong C_\Gamma(c_\tau)/\gen{c_\tau}\cong C_3\times\Sigma_3$, 
where the quotient is generated by the classes of $\alpha_3$, 
$\alpha_{12}\defeq\alpha_1\alpha_2$, and $c_{\mu_{12}}$. Thus 
$H_1\in\EE_\calf$ and 
	\[ \autf(H_1)=\gen{\Aut_S(H_1),\nu_1|_{H_1},\eta} \]
for some $\nu_1\in\autf(N_1)$ and $\eta\in\autf(H_1)$ of order three such 
that $\nu_1|_A=\alpha_3$ and $\eta|_A=\alpha_{12}$. (Note that for $P=N_1$ 
or $H_1$, $\Ker\bigl[\autf(P)\too\autf(A)\bigr]$ is a $2$-group, by Lemma 
\ref{l:res.aut.} and since $A$ is centric in $P$.)
Also, $[\nu_1,c_{\mu_{23}}]\ne1$ in $\outf(N_1)$ (since they do not commute 
after restriction to $Z(N_1)=\gen{a_1a_2,a_3,b_3}$), so 
$[c_{\mu_{23}}]\notin Z(\outf(N_1))$, and $N_1\in\EE_\calf$. 

By Step 3, $\nu_1(Q)=Q$, so $\nu_1$ induces the identity on $A\cap{}Q$ and 
(since $C_Q(A\cap Q)=A\cap Q$)
on $Q/(A\cap{}Q)\cong \klfour$.  Hence $\nu_1|_Q=\Id_Q$ by Lemma \ref{l:mod-Fr} 
and since $|\nu_1|=3$.  Also, $\eta(\tau)=\tau{}g$ for some 
$g\in{}A$, $g\cdot\alpha_{12}(g)\cdot\alpha_{12}^2(g)=1$ since 
$|\eta|=3$, $g\in{}C_A(\tau)$ since $|\eta(\tau)|=2$, and thus 
$g\in\gen{a_1a_2,b_1b_2}$.  So after replacing $\eta$ by 
$c_{a_1^ib_1^j}\circ\eta$ for appropriate $i,j$, we can assume that 
$\eta(\tau)=\tau$.  Thus $\autf(H_1)$, and hence $\calf$, is uniquely 
determined in this case.

\smallskip

\noindent\boldd{Step 6: \ Identifying the fusion system $\calf$. } We have 
now shown that $\calf$ is isomorphic to one of three fusion systems, 
$\calf_1$, $\calf_2$, or $\calf_3$, where $\Aut_{\calf_1}(A)\cong{}A_7$, 
$\Aut_{\calf_2}(A)\cong\GL_3(2)$, and 
$\Aut_{\calf_3}(A)\cong{}(C_3)^3\sd{}\Sigma_4$. 

When $G=\Omega_7(3)$, the group of all elements of $G$ which act up to sign 
on an orthonormal basis of $\F_3^7$ is isomorphic to $(C_2)^6\sd{}A_7$ and 
has odd index in $G$.  This follows from the formula for $|G|$ (cf. 
\cite[p. 166]{Taylor}). Since $S$ splits over $A$ and the two $A_7$-actions 
are isomorphic (see Step 5), $S$ is isomorphic to any $S_1\in\syl2{G}$. 
Since $\Out(S_1)$ is a $2$-group, $\calf_{S_1}(G)$ is reduced by 
Proposition \ref{p:FS(G)-red}(c). Hence $\Aut_G(A)\cong{}A_7$ (it can't be any 
bigger by Step 4), and so $\calf_{S_1}(G)\cong\calf_1$ in this case.

The group $G=\Sp_6(2)$ has a maximal parabolic subgroup $H\cong 
M_3^s(2)\sd{}\GL_3(2)$ (the stabilizer of a maximal isotropic subspace), 
where $M_3^s(2)\cong (C_2)^6$ is the group of symmetric $3\times3$ matrices 
over $\F_2$, and $A\in\GL_3(2)$ acts on it via $X\mapsto AXA^t$. By the 
order formulas for symplectic groups \cite[p. 70]{Taylor}, $H$ has odd index 
in $G$. Also, the $\GL_3(2)$-orbit of $\mxthree100000000$ has order $7$ and 
stabilizer subgroup isomorphic to $\Sigma_4$, and the other six elements 
form a basis for $M_3^s(2)$. So this is the same action as 
that described in Step 5, $S_2\cong S$ for $S_2\in\syl2{G}$, and 
$\calf_{S_2}(G)\cong\calf_2$ by Step 5 (and since $\calf_{S_2}(G)$ is 
reduced by Proposition \ref{p:FS(G)-red}(c)). 

Finally, when $G=A_{12}$ and $A<S$ are as defined here, then 
$S\in\syl2{G}$ by definition, $\calf_{S}(G)$ is reduced by Proposition 
\ref{p:FS(G)-red}(c), $\Aut_{\Sigma_{12}}(A)\cong\Sigma_3\wr\Sigma_3$ 
and hence $\Aut_G(A)\cong{}(C_3)^3\sd{}\Sigma_4$, and so 
$\calf_{S}(G)\cong\calf_3$.  

This finishes the proof that every reduced fusion system over $S$ is 
isomorphic to the $2$-fusion system of $\Omega_7(3)$, $\Sp_6(2)$, or 
$A_{12}$.  
\end{proof}

\bigskip

\appendix

\section{Background results on groups and representations}

We collect here some background results on groups and their representations 
which are needed elsewhere in the paper.  The first ones involve 
automorphisms and extensions of $p$-groups.

\begin{Lem} \label{l:mod-Fr}
Fix a prime $p$, a finite $p$-group $P$, a subgroup $P_0\le\Phi(P)$, and a
sequence of subgroups
        $$ P_0<P_1<\cdots< P_k=P $$
all normal in $P$. Set
        $$ \cala = \bigl\{\alpha\in\Aut(P)\,\big|\, 
        [\alpha,P_i]\le P_{i-1},
        \textup{ all $i=1,\dots,k$} \bigr\} \le \Aut(P)\,: $$
the group of automorphisms which leave each $P_i$ invariant, and which 
induce the identity on each quotient group $P_i/P_{i-1}$.  Then $\cala$ is 
a $p$-group.  If the $P_i$ are all characteristic in $P$, then 
$\cala\nsg\Aut(P)$, and hence $\cala\le{}O_p(\Aut(P))$.  
\end{Lem}

\begin{proof} If $\alpha\in\cala$ has order prime to $p$, then $\alpha$ 
induces the identity on $P/P_0$ and hence on $P/\Phi(P)$ by \cite[Theorem 
5.3.2]{Gorenstein}, and so $\alpha=\Id$ by \cite[Theorem 
5.1.4]{Gorenstein}. Thus $\cala$ is a $p$-group. If the $P_i$ are all 
characteristic, then $\cala$ is the kernel of a homomorphism from $\Aut(P)$ 
to $\prod_{i=1}^k\Aut(P_i/P_{i-1})$, and hence is normal in $\Aut(P)$.
\end{proof}

When $H\nsg G$, we let $\Aut(G,H)$ be the group of automorphisms of $G$ 
which normalize $H$, and set $\Out(G,H)=\Aut(G,H)/\Inn(G)$.

\begin{Lem} \label{l:res.aut.}
Fix a group $G$ and a normal subgroup $H\nsg{}G$ such that $C_G(H)\le{}H$
(i.e., $H$ is centric in $G$). Then there is an exact sequence
	\begin{multline} 
	1 \Right2{} H^1(G/H;Z(H)) \Right4{\eta} \Out(G,H) \Right4{R} \\
	N_{\Out(H)}(\Out_G(H))/\Out_G(H) \Right4{\chi} H^2(G/H;Z(H)),
	\label{eqaut3}
	\end{multline}
where $R$ is induced by restriction, and where all maps except (possibly) 
$\chi$ are homomorphisms.  If $H$ is abelian and the extension of $H$ by 
$G/H$ is split, then $R$ is onto.  If $Z(H)$ has exponent $p$ for some 
prime $p$, and has a basis which is permuted freely by $G/H$ under 
conjugation, then $R$ is an isomorphism.
\end{Lem}

\begin{proof} See \cite[Lemma 1.2 \& Corollary 1.3]{OV2}.
\end{proof}

\begin{Prop}[{\cite[Proposition 1.8]{OV2}}] \label{p:splittings}
Fix a prime $p$, a finite group $G$, and a normal abelian $p$-subgroup
$Q\nsg{}G$.  Let $H\le{}G$ be such that $Q\cap{}H=1$, and let $H_0\le{}H$
be of index prime to $p$.  Consider the set
        $$ \calh = \bigl\{H'\le{}G \,\big|\, H'\cap{}Q=1,\
        QH'=QH,\ H_0\le{}H' \bigr\}. $$
Then for each $H'\in\calh$, there is $g\in{}C_Q(H_0)$ such that
$H'=gHg^{-1}$.
\end{Prop}

We next note the following elementary properties of strongly $p$-embedded 
subgroups (Definition \ref{d:ess-crit}(a)). 

\begin{Lem} \label{str.emb.->>}
Let $H$ be a strongly $p$-embedded subgroup of a finite group $G$.
\begin{enuma} 
\item For each subgroup $H^*<G$ such that $H^*\ge H$, $H^*$ is also 
strongly $p$-embedded in $G$.
\item For each normal subgroup $K\nsg G$ of order prime to $p$ such that 
$HK<G$, $HK/K$ is strongly $p$-embedded in $G/K$.
\end{enuma}
\end{Lem}

\begin{proof} Point (a) follows from the definition as an easy exercise, 
and also follows immediately from the equivalence (when $k=1$) of points 
(1) and (3) in \cite[46.4]{A-FGT}. In the situation of (b), $HK$ is 
strongly $p$-embedded in $G$ by (a), and hence $HK/K$ is strongly 
$p$-embedded in $G/K$ by definition. 
\end{proof}

The next lemma involves subgroups of $\GL_5(2)$.

\begin{Lem} \label{l:G-V5}
If $G$ acts linearly, faithfully, and irreducibly on $\F_2^5$, then either 
$G$ has odd order, or $G\cong{}\GL_5(2)$.
\end{Lem}

\begin{proof}  See \cite[Theorem 1.1]{Wagner}. Wagner's theorem deals more 
generally with irreducible subgroups of $\PSL_5(2^a)$ for arbitrary 
$a\ge1$, but only cases (ii) and (v) in the theorem apply when $a=1$. 
\end{proof}

The next three lemmas involve representations over $\F_2$.

\begin{Lem} \label{GL6(2)}
Assume $G\le \GL_6(2)$ is such that $O_{2'}(G)\ne1$.  Then for some 
odd prime $p$, $G$ contains a normal elementary abelian $p$-subgroup 
$1\ne{}K\nsg G$ characteristic in $O_{2'}(G)$, and hence $G$ is 
contained up to conjugation in one of the following normalizers:
	\[ \renewcommand{\arraystretch}{1.5} 
	\begin{array}{|c||l|} \hline
	K\cong & N_{\GL_6(2)}(K)\cong \\\hline\hline
	C_3 & \Sigma_3\times\GL_4(2),~ 
	(\GL_2(4)\sd{}C_2)\times\Sigma_3,~\textup{or}~ \GL_3(4)\sd{}C_2 \\\hline
	(C_3)^2 & (\Sigma_3\wr C_2)\times\Sigma_3,~ 
	(\GL_2(4)\sd{}C_2)\times\Sigma_3,~ \textup{or}~
	(C_3)^3\sd{}(C_2\times\Sigma_3) \\\hline
	(C_3)^3 & \Sigma_3\wr\Sigma_3 \\\hline
	C_5 & (C_{15}\sd{}C_4)\times\Sigma_3 \\\hline
	C_7 & (C_7\sd{}C_3)\times\GL_3(2),~ (C_7)^2\sd{}C_6, 
	~\textup{or}~ \GL_2(8)\sd{}C_3 \\\hline
	(C_7)^2 & (C_7\sd{}C_3)\wr C_2 \\\hline
	C_{31} &  C_{31}\sd{}C_5 \\\hline
	\end{array} \]
When $K\cong (C_3)^3$ acts on $V\cong\F_2^6$, there is a decomposition 
$V=V_1\times V_2\times V_3$ such that each element of 
$N_{\GL_6(2)}(K)\cong\Sigma_3\wr\Sigma_3$ permutes the subgroups $V_i\cong\F_2^2$.
\end{Lem}

\begin{proof} Since $O_{2'}(G)\ne1$ is solvable by the odd order theorem 
\cite{FT}, it contains an elementary abelian $p$-subgroup $K$ which is 
characteristic (for some odd $p$).  Since 
$|\GL_6(2)|=2^{15}\cdot3^4\cdot5\cdot7^2\cdot31$, $K$ must be isomorphic to 
one of the subgroups in the above list.

The list of normalizers now follows from elementary representation 
theoretic considerations.  When $K\in\sylp{\GL_6(2)}$, then there is only 
one possibility for $K$ up to conjugacy, and its normalizer is as described 
above.  In all other cases, $N_{\GL_6(2)}(K)$ depends on the description of 
$\F_2^6$ as an $\F_2[K]$-module.

For example, when $K\cong (C_3)^2$, the module could be a 
sum of two irreducible modules of dimension $2$ and a $2$-dimensional module 
with trivial action, or three irreducibles of dimension $2$ of which two 
are isomorphic, or three pairwise nonisomorphic irreducible modules of 
dimension $2$.  The three possible normalizers listed above correspond to 
these three choices.
\end{proof}

\begin{Lem} \label{l:HxA5}
Let $P$ be a finite $2$-group such that $\rk(P/\Phi(P))\le6$.  Assume 
$G\le\Out(P)$ has a strongly $2$-embedded subgroup, where $S\in\syl2{G}$ and 
$S\cong \klfour$.  Then $G\cong A_5\times (C_3)^s$ for some $s\le2$.
\end{Lem}

\begin{proof} Set $H=O_{2'}(G)$. Since $G$ has a strongly $2$-embedded subgroup, 
$O^{2'}(G/H)\cong A_5$ by Bender's theorem \cite[Satz 1]{Bender}.  
Since $|\Out(A_5)|=2$ (see \cite[(3.2.17)]{Sz1}), $G/H\cong A_5$.  
In particular, $O_2(G)=1$, and by Lemma \ref{l:mod-Fr}, $G$ acts faithfully 
on $P/\Phi(P)\cong (C_2)^r$ ($r\le6$). We can thus identify $G$ as a 
subgroup of $\GL_6(2)$.

If $H=O_{2'}(G)\ne1$, then let $K\nsg G\le\GL_6(2)$ be as in Lemma 
\ref{GL6(2)}. Then $N_{\GL_6(2)}(K)$ is nonsolvable and 
$5\,\big|\,|N_{\GL_6(2)}(K)/K|$, and so by the same lemma, $K\cong C_3$ or 
$(C_3)^2$. From the list of possibilities for 
$N_{\GL_6(2)}((C_3)^s)$ for $s=1,2$, we see that $G\cong H\times(G/H)\cong 
H\times A_5$, where 
$H=K\cong C_3$ or $(C_3)^2$.  (Note that $\GL_2(4)\cong C_3\times A_5$.) 
\end{proof}

There are four distinct irreducible $\4\F_2[A_5]$-modules:  the trivial 
module, the natural 2-dimensional $\4\F_2[\SL_2(4)]$-module and its Galois 
conjugate, and the natural 4-dimensional module for $A_5$ 
(cf. \cite[\S\,18.6]{Serre}). The first and last are realizable over 
$\F_2$, while the 2-dimensional modules are realizable over $\F_4$ and 
hence induce a $4$-dimensional irreducible $\F_2[A_5]$-module by 
restriction of scalars.  Thus there are three irreducible 
$\F_2[A_5]$-modules:  the trivial module $\F_2$, and two $4$-dimensional 
modules $W_1$ and $W_2$ described as follows: 
\begin{itemize}  
\item $W_1$ is generated by a $G$-orbit of five elements whose sum is zero; 
and
\item $W_2\cong\F_4^2$ with the canonical action of $G\cong\SL_2(4)$.
\end{itemize}
We keep this notation for the irreducible modules in the statement of the 
following lemma.

\begin{Lem} \label{l:F2[A5]-mod}
Fix $G\cong{}A_5$.  Let $V$ be a finitely generated $\F_2[G]$-module such 
that for each involution $x\in{}G$, $\dim([x,V])=2$.  Then the composition 
factors of $V$ include exactly one irreducible $4$-dimensional module.  
Also, the following hold for any $S\in\syl2{G}$. 
\begin{enuma}  

\item If $[S,V]\nleq{}C_V(S)$, then $V\cong\F_2^{k}\oplus{}W_1$ 
($k=\dim(V)-4$).

\item Assume $[S,V]\le{}C_V(S)$. Then there is a composition factor of $V$ 
isomorphic to $W_2$, and $V$ is indecomposable if and only if 
$[S,V]=C_V(S)$. Also, $C_V(G)=C_V(N_G(S))$ and 
$\dim(C_V(G))=\dim(C_V(S))-2$. If $V$ is indecomposable and 
$C_V(G)\ne0$, then $\bigcap_{1\ne{}x\in{}S}[x,V]=0$. 

\iffalse
\item  If $[S,V]\le{}C_V(S)$, then there is a composition factor 
of $V$ isomorphic to $W_2$. Also, $C_V(G)=C_V(N_G(S))$ and 
$\dim(C_V(G))=\dim(C_V(S))-2$, and $V$ is indecomposable if and only if 
$[S,V]=C_V(S)$. If $V$ is indecomposable and $C_V(G)\ne0$, then 
$\bigcap_{1\ne{}x\in{}S}[x,V]=0$. 
\fi

\end{enuma}
\end{Lem}

\begin{proof}  As noted above, each irreducible $\F_2[G]$-module is 
isomorphic to $\F_2$, $W_1$, or $W_2$.  If each composition factor of $V$ 
is 1-dimensional, then for an appropriate choice of basis, $G$ would act on 
$V$ via upper triangular matrices, and thus via the identity (Lemma 
\ref{l:mod-Fr}).  

By a direct check, $\dim([x,W_j])=2$ for $1\ne{}x\in{}S$ and $j=1,2$.  If 
the composition factors of $V$ include more than one $4$-dimensional 
module, then this would mean $\dim([x,V])\ge4$.  Since $\dim([x,V])=2$ for 
$1\ne{}x\in{}S$, there is exactly one such composition factor.

\smallskip

\noindent\textbf{(a) } Let $V_1\le{}V_2\le{}V$ be submodules such that 
$V_2/V_1$ is $4$-dimensional and irreducible. We just saw that $G$ acts 
trivially on $V_1$ and on $V/V_2$.  In particular, we can assume they were 
chosen so that $V_1=C_V(G)$.  Also, $\dim([x,V_2])=2$ for $1\ne{}x\in{}S$, 
and thus $[x,V_2]=[x,V]$ and $[S,V_2]=[S,V]$.  Likewise, 
$\dim([x,V_2])=\dim([x,V_2/V_1])$ for $1\ne{}x\in{}S$ implies 
$V_1\cap[x,V_2]=0$, and hence $C_{V_2}(x)$ surjects onto $C_{V_2/V_1}(x)$.  
Since $G$ acts trivially on $V_1$, this shows that 
$C_{V_2/V_1}(S)=C_{V_2}(S)/V_1$.

Upon combining these observations, we see that 
	\begin{align*} 
	[S,V] \le C_V(S)\quad 
	&\iff\quad [S,V_2]=[S,V]\le C_V(S)\cap V_2 = C_{V_2}(S) \\
	&\iff\quad [S,V_2]+V_1\le C_{V_2}(S) \\
	&\iff\quad [S,V_2/V_1] \le C_{V_2}(S)/V_1 = C_{V_2/V_1}(S)\,.
	\end{align*}
Since $[S,W_2]=C_{W_2}(S)$ while $[S,W_1]\nleq{}C_{W_1}(S)$, we conclude 
that $[S,V]\le{}C_V(S)$ if and only if $V_2/V_1\cong{}W_2$. 

Since $W_1$ is free (hence projective) as an $\F_2[S]$-module, it is also 
projective as an $\F_2[G]$-module (see \cite[Corollary 
3.6.10]{Benson1}). Hence $W_1$ is also injective by, e.g., 
\cite[Proposition 3.1.2]{Benson1}. So if $V_2/V_1\cong W_1$, then $V\cong 
W_1\oplus\F_2^k$ where $k=\dim(V)-4$.

\smallskip

\noindent\textbf{(b) } Now assume that $[S,V]\le{}C_V(S)$, and thus as 
shown above that $V_2/V_1\cong{}W_2$.  If $V$ is decomposable, then it 
contains a nonzero direct summand with trivial action, in which case 
$[S,V]<{}C_V(S)$.  In other words, $V$ is indecomposable if $[S,V]=C_V(S)$.  

Conversely, assume $[S,V]<C_V(S)$.  Fix a set $X\subseteq G$ of 
representatives for the left cosets $gS\subseteq G$.  For each $v\in 
C_V(S)$, set $v^*=\sum_{g\in X}gv$.  Then $v^*\in C_V(G)=V_1$ and 
$v^*\equiv |G/S|\cdot v$ (mod $[G,V]$), so $v\in V_1+[G,V]\le V_2$.  Hence 
$C_V(S)\le V_2$, and $C_V(S)/V_1=C_{V_2/V_1}(S)=[S,V_2/V_1]\le([S,V]+V_1)/V_1$ 
by the above remarks. Thus $[S,V]\ngeq V_1$ since $[S,V]<C_V(S)$.  So there 
are $v\in{}V_1$ and $\varphi\in\Hom_S(V,\F_2)$ such that $\varphi(v)\ne0$.  
Define $\psi\in\Hom_G(V,\F_2)$ by $\psi(x)=\sum_{g\in X}\varphi(g^{-1}x)$; then 
$\psi(v)=|G/S|\varphi(v)\ne0$.  Thus the inclusion $\gen{v}<V$ is split by 
the map $(x\mapsto \psi(x)v)$, and $V$ is decomposable.

It remains to prove the claims in the last two sentences in (b). Since $V$ 
is the direct sum of an indecomposable module and a module with trivial 
action, it suffices to show them when $V$ is indecomposable, and thus when 
$[S,V]=C_V(S)$.  Recall that $V_1=C_V(G)$.  Also, $C_V(S)=[S,V]\le[G,V]\le 
V_2$, and thus $C_V(S)/V_1=C_{V_2}(S)/V_1=C_{V_2/V_1}(S)\cong C_{W_2}(S)$ 
is $2$-dimensional. So $\dim(C_V(G))=\dim(C_V(S))-2$, and 
$C_V(N_G(S))=C_V(G)$ since $C_{W_2}(N_G(S))=0$. 

Recall that $\dim([x,V])=2$ and $[x,V]\cap{}C_V(G)=0$ for $1\ne{}x\in{}S$. 
Set $V_0=\bigcap_{1\ne{}x\in{}S}[x,V]$. Either $V_0=0$; or $\dim(V_0)=1$ 
and $V_0\le C_V(N_G(S))=C_V(G)$, which is impossible; or $V_0$ is 
$2$-dimensional, hence is equal to $[S,V]=C_V(S)$, and so $C_V(G)=0$. 
\end{proof}

In fact, up to isomorphism, there are two distinct indecomposable 
$\F_2[A_5]$-modules of dimension $5$ and three of dimension $6$ which 
satisfy the hypotheses of Lemma \ref{l:F2[A5]-mod}(b).  The three of 
dimension $6$ are the permutation module for the $A_5$-action on 
$A_5/D_{10}$, the group of symmetric $(2\times2)$ matrices over 
$\F_4$ with the canonical action of $A_5\cong\SL_2(4)$, and the dual of 
this last module.  

\bigskip

%%%%%%%%%%%%%%%%%%%%%%%%%%%%%%%%%%%%%%%%%%%%%%%%

\end{document}